\documentclass[10pt, twoside]{amsart}
\usepackage{amssymb, latexsym, amsmath}
\usepackage{graphicx,latexsym, epsfig}
\usepackage[all]{xy}
\usepackage{hyperref}

\numberwithin{equation}{section}

\theoremstyle{plain}
\newtheorem{theo}{Theorem}[section]

\newtheorem{lemm}[theo]{Lemma}
\newtheorem{prop}[theo]{Proposition}

\theoremstyle{definition}

\newtheorem{exam}[theo]{Example}

\theoremstyle{remark}
\newtheorem{rema}[theo]{Remark}

\newcommand{\ZZZ}{\mathbb{Z}} 
\newcommand{\ZP}{ \mathbb{\ZZZ}_{p}} 
\newcommand{\QQQ}{\mathbb{Q}} 
\newcommand{\QP}{\mathbb{\QQQ}_{p}} 
\newcommand{\qpb}{\overline{\QQQ}_{p}} 

\newcommand{\GL}[1]{\mathrm{GL}_{#1}} 

\newcommand{\Dim}[1]{\mathrm{dim}_{#1}} 

\newcommand{\baseo}{\emph{e}_{1}}
\newcommand{\baset}{\emph{e}_{2}}
\newcommand{\baseth}{\emph{e}_{3}}
\newcommand{\hodgen}{\mathrm{t}_{H}}
\newcommand{\newtonn}{\mathrm{t}_{N}}
\newcommand{\val}[1]{v_{p}(#1)}
\newcommand{\vphi}{\varphi}
\newcommand{\lambdao}{\lambda_{1}}
\newcommand{\lambdat}{\lambda_{2}}
\newcommand{\lambdath}{\lambda_{3}}

\newcommand{\filr}{\mathrm{Fil}^{r}}
\newcommand{\fils}{\mathrm{Fil}^{s}}
\newcommand{\fil}[1]{\mathrm{Fil}^{#1}}

\newcommand{\Bst}{\mathrm{B_{st}}}
\newcommand{\Bdr}{\mathrm{B_{dR}}}
\newcommand{\Dst}{\mathrm{D_{st}}}

\newcommand{\mfl}{\mathfrak{L}}
\newcommand{\mflo}{\mathfrak{L}_{1}}
\newcommand{\mflt}{\mathfrak{L}_{2}}
\newcommand{\mflth}{\mathfrak{L}_{3}}

\DeclareMathOperator{\rank}{rank}

\hypersetup{
    colorlinks,%
    citecolor=blue,%
    filecolor=black,%
    linkcolor=blue,%
    urlcolor=black
}

\begin{document}

\title{Regular filtered $(\phi,N)$-modules of dimension $3$}

\author{Chol Park}

\address{Department of Mathematics \\
		The University of Arizona}
		
\email{cpark@math.arizona.edu}

\begin{abstract}
We classify $3$-dimensional semi-stable representations of $G_{\QP}$ with coefficients and regular Hodge--Tate weights, by determining the isomorphism classes of admissible filtered $(\phi,N)$-modules of Hodge type $(0,r,s)$ with $0<r<s$.
\end{abstract}

\maketitle

\section{Introduction}
Let $p$ be a prime number, and write $G_{\QP}$ for the Galois group $\mathrm{Gal}(\qpb/\QP)$.  In this paper, we classify $3$-dimensional semi-stable representations of $G_{\QP}$ with coefficients and with regular Hodge--Tate
weights. By a theorem of Colmez and Fontaine \cite{CF}, this is equivalent to determining the isomorphism classes of admissible filtered $(\phi,N)$-modules with Hodge type $(0,r,s)$ with $0 < r < s$, and that is what we do.

Let us explain our motivation for doing this.  This work is the first part of the author's Ph.D. thesis in which we construct deformation spaces whose characteristic $0$ closed points are the semi-stable lifts with Hodge--Tate weights $(0,1,2)$ of a fixed irreducible representation $\overline{\rho} : G_{\QP} \rightarrow \mathrm{GL}_3(\overline{\mathbb{F}}_p)$.  The existence of these deformation spaces was proved by Kisin \cite{Kisin}, and their geometry plays an essential role in the Taylor--Wiles--Kisin method \cite{Wiles, TW, KisinM} for proving the modularity of Galois representations.  In particular, the special fibers of these deformation spaces are described  by a conjecture of Breuil--M\'ezard \cite{BM} as well as a refinement of this conjecture due to Emerton--Gee \cite{EG}.  The $\mathrm{GL}_2(\QP)$ case of the Breuil--M\'ezard conjecture is a theorem of Kisin \cite{KisinF}, and implies the Fontaine--Mazur conjecture for $\mathrm{GL}_2(\mathbb{Q})$.

Our goal is to address certain special cases of the Breuil--M\'ezard conjecture for $\mathrm{GL}_{3}(\QP)$ --- namely,
the semi-stable case with Hodge--Tate weights $(0,1,2)$, with irreducible $\overline{\rho}$ --- following the method of \cite{Savitt}.  The plan, roughly speaking, is to classify lattices in semi-stable representations $\rho : G_{\QP}
\rightarrow \mathrm{GL}_3(\qpb)$ with Hodge--Tate weights $(0,1,2)$, by classifying the corresponding strongly divisible modules \cite{Breuil}  (these are certain integral structures closely related to filtered $(\phi,N)$-modules).  The first step is to classify the semi-stable representations $\rho$ with these Hodge--Tate weights, and that is what is done in this paper.  In fact, since it will add relatively little extra work, we consider the case of distinct Hodge--Tate weights $0 < r < s$, rather than simply the case $(0,1,2)$.

We find 49 families of admissible filtered $(\phi,N)$-modules of dimension $3$ for general $r$ and $s$ with $0<r<s$. Among them, there are $26$ families with $N=0$ (i.e., the crystalline case; see Subsection \ref{ssec:list of N=0}), there are $20$ families with $\rank N=1$ (see Subsection \ref{ssec:list of N=1}), and there are $3$
families with $\rank N=2$ (see Subsection \ref{ssec:list of N=2}).  This is in contrast to the $\GL{2}$ setting, where there are only three families with $N=0$ and one with $\rank N = 1$.  We also determine which of these admissible filtered $(\phi,N)$-modules correspond to irreducible representations; in fact, we determine all submodules of these admissible filtered $(\phi,N)$-modules.

We note the following mild hypothesis.  Let $E/\QP$ be a finite extension.  What we actually do is classify the admissible filtered $(\phi,N)$-modules with $E$-coefficients, corresponding to representations $\rho : G_{\QP} \rightarrow \mathrm{GL}_3(E)$, under the hypothesis that the Jordan form of the Frobenius map $\phi$ is defined over~$E$. Since the correspondence between Galois representations and filtered $(\phi,N)$-modules is compatible with extension of coefficients, this assumption is harmless.

Dousmanis \cite{Dous2} has independently treated the case of $3$-dimensional Frobenius-semisimple semi-stable representations of $G_K$ with $K/\QP$ unramified. Related problems in the case of $2$-dimensional representations have been treated in several articles.  Savitt \cite{Savitt} classifies the potentially crystalline $2$-dimensional representations of $G_{\QP}$ with tamely ramified Galois type.  This is extended to all potentially semi-stable representations of $G_{\QP}$ by Ghate and M\'ezard
\cite{GM} at least in the case when $p$ is odd, and to potentially semi-stable representations of $G_{K}$ for $K/\QP$ finite by Dousmanis \cite{Dous}.

This paper is organized as follows.  In the remainder of the introduction we give a brief review of $p$-adic Hodge theory, and introduce notation that will be used throughout the paper.  In Section $2$, we collect the admissible filtered $\phi$-modules for each Jordan form of $\phi$, and list the isomorphism classes with $N=0$ in
Subsection \ref{ssec:list of N=0}.  In Section $3$, we first find the possible types of $\phi$ satisfying $N\phi=p\phi N$ under the assumption that $N$ has rank $1$.  We then collect the admissible filtered $(\phi,N)$-modules for each type of $\phi$, and list the isomorphism classes with $\rank N=1$ in Subsection \ref{ssec:list of N=1}. The case $\rank N=2$ is treated in Section $4$, following the same method as in Section $3$.  There is only one type of $\phi$ satisfying $N\phi=p\phi N$ in this case and we list the isomorphism classes with $\rank N=2$ in Subsection \ref{ssec:list of N=2}.

\subsection{Review of filtered $(\phi,N)$-modules}
Let $K$ and $E$ be finite extensions of $\QP$ inside $\qpb$ and $K_{0}$ the maximal absolutely unramified subextension of $K$. We write $\sigma$ for the absolute Frobenius element on $K_{0}$, and $G_{K} = \mathrm{Gal}(\qpb/K)$ for the absolute Galois group of $K$. Fix a uniformizer $\pi_{K}$ for $K$, thereby fixing the inclusion
$$K\otimes_{K_{0}}\Bst\hookrightarrow \Bdr,$$
where $\Bst,\Bdr$ are rings of $p$-adic periods defined in \cite{Fontaine}.  Let $v_{p}$ be a valuation on $\qpb$ with $\val{p}=1$.

A \emph{filtered $(\phi,N)$-module} (strictly speaking, a filtered $(\phi,N,K,E)$-module) of rank $d$ is a free $(K_{0}\otimes_{\QP}E)$-module $D$ of rank $d$ together with a triple $(\phi,N,\{\fil{i}D_{K}\}_{i\in\ZZZ})$ where
\begin{itemize}
\item the \emph{Frobenius map} $\phi$ is a $\sigma$-semilinear and $E$-linear automorphism,
\item the \emph{monodromy operator} $N$ is a (nilpotent) $K_{0}\otimes_{\QP}E$-linear endomorphism such that $N\phi=p\phi N$, and
\item the \emph{Hodge filtration} $\{\fil{i}D_{K}\}_{i\in\ZZZ}$ is a decreasing filtration on $D_{K}:=K\otimes_{K_{0}}D$ such that a $(K\otimes_{\QP}E)$-module $\fil{i}D_{K}$ is
$D_{K}$ if $i\ll 0$ and $0$ if $i\gg 0$.
\end{itemize}
A \emph{filtered $\phi$-module} is a filtered $(\phi,N)$-module with trivial monodromy operator $N$. A \emph{morphism of filtered $(\phi,N)$-modules} $$\eta:D=(\phi,N,\{\fil{i}D_{K}\}_{i\in\ZZZ})\longrightarrow D'=(\phi',N',\{\fil{i}D'_{K}\}_{i\in\ZZZ})$$ is a $(K_{0}\otimes_{\QP} E)$-module homomorphism such that
\begin{itemize}
\item $\phi'\circ\eta=\eta\circ\phi$,
\item $N'\circ\eta=\eta\circ N$, and
\item the induced map $\eta_{K}:D_{K}\rightarrow D'_{K}$ satisfies $\eta_{K}(\fil{i}D_{K})\subset\eta_{K}(\fil{i}D'_{K})$ for each $i\in\ZZZ$.
\end{itemize}
The Hodge--Tate weights of a filtered $(\phi,N)$-module $D$ are the integers $r$ such that $\fil{r}D_K \neq \fil{r+1}D_K$, each counted with multiplicity $$\Dim{E}(\fil{r}D_{K}/\fil{r+1}D_{K}).$$  If the rank of $D$ over $K_0 \otimes_{\QP} E$ is $d$, then there are precisely $d \cdot [K:\QP]$ Hodge--Tate weights, with multiplicity.  If $K = \QP$ and the Hodge--Tate weights are $r_1 \le \cdots \le r_d$, we say that the filtered $(\phi,N)$-module $D$ is of \emph{Hodge type $(r_{1},...,r_{d})$}.  When the Hodge--Tate weights of $D$ are distinct, we
say that $D$ is \emph{regular} (or that it has regular Hodge--Tate weights).

If $D$ is a filtered $(\phi,N)$-module of dimension $n$ as a $K_{0}$-vector space, then we give $\otimes^{n}_{K_{0}}D$ the structure of a filtered $(\phi,N)$-module by setting
\begin{itemize}
\item $\phi=\otimes^{n}\phi$,
\item $N=N\otimes1\otimes\cdot\cdot\cdot\otimes1+1\otimes N\otimes
  1\otimes\cdot\cdot\cdot\otimes 1+ \cdots +1\otimes\cdot\cdot\cdot\otimes1\otimes N$, and
\item
  $\fil{i}(K\otimes_{K_{0}}(\otimes^{n}_{K_{0}}D))=\underset{i_{1}+i_{2}+\cdots
    + i_{n}=i}\sum \fil{i_{1}}D_{K}\otimes_{K}...\otimes_{K}\fil{i_{n}}D_{K}$.
\end{itemize}
Taking the image structure on $\wedge_{K_{0}}^{n}D$ makes $\wedge_{K_{0}}^{n}D$ into a filtered $(\phi,N)$-module as well. Since $\Dim{K_{0}} \wedge_{K_{0}}^{n}D=1$, we define $$\hodgen(D)=\frac{\mathrm{max}\{i\in\ZZZ\,|\,\fil{i}(K\otimes_{K_{0}}(\wedge_{K_{0}}^{n}D))\not=0\}}{[E:\QP]}\mbox{
  and  }\newtonn(D)=\frac{\val{\phi(x)/x}}{[E:\QP]}$$ for a nonzero element $x$ in $\wedge_{K_{0}}^{n}D$.  (Our definitions of the Hodge invariant $\hodgen$ and Newton invariant $\newtonn$ are normalized differently from those in \cite{CF} because we divide by $[E:\QP]$, but of course will still give the same notion of admissibility below.)

A $(K_{0}\otimes_{\QP}E)$-submodule $D'$ of a filtered $(\phi,N)$-module $D$ is a \emph{filtered $(\phi,N)$-submodule} if it is $\phi$-invariant and $N$-invariant, in which case $D'$ has a Frobenius map $\phi|_{D'}$, a monodromy operator $N|_{D'}$, and the filtration $\fil{i}D'_{K}=\fil{i}D_{K}\cap D'_{K}$.  A filtered $(\phi,N)$-module $D$ is said to be \emph{admissible} if $\hodgen(D)=\newtonn(D)$, and if $\hodgen(D')\leq\newtonn(D')$ for each filtered $(\phi,N)$-submodule $D'$ of $D$.

Let $V$ be a finite-dimensional $E$-vector space equipped with continuous action of $G_{K}$, and define $$\Dst(V):=(\Bst\otimes_{\QP}V)^{G_{K}}.$$  Then $\rank_{K_{0}\otimes E}\Dst(V)\leq\Dim{E}V$.  If equality holds,
then we say that $V$ is \emph{semi-stable}; in that case $\Dst(V)$ inherits from $\Bst$ the structure of an admissible filtered $(\phi,N)$-module.  (See \cite{Fontaine} for details.)   We say that $V$ is \emph{crystalline} if $V$ is semi-stable and the monodromy operator $N$ on $\Dst(V)$ is~$0$.

If $V$ is semi-stable, then when we refer to the Hodge--Tate weights or the
Hodge type of $V$, we mean those of $\Dst(V)$.  Our normalizations imply that the cyclotomic character $\varepsilon : G_{\QP} \to E^{\times}$ has Hodge--Tate weight $-1$.  Twisting $V$ by a power $\varepsilon^n$ of the cyclotomic character has the effect of shifting all the Hodge--Tate weights of $V$ by $-n$; after a suitable twist, we are therefore free to assume that the lowest weight of $V$ is $0$.

\subsection{Notation and Terminology}
If $D$ is an admissible filtered $(\phi,N)$-module with Hodge--Tate weights $0 < r < s$ corresponding to a representation of $G_{\QP}$, we will write
$$
\mathrm{Fil}^{i}D=
\left\{
  \begin{array}{ll}
    D, & \hbox{if $i\leq 0$;} \\
    L_{2}, & \hbox{if $0 < i\leq r$;}\\
    L_{1}, & \hbox{if $r < i\leq s$;} \\
    0, & \hbox{if $s < i$,}
  \end{array}
\right.
$$
where $D=E(\baseo,\baset,\baseth)$ is an $E$-vector space with basis $\baseo,\baset,\baseth$ and $L_{j}$ is a subspace of $D$ of dimension $j$ for $j=1,2$.  We assume that $E$ is large enough so that the Jordan form of $\phi$ is well-defined over $E$. We let $[T]$ be the matrix presentation of an endomorphism $T$ on $D$ with respect to
$\baseo,\baset,\baseth$ and $\mathbb{P}^{1}(E)$ the $E$-rational points in the projective line. In this paper, we say that a representation is \emph{non-split reducible} if it is reducible but indecomposable, and that a representation is \emph{irreducible} means that it is absolutely irreducible.

\subsection{Acknowledgments}
The author thanks his advisor, David Savitt, for his encouragement, guidance, and numerous helpful comments and suggestions. The author also thanks Matthew Emerton for his helpful comments and suggestions.

This work was completed while the author was a visiting student at Northwestern University during the 2011-12 academic
year, and the author is grateful to the mathematics department at Northwestern for its hospitality.

\section{Admissible filtered $\phi$-modules}
In this section, we classify the admissible filtered $\phi$-modules of Hodge type $(0,r,s)$ for $0<r<s$. At the first six subsections, we collect such modules for each Jordan form of $\phi$, and, at the last subsection, we classify them and list the isomorphism classes of admissible filtered $\phi$-modules.

\subsection{The first case of $N=0$}
Assume that $\phi$ has a minimal polynomial of the form $(x-\lambda)$. So $\phi=\lambda I$, where $\lambda\in E^{\times}$ and $I$ is the $3\times 3$-identity matrix. By admissibility, $\val{\lambda}=\frac{r+s}{3}$. It is clear that every subspace in $D$ is $\phi$-invariant. In particular, $L_{1}$ is $\phi$-invariant. So, by admissibility, $s=\hodgen(L_{1})\leq\newtonn(L_{1})=\val\lambda$. But $s>\frac{r+s}{3}$. Hence, there are no admissible filtered $\phi$-modules in this case.

\subsection{The second case of $N=0$}
Assume that $\phi$ has a minimal polynomial of the form $(x-\lambda)^{2}$. So we may assume that $\phi\baseo=\lambda\baseo+\baset$, $\phi\baset=\lambda\baset$, and $\phi\baseth=\lambda\baseth$. By admissibility, we have $$\val{\lambda}=\frac{r+s}{3}.$$

\begin{lemm} \label{typeofPsecond}
For a $3\times3$-matrix $P=(P_{i,j})$, $P[\phi]=[\phi] P$ if and only if $P$ is a matrix with $P_{1,1}=P_{2,2}$ and $P_{1,2}=P_{1,3}=P_{3,2}=0$.
\end{lemm}

\begin{proof}
Let
$
P=\tiny\left(
  \begin{array}{ccc}
    a & b & c \\
    d & e & f \\
    g & h & i \\
  \end{array}
\right)
$.
We compute the matrix products:
$$P[\phi]=
\begin{tiny}
\left(
  \begin{array}{ccc}
    \lambda a + b & \lambda b & \lambda c \\
    \lambda d + e & \lambda e & \lambda f \\
    \lambda g + h & \lambda h & \lambda i \\
  \end{array}
\right)
\end{tiny}\mbox{ and }
[\phi] P=
\begin{tiny}
\left(
  \begin{array}{ccc}
   \lambda a &\lambda b &\lambda c \\
    a + \lambda d & b + \lambda e & c + \lambda f \\
    \lambda g & \lambda h & \lambda i \\
  \end{array}
\right)
\end{tiny}.
$$
By comparing the entries we get the result.
\end{proof}

\begin{lemm} \label{crysinvariantsecond}
\begin{enumerate}
\item Every $1$-dimensional subspace in $E(\baset,\baseth)$ is the only $\phi$-invariant subspace of dimension $1$.
\item A two dimensional subspace of $E(\baseo,\baset,\baseth)$ is $\phi$-invariant if and only if it contains $\baset$.
\end{enumerate}
\end{lemm}

\begin{proof}
The case (1) is clear since the $\phi$-invariant subspaces of dimension 1 are the eigenspaces. Let $D_{2}$ be a $\phi$-invariant subspace of dimension 2. If $\baset\not\in D_{2}$, then there exists an element $\baseo+b\baset+c\baseth$ in $D_{2}$. Then $\phi(\baseo+b\baset+c\baseth)=\lambda(\baseo+b\baset+c\baseth)+\baset$, and so $\baset\in D_{2}$, which is a contradiction. The converse is clear.
\end{proof}

We start to collect the admissible filtered $\phi$-modules in this case.
\subsubsection{}\label{}
Assume first that $L_{1}$ is $\phi$-invariant. Then, by admissibility, $s=\hodgen(L_{1})\leq\newtonn(L_{1})=\val\lambda$, which contradicts to $\val\lambda=\frac{r+s}{3}$.

\subsubsection{}\label{}
If $L_{2}$ is $\phi$-invariant, then, by admissibility, $r+s=\hodgen(L_{2})\leq\newtonn(L_{2})=2\val\lambda$, which contradicts to $\val\lambda=\frac{r+s}{3}$.

\subsubsection{}\label{cryssecond1}
Finally, assume that neither $L_{1}$ nor $L_{2}$ are $\phi$-invariant. We let $D_{1}$, $D_{2}$ be $\phi$-invariant subspaces of dimension 1 and of dimension 2, respectively. Then, by admissibility, for $D_{1}$ with $D_{1}\subset L_{2}$ $r=\hodgen(D_{1})\leq\newtonn(D_{1})=\val\lambda(s\geq 2r)$, for $D_{1}$ with $D_{1}\nsubseteq L_{2}$ $0=\hodgen(D_{1})\leq\newtonn(D_{1})=\val\lambda$, for $D_{2}$ with $D_{2}\cap L_{2}=L_{1}$ $s=\hodgen(D_{2})\leq\newtonn(D_{2})=2\val\lambda(s\leq 2r)$, and for $D_{2}$ with $L_{1}\nsubseteq D_{2}$ $r=\hodgen(D_{2})\leq\newtonn(D_{2})=2\val\lambda$. So admissible filtered $\phi$-modules occur in this case if and only if $s=2r$, and the corresponding representations are decomposable with submodules $L_{2}\cap E(\baset,\baseth)$ and $D_{2}$ with $L_{1}\subset D_{2}$.

\subsection{The third case of $N=0$}
Assume that $\phi$ has a minimal polynomial of the form $(x-\lambda)^{3}$. So we may assume that $\phi\baseo=\lambda\baseo+\baset$, $\phi\baset=\lambda\baset+\baseth$, and $\phi\baseth=\lambda\baseth$. By admissibility, we have $$\val{\lambda}=\frac{r+s}{3}.$$

\begin{lemm} \label{typeofPthird}
For a $3\times3$-matrix $P=(P_{i,j})$, $P[\phi]=[\phi] P$ if and only if $P$ is a lower triangle matrix with $P_{1,1}=P_{2,2}=P_{3,3}$ and $P_{2,1}=P_{3,2}$.
\end{lemm}

\begin{proof}
Let $P$ be as in the proof of Lemma \ref{typeofPsecond}. We compute the matrix products:
$$P[\phi]=
\begin{tiny}
\left(
  \begin{array}{ccc}
    \lambda a + b & \lambda b + c & \lambda c \\
    \lambda d + e & \lambda e + f & \lambda f \\
    \lambda g + h & \lambda h + i & \lambda i \\
  \end{array}
\right)
\end{tiny}\mbox{ and }
[\phi] P=
\begin{tiny}
\left(
  \begin{array}{ccc}
   \lambda a &\lambda b &\lambda c \\
    a + \lambda d & b + \lambda e & c + \lambda f \\
    d + \lambda g & e + \lambda h & f + \lambda i \\
  \end{array}
\right)
\end{tiny}.
$$
By comparing the entries we get the result.
\end{proof}

\begin{lemm} \label{crysinvariantthird}
$E\baseth$ and $E(\baset,\baseth)$ are the only nontrivial and proper $\phi$-invariant subspaces.
\end{lemm}

\begin{proof}
The vector subspaces listed above are obviously $\phi$-invariant. Obviously $E\baseth$ is the only $\phi$-invariant subspace of dimension 1. Let $D_{2}$ be a $\phi$-invariant subspace of dimension 2 with $D_{2}\neq E(\baset,\baseth)$. Then there exists an element $\baseo+b\baset+c\baseth$ in $D$, and $\phi(\baseo+b\baset+c\baseth)=\lambda(\baseo+b\baset+c\baseth)+\baset+b\baseth$. So $\baset+b\baseth\in D_{2}$, and $\phi(\baset+b\baseth)=\lambda(\baset+b\baseth)+\baseth$. Hence, $\baseth\in D_{2}$, i.e., $D_{2}$ has three linearly independent vectors, which is contradiction.
\end{proof}

We start to collect the admissible filtered $\phi$-modules in this case.
\subsubsection{}\label{}
Assume first that $L_{1}=E\baseth$. Then, by admissibility, $s=\hodgen(E\baseth)\leq\newtonn(E\baseth)=\val\lambda$, which contradicts to $\val\lambda=\frac{r+s}{3}$.

\subsubsection{}\label{}
If $L_{2}=E(\baset,\baseth)$, then, by admissibility, $r+s=\hodgen(E(\baset,\baseth))\leq\newtonn(E(\baset,\baseth))=2\val\lambda$, which contradicts to $\val\lambda=\frac{r+s}{3}$.

\subsubsection{}\label{crysthird1}
Assume that neither $L_{1}$ nor $L_{2}$ are $\phi$-invariant and $\baseth\in L_{2}$. Then $L_{2}\cap E(\baset,\baseth)=E\baseth$ and $L_{1}\nsubseteq E(\baset,\baseth)$. By admissibility, $r=\hodgen(E\baseth)\leq\newtonn(E\baseth)=\val\lambda\hspace{0.1cm}(s\geq 2r)$ and $r=\hodgen(E(\baset,\baseth))\leq\newtonn(E(\baset,\baseth))=2\val\lambda$. So admissible filtered $\phi$-modules occur in this case if and only if $s\geq 2r$. The corresponding representations are non-split reducible with submodule $E\baseth$ if $s=2r$ and irreducible if $s>2r$.

\subsubsection{}\label{crysthird2}
Assume that neither $L_{1}$ nor $L_{2}$ are $\phi$-invariant and $L_{1}\subset E(\baset,\baseth)$. Then $L_{2}\cap E(\baset,\baseth)=L_{1}$ and $\baseth\not\in L_{2}$. By admissibility, $0=\hodgen(E\baseth)\leq\newtonn(E\baseth)=\val\lambda$ and $s=\hodgen(E(\baset,\baseth))\leq\newtonn(E(\baset,\baseth))=2\val\lambda\hspace{0.1cm}(s\leq 2r)$. So admissible filtered $\phi$-modules occur in this case if and only if $s\leq 2r$. The corresponding representations are non-split reducible with submodule $E(\baset,\baseth)$ if $s=2r$ and irreducible if $s<2r$.

\subsubsection{}\label{crysthird3}
Assume that neither $L_{1}$ nor $L_{2}$ are $\phi$-invariant, $\baseth\not\in \L_{2}$, and $L_{1}\not\subset E(\baset,\baseth)$. By admissibility, $0=\hodgen(E\baseth)\leq\newtonn(E\baseth)=\val\lambda$ and $r=\hodgen(E(\baset,\baseth))\leq\newtonn(E(\baset,\baseth))=2\val\lambda$. So admissible filtered $\phi$-modules always occur in this case and the corresponding representations are irreducible.

\subsection{The fourth case of $N=0$}
Assume that $\phi$ has a minimal polynomial of the form $(x-\lambda)(x-\lambdath)$ and a characteristic polynomial of the form $(x-\lambda)^{2}(x-\lambdath)$  with $\lambda\neq\lambdath$. So we may assume that $\phi\baseo=\lambda\baseo$, $\phi\baset=\lambda\baset$, and $\phi\baseth=\lambdath\baseth$. By admissibility, we have $$2\val{\lambda}+\val{\lambda_{3}}=r+s.$$

\begin{lemm} \label{typeofPfourth}
For a $3\times3$-matrix $P=(P_{i,j})$, $P[\phi]=[\phi] P$ if and only if $P$ is a matrix with $P_{1,3}=P_{2,3}=P_{3,1}=P_{3,2}=0$.
\end{lemm}

\begin{proof}
Let $P$ be as in the proof of Lemma \ref{typeofPsecond}. We compute the matrix products:
$$P[\phi]=
\begin{tiny}
\left(
  \begin{array}{ccc}
    \lambda a & \lambda b & \lambdath c \\
    \lambda d & \lambda e & \lambdath f \\
    \lambda g & \lambda h & \lambdath i \\
  \end{array}
\right)
\end{tiny}\mbox{ and }
[\phi] P=
\begin{tiny}
\left(
  \begin{array}{ccc}
   \lambda a &\lambda b &\lambda c \\
   \lambda d & \lambda e & \lambda f \\
   \lambdath g & \lambdath h & \lambdath i \\
  \end{array}
\right)
\end{tiny}.
$$
By comparing the entries we get the result.
\end{proof}

\begin{lemm}\label{crysinvariantfourth}
\begin{enumerate}
\item Every $1$-dimensional subspace of $E(\baseo,\baset)$ and $E\baseth$ are the only $\phi$-invariant subspaces of dimension $1$.
\item For any $(a,b)\in E^{2}\setminus\{(0,0)\}$ $E(a\baseo+b\baset,\baseth)$ and $E(\baseo,\baset)$ are the only $\phi$-invariant subspaces of dimension $2$.
\end{enumerate}
\end{lemm}

\begin{proof}
The vector subspaces listed above are obviously $\phi$-invariant.
The case (1) is clear since every $\phi$-invariant subspace of dimension 1 is the eigenspaces. Let $D_{2}$ be a $\phi$-invariant subspace of dimension 2. For the case (2), assume that $\baseth\not\in D_{2}$ and $D_{2}\neq E(\baseo,\baset)$. Then $a\baseo+b\baset+\baseth\in D_{2}$, and $\phi(a\baseo+b\baset+\baseth)=\lambda(a\baseo+b\baset+\baseth)+(\lambdath-\lambda)\baseth$. Hence, $\baseth\in D_{2}$, which is contradiction.
\end{proof}

We start to collect the admissible filtered $\phi$-modules in this case.
\subsubsection{}\label{}
Assume first that $L_{1}=E\baseth$. Then $L_{2}$ is $\phi$-invariant. So, by admissibility,
$r+s=\hodgen(L_{2})\leq\newtonn(L_{2})=\val\lambda+\val\lambdath$ and
$r=\hodgen(E(\baseo,\baset))\leq\newtonn(E(\baseo,\baset))=2\val\lambda$, which contradicts to $2\val\lambda+\val\lambdath=r+s$.

\subsubsection{}\label{}
Assume that $L_{1}\subset E(\baseo,\baset)$. Then $L_{1}$ is $\phi$-invariant. So, by admissibility, $s=\hodgen(L_{1})\leq\newtonn(L_{1})=\val\lambda$ and $0\leq\hodgen(E\baseth)\leq\newtonn(E\baseth)=\val\lambdath$, which contradicts to $2\val\lambda+\val\lambdath=r+s$.

\subsubsection{}\label{}
Assume that $L_{1}$ is not $\phi$-invariant, but $L_{2}$ is. Then $L_{2}\neq E(\baseo,\baset)$. So, by admissibility, $r+s=\hodgen(L_{2})\leq\newtonn(L_{2})=\val\lambda+\val\lambdath$ and $r=\hodgen(L_{2}\cap E(\baseo,\baset))\leq\newtonn(L_{2}\cap E(\baseo,\baset))=\val\lambda$, which contradicts to $2\val\lambda+\val\lambdath=r+s$.

\subsubsection{}\label{crysfourth1}
Assume finally that neither $L_{1}$ nor $L_{2}$ are $\phi$-invariant. Then $L_{1}\not\subset E(\baseo,\baset)$ and $\baseth\not\in L_{2}$. Let $D_{1},D_{2}$ be $\phi$-invariant subspaces of dimension 1 and of dimension $2$, respectively. By admissibility, $0=\hodgen(E\baseth)\leq\newtonn(E\baseth)=\val\lambdath$, $r=\hodgen(E(\baseo,\baset))\leq\newtonn(E(\baseo,\baset))=2\val\lambda$,
for $D_{1}$ with $D_{1}\subset L_{2}$ $r=\hodgen(D_{1})\leq\newtonn(D_{1})=\val\lambda$,
for $D_{1}$ with $D_{1}\nsubseteq L_{2}$ and $D_{1}\neq E\baseth$ $0=\hodgen(D_{1})\leq\newtonn(D_{1})=\val\lambda$,
for $D_{2}$ with $L_{1}\subset D_{2}$ $s=\hodgen(D_{2})\leq\newtonn(D_{2})=\val\lambda+\val\lambdath$,
for $D_{2}$ with $L_{1}\nsubseteq D_{2}$ and $\baseth\in D_{2}$ $r=\hodgen(D_{2})\leq\newtonn(D_{2})=\val\lambda+\val\lambdath$.
So, for $\val\lambda=r$ and $\val\lambdath=s-r$, we have admissible filtered $\phi$-modules in this case. The submodules are $L_{2}\cap E(\baseo,\baset)$ and $D_{2}$ with $L_{1}\subset D_{2}$. Hence, the corresponding representations are decomposable.

\subsection{The fifth case of $N=0$}
Assume that $\phi$ has a minimal polynomial $(x-\lambda)^{2}(x-\lambdath)$ with $\lambda\neq\lambdath$. So we may assume that $\phi\baseo=\lambda\baseo+\baset$, $\phi\baset=\lambda\baset$, and $\phi\baseth=\lambdath\baseth$. By admissibility, we have $$2\val{\lambda}+\val{\lambdath}=r+s.$$

\begin{lemm} \label{typeofPfifth}
For a $3\times3$-matrix $P=(P_{i,j})$, $P[\phi]=[\phi] P$ if and only if $P$ is a lower triangle matrix with $P_{1,1}=P_{2,2}$ and $P_{3,1}=0=P_{3,2}$.
\end{lemm}

\begin{proof}
Let $P$ be as in the proof of Lemma \ref{typeofPsecond}. Then we compute the matrix products:
$$P[\phi]=
\begin{tiny}\left(
  \begin{array}{ccc}
    \lambda a + b & \lambda b & \lambdath c \\
    \lambda d + e & \lambda e & \lambdath f \\
    \lambda g + h & \lambda h & \lambdath i \\
  \end{array}
\right)
\end{tiny}\mbox{ and }
[\phi] P=
\begin{tiny}
\left(
  \begin{array}{ccc}
   \lambda a &\lambda b &\lambda c \\
   a + \lambda d & b + \lambda e & c + \lambda f \\
   \lambdath g &\lambdath h &\lambdath i \\
  \end{array}
\right)
\end{tiny}.$$
By comparing the entries, we get the result.
\end{proof}

\begin{lemm}\label{crysinvariantfifth}
$E\baset$, $E\baseth$, $E(\baseo,\baset)$, and $E(\baset,\baseth)$ are the only nontrivial and proper $\phi$-invariant subspaces.
\end{lemm}

\begin{proof}
The vector subspaces listed above are obviously $\phi$-invariant. Obviously $E\baset,E\baseth$ are the only $\phi$-invariant subspaces of dimension 1 since they are the eigenspaces. Let $D_{2}$ be a $\phi$-invariant subspace of dimension 2. Assume that $D_{2}$ is not a subspace listed above. Then there is an element $a\baseo+b\baset+\baseth\in D_{2}$ with $a\neq 0$, and $\phi(a\baseo+b\baset+\baseth)=\lambda(a\baseo+b\baset+\baseth)+a\baset+(\lambdath-\lambda)\baseth$. Thus $a\baset+(\lambdath-\lambda)\baseth\in D_{2}$ and $\phi(a\baset+(\lambdath-\lambda)\baseth)=\lambda(a\baset+(\lambdath-\lambda)\baseth)+ (\lambdath-\lambda)^{2}\baseth$. Hence, $\baseth\in D_{2}$ and so $D_{2}$ is of dimension 3, which is contradiction.
\end{proof}

We start to collect the admissible filtered $\phi$-modules in this case.
\subsubsection{}\label{}
Assume first that $L_{1}=E\baset$. By admissibility, $s=\hodgen(E\baset)\leq\newtonn(E\baset)=\val\lambda$ and $0\leq\hodgen(E\baseth)\leq\newtonn(E\baseth)=\val\lambdath$, which contradicts to $2\val\lambda+\val\lambdath=r+s$.

\subsubsection{}\label{crysfifth1}
Assume that $L_{1}=E\baseth$ and $L_{2}\neq E(\baset,\baseth)$. Then $\baset\not\in L_{2}$. By admissibility, $0=\hodgen(E\baset)\leq\newtonn(E\baset)=\val\lambda$, $s=\hodgen(E\baseth)\leq\newtonn(E\baseth)=\val\lambdath$, $r=\hodgen(E(\baseo,\baset))\leq\newtonn(E(\baseo,\baset))=2\val\lambda$, and $s=\hodgen(E(\baset,\baseth))\leq\newtonn(E(\baset,\baseth))=\val\lambda+\val\lambdath$. So, for $\val\lambda=\frac{r}{2}$ and $\val\lambdath=s$, we have admissible filtered $\phi$-modules. The corresponding representations are decomposable with submodules $E(\baseo,\baset)$ and $E\baseth$.

\subsubsection{}\label{}
Assume that $L_{2}=E(\baset,\baseth)$. By admissibility, $r+s=\hodgen(E(\baset,\baseth))\leq\newtonn(E(\baset,\baseth))=\val\lambda+\val\lambdath$ and $r\leq\hodgen(E\baset)\leq\newtonn(E\baset)=\val\lambda$, which contradicts to $2\val\lambda+\val\lambdath=r+s$.

\subsubsection{}\label{crysfifth2}
Assume that $L_{1}$ is not $\phi$-invariant and $L_{2}=E(\baseo,\baset)$. By admissibility, $r=\hodgen(E\baset)\leq\newtonn(E\baset)=\val\lambda$, $0=\hodgen(E\baseth)\leq\newtonn(E\baseth)=\val\lambdath$, $r+s=\hodgen(E(\baseo,\baset))\leq\newtonn(E(\baseo,\baset))=2\val\lambda$, and $r=\hodgen(E(\baset,\baseth))\leq\newtonn(E(\baset,\baseth))=\val\lambda+\val\lambdath$. So, for $\val\lambda=\frac{r+s}{2}$ and $\val\lambdath=0$, we have admissible filtered $\phi$-modules. The corresponding representations are decomposable with submodules $E(\baseo,\baset)$ and $E\baseth$.

\subsubsection{}\label{crysfifth3}
Assume that neither $L_{1}$ nor $L_{2}$ are $\phi$-invariant, $L_{1}\subset E(\baseo,\baset)$, and $\baseth\in L_{2}$. Then $\baset\not\in L_{2}$ and, by admissibility, $0=\hodgen(E\baset)\leq\newtonn(E\baset)=\val\lambda$, $r=\hodgen(E\baseth)\leq\newtonn(E\baseth)=\val\lambdath$, $s=\hodgen(E(\baseo,\baset))\leq\newtonn(E(\baseo,\baset))=2\val\lambda$, and $r=\hodgen(E(\baset,\baseth))\leq\newtonn(E(\baset,\baseth))=\val\lambda+\val\lambdath$. So, for $\val\lambda=\frac{s}{2}$ and $\val\lambdath=r$, we have admissible filtered $\phi$-modules. The submodules are $E\baseth$ and $E(\baseo,\baset)$. So the corresponding representations are decomposable.

\subsubsection{}\label{crysfifth4}
Assume that neither $L_{1}$ nor $L_{2}$ are $\phi$-invariant, $L_{1}\subset E(\baseo,\baset)$, and $\baseth\not\in L_{2}$. Then $\baset\not\in L_{2}$ and, by admissibility, $0=\hodgen(E\baset)\leq\newtonn(E\baset)=\val\lambda$, $0=\hodgen(E\baseth)\leq\newtonn(E\baseth)=\val\lambdath$, $s=\hodgen(E(\baseo,\baset))\leq\newtonn(E(\baseo,\baset))=2\val\lambda$, and $r=\hodgen(E(\baset,\baseth))\leq\newtonn(E(\baset,\baseth))=\val\lambda+\val\lambdath$. So, for $\frac{s}{2}\leq\val\lambda\leq\frac{r+s}{2}$ and $2\val\lambda+\val\lambdat=r+s$, we have admissible filtered $\phi$-modules. The corresponding representations are non-split reducible with submodule $E(\baseo,\baset)$ if $\val\lambda=\frac{s}{2}$, non-split reducible with submodule $E\baseth$ if $\val\lambda=\frac{r+s}{2}$, and irreducible if $\frac{s}{2}<\val\lambda<\frac{r+s}{2}$.

\subsubsection{}\label{crysfifth5}
Assume that neither $L_{1}$ nor $L_{2}$ are $\phi$-invariant and $L_{1}\subset E(\baset,\baseth)$. Then $\baset,\baseth\not\in L_{2}$ and, by admissibility, $0=\hodgen(E\baset)\leq\newtonn(E\baset)=\val\lambda$, $0=\hodgen(E\baseth)\leq\newtonn(E\baseth)=\val\lambdath$, $r=\hodgen(E(\baseo,\baset))\leq\newtonn(E(\baseo,\baset))=2\val\lambda$, and $s=\hodgen(E(\baset,\baseth))\leq\newtonn(E(\baset,\baseth))=\val\lambda+\val\lambdath$. So, for $\frac{r}{2}\leq\val\lambda\leq r$ and $2\val\lambda+\val\lambdat=r+s$, we have admissible filtered $\phi$-modules. The corresponding representations are non-split reducible with submodule $E(\baseo,\baset)$ if $\val\lambda=\frac{r}{2}$, non-split reducible with submodule $E(\baset,\baseth)$ if $\val\lambda=r$, and irreducible if $\frac{r}{2}<\val\lambda<r$.

\subsubsection{}\label{crysfifth6}
Assume that neither $L_{1}$ nor $L_{2}$ are $\phi$-invariant, $L_{1}$ is not contained in any $\phi$-invariant subspace, and $\baset\in L_{2}$. Then $\baseth\not\in L_{2}$ and, by admissibility, $r=\hodgen(E\baset)\leq\newtonn(E\baset)=\val\lambda$, $0=\hodgen(E\baseth)\leq\newtonn(E\baseth)=\val\lambdath$, $r=\hodgen(E(\baseo,\baset))\leq\newtonn(E(\baseo,\baset))=2\val\lambda$, and $r=\hodgen(E(\baset,\baseth))\leq\newtonn(E(\baset,\baseth))=\val\lambda+\val\lambdath$. So, for $r\leq\val\lambda\leq \frac{r+s}{2}$ and $2\val\lambda+\val\lambdat=r+s$, we have admissible filtered $\vphi$-modules. The corresponding representations are non-split reducible with submodule $E\baset$ if $\val\lambda=r$, non-split reducible with submodule $E\baseth$ if $\val\lambda=\frac{r+s}{2}$, and irreducible if $r<\val\lambda<\frac{r+s}{2}$.

\subsubsection{}\label{crysfifth7}
Assume that neither $L_{1}$ nor $L_{2}$ are $\phi$-invariant, $L_{1}$ is not contained in any $\phi$-invariant subspace, and $\baseth\in L_{2}$. Then $\baset\not\in L_{2}$ and, by admissibility,
$0=\hodgen(E\baset)\leq\newtonn(E\baset)=\val\lambda$, $r=\hodgen(E\baseth)\leq\newtonn(E\baseth)=\val\lambdath$, $r=\hodgen(E(\baseo,\baset))\leq\newtonn(E(\baseo,\baset))=2\val\lambda$, and $r=\hodgen(E(\baset,\baseth))\leq\newtonn(E(\baset,\baseth))=\val\lambda+\val\lambdath$. So, for $\frac{r}{2}\leq\val\lambda\leq \frac{s}{2}$ and $2\val\lambda+\val\lambdat=r+s$, we have admissible filtered $\phi$-modules. The corresponding representations are non-split reducible with submodule $E(\baseo,\baset)$ if $\val\lambda=\frac{r}{2}$, non-split reducible with submodule $E\baseth$ if $\val\lambda=\frac{s}{2}$, and irreducible if $\frac{r}{2}<\val\lambda<\frac{s}{2}$.

\subsubsection{}\label{crysfifth8}
Assume that neither $L_{1}$ nor $L_{2}$ are $\phi$-invariant, $L_{1}$ is not contained in any $\phi$-invariant subspace, and $\baset,\baseth\not\in L_{2}$. By admissibility, $0=\hodgen(E\baset)\leq\newtonn(E\baset)=\val\lambda$, $0=\hodgen(E\baseth)\leq\newtonn(E\baseth)=\val\lambdath$, $r=\hodgen(E(\baseo,\baset))\leq\newtonn(E(\baseo,\baset))=2\val\lambda$, and $r=\hodgen(E(\baset,\baseth))\leq\newtonn(E(\baset,\baseth))=\val\lambda+\val\lambdath$. So, for $\frac{r}{2}\leq\val\lambda\leq \frac{r+s}{2}$ and $2\val\lambda+\val\lambdat=r+s$, we have admissible filtered $\phi$-modules. The corresponding representations are non-split reducible with submodule $E(\baseo,\baset)$ if $\val\lambda=\frac{r}{2}$, non-split reducible with submodule $E\baseth$ if $\val\lambda=\frac{r+s}{2}$, and irreducible if $\frac{r}{2}<\val\lambda<\frac{r+s}{2}$.

\subsection{The sixth case of $N=0$}
Assume that $\phi$ has a minimal polynomial of the form $(x-\lambdao)(x-\lambdat)(x-\lambdath)$ with distinct $\lambda_{i}$. So we may assume that $\phi\baseo=\lambdao\baseo$, $\phi\baset=\lambdat\baset$, and $\phi\baseth=\lambdath\baseth$. By admissibility, we have $$\val{\lambda_{1}}+\val{\lambda_{2}}+\val{\lambda_{3}}=r+s.$$ Without loss of generality, we assume that $$\val{\lambdao}\geq\val{\lambdat}\geq\val{\lambdath}.$$

\begin{lemm}\label{typeofPsixth}
Let $[\phi]=[\lambdao,\lambdat,\lambdath]$ and $[\phi']=[\lambdao',\lambdat',\lambdath']$ be the diagonal matrices. Then, for a $3\times3$-matrix $P$, $P[\phi]=[\phi']P$ if and only if
$$P=\left\{
  \begin{array}{ll}
     \tiny\left(
  \begin{array}{ccc}
    x & 0 & 0 \\
    0 & y & 0 \\
    0 & 0 & z \\
  \end{array}
\right), & \hbox{if $\lambdao=\lambdao',\lambdat=\lambdat',\lambdath=\lambdath'$;} \\
     \tiny\left(
  \begin{array}{ccc}
    0 & 0 & z \\
    x & 0 & 0 \\
    0 & y & 0 \\
  \end{array}
\right), & \hbox{if $\lambdao=\lambdat',\lambdat=\lambdath',\lambdath=\lambdao'$;} \\
     \tiny\left(
  \begin{array}{ccc}
    0 & y & 0 \\
    0 & 0 & z \\
    x & 0 & 0 \\
  \end{array}
\right), & \hbox{if $\lambdao=\lambdath',\lambdat=\lambdao',\lambdath=\lambdat'$;} \\
     \tiny\left(
  \begin{array}{ccc}
    x & 0 & 0 \\
    0 & 0 & y \\
    0 & z & 0 \\
  \end{array}
\right), & \hbox{if $\lambdao=\lambdao',\lambdat=\lambdath',\lambdath=\lambdat'$;} \\
     \tiny\left(
  \begin{array}{ccc}
    0 & 0 & z \\
    0 & y & 0 \\
    x & 0 & 0 \\
  \end{array}
\right), & \hbox{if $\lambdao=\lambdath',\lambdat=\lambdat',\lambdath=\lambdao'$;} \\
     \tiny\left(
  \begin{array}{ccc}
    0 & y & 0 \\
    x & 0 & 0 \\
    0 & 0 & z \\
  \end{array}
\right), & \hbox{if $\lambdao=\lambdat',\lambdat=\lambdao',\lambdath=\lambdath'$.}
  \end{array}
\right.
$$
\end{lemm}

\begin{proof}
Let $P$ be as in the proof of the Lemma \ref{typeofPsecond}. Then we compute the matrix products:
$$P[\phi]=
\begin{tiny}
\left(
  \begin{array}{ccc}
    \lambdao a & \lambdat b & \lambdath c \\
    \lambdao d & \lambdat e & \lambdath f \\
    \lambdao g & \lambdat h & \lambdath i \\
  \end{array}
\right)
\end{tiny}
\mbox{ and }
[\phi']P=
\begin{tiny}
\left(
  \begin{array}{ccc}
   \lambdao' a &\lambdao' b &\lambdao' c \\
   \lambdat' d &\lambdat' e &\lambdat' f \\
   \lambdath' g &\lambdath' h &\lambdath' i \\
  \end{array}
\right)
\end{tiny}.$$
By comparing the entries for each case, we get the results.
\end{proof}

\begin{lemm}\label{crysinvariantsixth}
$E\baseo$, $E\baset$, $E\baseth$, $E(\baseo,\baset)$, $E(\baset,\baseth)$, and $E(\baseo,\baseth)$ are the only nontrivial and proper $\phi$-invariant subspaces.
\end{lemm}

\begin{proof}
The vector subspaces listed above are obviously $\phi$-invariant. For the dimension 1, it is clear since every $\phi$-invariant subspace of dimension 1 is the eigenspace. Let $D_{2}$ be a $\phi$-invariant subspace of dimension 2 that are not listed above. Then there is an element $a\lambdao+b\lambdat+c\lambdath$ in $D_{2}$ with $abc\neq0$ and $\phi(a\lambdao+b\lambdat+c\lambdath)=\lambdao(a\lambdao+b\lambdat+c\lambdath)+(\lambdat-\lambdao)b\baset+(\lambdath-\lambdao)c\baseth$. So $(\lambdat-\lambdao)b\baset+(\lambdath-\lambdao)c\baseth\in D_{2}$ and $\phi((\lambdat-\lambdao)b\baset+(\lambdath-\lambdao)c\baseth)=\lambdat((\lambdat-\lambdao)b\baset+(\lambdath-\lambdao)c\baseth)+(\lambdath-\lambdat)(\lambdath-\lambdao)c\baseth$. Hence, $\baseth\in D_{2}$, i.e., $D_{2}$ is of dimension 3, which is contradiction.
\end{proof}

We start to collect the admissible filtered $\phi$-modules in this case.
\subsubsection{}\label{cryssixth1}
Assume first that $L_{1}=E\baseo$ and $L_{2}=E(\baseo,\baset)$. By admissibility, we have $s=\hodgen(E\baseo)\leq\newtonn(E\baseo)=\val{\lambdao}$, $r=\hodgen(E\baset)\leq\newtonn(E\baset)=\val{\lambdat}$, $0=\hodgen(E\baseth)\leq\newtonn(E\baseth)=\val{\lambdath}$, $r+s=\hodgen(E(\baseo,\baset))\leq\newtonn(E(\baseo,\baset))=\val{\lambdao}+\val{\lambdat}$, $r=\hodgen(E(\baset,\baseth))\leq\newtonn(E(\baset,\baseth))=\val{\lambdat}+\val{\lambdath}$, and $s=\hodgen(E(\baseo,\baseth))\leq\newtonn(E(\baseo,\baseth))=\val{\lambdao}+\val{\lambdath}$. So, for $\val\lambdao=s$, $\val\lambdat=r$, and $\val\lambdath=0$, we have admissible filtered $\phi$-modules in this case. The corresponding representations are obviously decomposable.

\subsubsection{}\label{}
Assume that $L_{1}=E\baseo$ and $L_{2}=E(\baseo,\baseth)$. By the same argument as above, we have $\val\lambdao=s,\val\lambdat=0,\val\lambdath=r$ which contradict to the assumption $\val\lambdao\geq\val\lambdat\geq\val\lambdath$. Similarly, if $L_{1}$ is a $\phi$-invariant subspace other than $E\baseo$ and if $L_{2}$ is a $\phi$-invariant subspace, it violates the assumption $\val\lambdao\geq\val\lambdat\geq\val\lambdath$.

\subsubsection{}\label{cryssixth2}
Assume that $L_{1}=E\baseo$ and $L_{2}$ is not $\phi$-invariant. Then $\baset,\baseth\not\in L_{2}$ and, by admissibility, we have $s=\hodgen(E\baseo)\leq\newtonn(E\baseo)=\val{\lambdao}$, $0=\hodgen(E\baset)\leq\newtonn(E\baset)=\val{\lambdat}$, $0=\hodgen(E\baseth)\leq\newtonn(E\baseth)=\val{\lambdath}$, $s=\hodgen(E(\baseo,\baset))\leq\newtonn(E(\baseo,\baset))=\val{\lambdao}+\val{\lambdat}$, $r=\hodgen(E(\baset,\baseth))\leq\newtonn(E(\baset,\baseth))=\val{\lambdat}+\val{\lambdath}$, and $s=\hodgen(E(\baseo,\baseth))\leq\newtonn(E(\baseo,\baseth))=\val{\lambdao}+\val{\lambdath}$. So, for $s=\val\lambdao\geq\val\lambdat\geq\val\lambdath\geq0$ and $\val\lambdat+\val\lambdath=r$, we have admissible filtered $\phi$-modules in this case. The corresponding representations are decomposable with submodules $E\baseo$ and $E(\baset,\baseth)$ in this case; moreover, if $\val\lambdath=0$ then $E\baseth$ and $E(\baseo,\baseth)$ are also submodules.

\subsubsection{}\label{}
Assume that $L_{1}=E\baset$ and $L_{2}$ is not $\vphi$-invariant. Then $\baseo,\baseth\not\in L_{2}$ and, by admissibility, we have $0=\hodgen(E\baseo)\leq\newtonn(E\baseo)=\val{\lambdao}$, $s=\hodgen(E\baset)\leq\newtonn(E\baset)=\val{\lambdat}$, $0=\hodgen(E\baseth)\leq\newtonn(E\baseth)=\val{\lambdath}$, $s=\hodgen(E(\baseo,\baset))\leq\newtonn(E(\baseo,\baset))=\val{\lambdao}+\val{\lambdat}$, $s=\hodgen(E(\baset,\baseth))\leq\newtonn(E(\baset,\baseth))=\val{\lambdat}+\val{\lambdath}$, and $r=\hodgen(E(\baseo,\baseth))\leq\newtonn(E(\baseo,\baseth))=\val{\lambdao}+\val{\lambdath}$. So we get $\val\lambdao\geq\val\lambdat\geq s$ and $\val\lambdath\geq0$, which violates $\val\lambdao+\val\lambdat+\val\lambdath=r+s$. Hence, there are no admissible filtered $\phi$-modules in this case.

\subsubsection{}\label{}
Assume that $L_{1}=E\baseth$ and $L_{2}$ is not $\vphi$-invariant. Then $\baseo,\baset\not\in L_{2}$ and, by admissibility, we have $0=\hodgen(E\baseo)\leq\newtonn(E\baseo)=\val{\lambdao}$, $0=\hodgen(E\baset)\leq\newtonn(E\baset)=\val{\lambdat}$, $s=\hodgen(E\baseth)\leq\newtonn(E\baseth)=\val{\lambdath}$, $r=\hodgen(E(\baseo,\baset))\leq\newtonn(E(\baseo,\baset))=\val{\lambdao}+\val{\lambdat}$, $s=\hodgen(E(\baset,\baseth))\leq\newtonn(E(\baset,\baseth))=\val{\lambdat}+\val{\lambdath}$, and $s=\hodgen(E(\baseo,\baseth))\leq\newtonn(E(\baseo,\baseth))=\val{\lambdao}+\val{\lambdath}$. So we get $\val\lambdao\geq\val\lambdat\geq\val\lambdath\geq s$, which violates $\val\lambdao+\val\lambdat+\val\lambdath=r+s$. Hence, there are no admissible filtered $\phi$-modules in this case.

\subsubsection{}\label{cryssixth3}
Assume that $L_{1}$ is not $\phi$-invariant and $L_{2}=E(\baseo,\baset)$. By admissibility, we have $r=\hodgen(E\baseo)\leq\newtonn(E\baseo)=\val{\lambdao}$, $r=\hodgen(E\baset)\leq\newtonn(E\baset)=\val{\lambdat}$, $0=\hodgen(E\baseth)\leq\newtonn(E\baseth)=\val{\lambdath}$, $r+s=\hodgen(E(\baseo,\baset))\leq\newtonn(E(\baseo,\baset))=\val{\lambdao}+\val{\lambdat}$, $r=\hodgen(E(\baset,\baseth))\leq\newtonn(E(\baset,\baseth))=\val{\lambdat}+\val{\lambdath}$, and $r=\hodgen(E(\baseo,\baseth))\leq\newtonn(E(\baseo,\baseth))=\val{\lambdao}+\val{\lambdath}$. So, for $\val\lambdao\geq\val\lambdat\geq r$, $\val\lambdath=0$, and $\val\lambdao+\val\lambdat=r+s$, we have admissible filtered $\phi$-modules in this case. The corresponding representations are decomposable with submodules $E\baseth$ and $E(\baseo,\baset)$; moreover, if $\val\lambdat=r$ then $E\baset$ and $E(\baset,\baseth)$ are also submodules.

\subsubsection{}\label{}
Assume that $L_{1}$ is not $\phi$-invariant and $L_{2}=E(\baset,\baseth)$. By admissibility, we have
$0=\hodgen(E\baseo)\leq\newtonn(E\baseo)=\val{\lambdao}$, $r=\hodgen(E\baset)\leq\newtonn(E\baset)=\val{\lambdat}$, $r=\hodgen(E\baseth)\leq\newtonn(E\baseth)=\val{\lambdath}$, $r=\hodgen(E(\baseo,\baset))\leq\newtonn(E(\baseo,\baset))=\val{\lambdao}+\val{\lambdat}$, $r+s=\hodgen(E(\baset,\baseth))\leq\newtonn(E(\baset,\baseth))=\val{\lambdat}+\val{\lambdath}$, and $r=\hodgen(E(\baseo,\baseth))\leq\newtonn(E(\baseo,\baseth))=\val{\lambdao}+\val{\lambdath}$. So we get $\val\lambdao=0$, and $\val\lambdat+\val\lambdath=r+s$, which violates the assumption $\val\lambdao\geq\val\lambdat\geq\val\lambdath$.

\subsubsection{}\label{}
Assume that $L_{1}$ is not $\phi$-invariant and $L_{2}=E(\baseo,\baseth)$. By admissibility, we have $r=\hodgen(E\baseo)\leq\newtonn(E\baseo)=\val{\lambdao}$, $0=\hodgen(E\baset)\leq\newtonn(E\baset)=\val{\lambdat}$, $r=\hodgen(E\baseth)\leq\newtonn(E\baseth)=\val{\lambdath}$, $r=\hodgen(E(\baseo,\baset))\leq\newtonn(E(\baseo,\baset))=\val{\lambdao}+\val{\lambdat}$, $r=\hodgen(E(\baset,\baseth))\leq\newtonn(E(\baset,\baseth))=\val{\lambdat}+\val{\lambdath}$, and $r+s=\hodgen(E(\baseo,\baseth))\leq\newtonn(E(\baseo,\baseth))=\val{\lambdao}+\val{\lambdath}$. So we get $\val\lambdat=0$, $\val\lambdath\geq r$, and $\val\lambdao+\val\lambdath=r+s$, which violates the assumption $\val\lambdao\geq\val\lambdat\geq\val\lambdath$.

\subsubsection{}\label{cryssixth4}
Assume that neither $L_{1}$ nor $L_{2}$ are $\phi$-invariant, $L_{1}\subset E(\baseo,\baset)$, and $\baseth\in L_{2}$. Then $\baseo,\baset\not\in L_{2}$ and, by admissibility, we have $0=\hodgen(E\baseo)\leq\newtonn(E\baseo)=\val{\lambdao}$, $0=\hodgen(E\baset)\leq\newtonn(E\baset)=\val{\lambdat}$, $r=\hodgen(E\baseth)\leq\newtonn(E\baseth)=\val{\lambdath}$, $s=\hodgen(E(\baseo,\baset))\leq\newtonn(E(\baseo,\baset))=\val{\lambdao}+\val{\lambdat}$, $r=\hodgen(E(\baset,\baseth))\leq\newtonn(E(\baset,\baseth))=\val{\lambdat}+\val{\lambdath}$, and $r=\hodgen(E(\baseo,\baseth))\leq\newtonn(E(\baseo,\baseth))=\val{\lambdao}+\val{\lambdath}$. So, for $\val\lambdao\geq\val\lambdat\geq\val\lambdath=r$, $\val\lambdao+\val\lambdat=s$, and $s\geq2r$, we have admissible filtered $\phi$-modules. The corresponding representations are decomposable with submodules $E\baseth$ and $E(\baseo,\baset)$.

\subsubsection{}\label{cryssixth5}
Assume that neither $L_{1}$ nor $L_{2}$ are $\phi$-invariant, $L_{1}\subset E(\baseo,\baset)$, and $\baseth\not\in L_{2}$. Then $\baseo,\baset\not\in L_{2}$ and, by admissibility, we have $0=\hodgen(E\baseo)\leq\newtonn(E\baseo)=\val{\lambdao}$, $0=\hodgen(E\baset)\leq\newtonn(E\baset)=\val{\lambdat}$, $0=\hodgen(E\baseth)\leq\newtonn(E\baseth)=\val{\lambdath}$, $s=\hodgen(E(\baseo,\baset))\leq\newtonn(E(\baseo,\baset))=\val{\lambdao}+\val{\lambdat}$, $r=\hodgen(E(\baset,\baseth))\leq\newtonn(E(\baset,\baseth))=\val{\lambdat}+\val{\lambdath}$, and $r=\hodgen(E(\baseo,\baseth))\leq\newtonn(E(\baseo,\baseth))=\val{\lambdao}+\val{\lambdath}$. So, for $s\geq\val\lambdao\geq\val\lambdat\geq\val\lambdath\geq0$, $\val\lambdath\leq r$, and $\val\lambdao+\val\lambdat+\val\lambdath=r+s$, we have admissible filtered $\phi$-modules. The submodules are $E\baseth$ if $\val\lambdath=0$, $E(\baseo,\baset)$ if $\val\lambdath=r$, and $E(\baset,\baseth)$ if $\val\lambdao=s$.
Hence, we have the irreducible admissible filtered $\phi$-modules if $s>\val\lambdao\geq\val\lambdat\geq\val\lambdath>0$ and $\val\lambdath< r.$

\subsubsection{}\label{cryssixth6}
Assume that neither $L_{1}$ nor $L_{2}$ are $\phi$-invariant, $L_{1}\subset E(\baset,\baseth)$, and $\baseo\in L_{2}$. Then $\baset,\baseth\not\in L_{2}$ and, by admissibility, we have $r=\hodgen(E\baseo)\leq\newtonn(E\baseo)=\val{\lambdao}$, $0=\hodgen(E\baset)\leq\newtonn(E\baset)=\val{\lambdat}$, $0=\hodgen(E\baseth)\leq\newtonn(E\baseth)=\val{\lambdath}$, $r=\hodgen(E(\baseo,\baset))\leq\newtonn(E(\baseo,\baset))=\val{\lambdao}+\val{\lambdat}$, $s=\hodgen(E(\baset,\baseth))\leq\newtonn(E(\baset,\baseth))=\val{\lambdat}+\val{\lambdath}$, and $r=\hodgen(E(\baseo,\baseth))\leq\newtonn(E(\baseo,\baseth))=\val{\lambdao}+\val{\lambdath}$. So, for $r=\val\lambdao\geq\val\lambdat\geq\val\lambdath$, $\val\lambdat+\val\lambdath=s$, and $s\leq 2r$, we have admissible filtered $\phi$-modules. The corresponding representations are decomposable with submodules $E\baseo$ and $E(\baset,\baseth)$.

\subsubsection{}\label{cryssixth7}
Assume that neither $L_{1}$ nor $L_{2}$ are $\phi$-invariant, $L_{1}\subset E(\baset,\baseth)$, and $\baseo\not\in L_{2}$. Then $\baset,\baseth\not\in L_{2}$ and, by admissibility, we have
$0=\hodgen(E\baseo)\leq\newtonn(E\baseo)=\val{\lambdao}$, $0=\hodgen(E\baset)\leq\newtonn(E\baset)=\val{\lambdat}$, $0=\hodgen(E\baseth)\leq\newtonn(E\baseth)=\val{\lambdath}$, $r=\hodgen(E(\baseo,\baset))\leq\newtonn(E(\baseo,\baset))=\val{\lambdao}+\val{\lambdat}$, $s=\hodgen(E(\baset,\baseth))\leq\newtonn(E(\baset,\baseth))=\val{\lambdat}+\val{\lambdath}$, and $r=\hodgen(E(\baseo,\baseth))\leq\newtonn(E(\baseo,\baseth))=\val{\lambdao}+\val{\lambdath}$. So, for $r\geq\val\lambdao\geq\val\lambdat\geq\val\lambdath$, $\val\lambdao+\val\lambdat+\val\lambdath=r+s$, and $s\leq 2r$, we have admissible filtered $\phi$-modules. The submodules are $E(\baset,\baseth)$ if $\val\lambdao=r$. Hence, we have the irreducible admissible filtered $\phi$-modules if $r>\val\lambdao\geq\val\lambdat\geq\val\lambdath$ and $s\leq 2r-1.$

\subsubsection{}\label{cryssixth8}
Assume that neither $L_{1}$ nor $L_{2}$ are $\phi$-invariant, $L_{1}\subset E(\baseo,\baseth)$, and $\baset\in L_{2}$. Then $\baseo,\baseth\not\in L_{2}$ and, by admissibility, we have $0=\hodgen(E\baseo)\leq\newtonn(E\baseo)=\val{\lambdao}$, $r=\hodgen(E\baset)\leq\newtonn(E\baset)=\val{\lambdat}$, $0=\hodgen(E\baseth)\leq\newtonn(E\baseth)=\val{\lambdath}$, $r=\hodgen(E(\baseo,\baset))\leq\newtonn(E(\baseo,\baset))=\val{\lambdao}+\val{\lambdat}$, $r=\hodgen(E(\baset,\baseth))\leq\newtonn(E(\baset,\baseth))=\val{\lambdat}+\val{\lambdath}$, and $s=\hodgen(E(\baseo,\baseth))\leq\newtonn(E(\baseo,\baseth))=\val{\lambdao}+\val{\lambdath}$. So, for $\val\lambdao\geq\val\lambdat=r\geq\val\lambdath\geq0$ and $\val\lambdao+\val\lambdath=s$, we have admissible filtered $\phi$-modules. The submodules are $E\baset$ and $E(\baseo,\baseth)$. Moreover, $E\baseth$ and $E(\baset,\baseth)$ are also submodules if $\val\lambdao=s$. So the corresponding representations are decomposable.

\subsubsection{}\label{cryssixth9}
Assume that neither $L_{1}$ nor $L_{2}$ are $\phi$-invariant, $L_{1}\subset E(\baseo,\baseth)$, and $\baset\not\in L_{2}$. Then $\baseo,\baseth\not\in L_{2}$ and, by admissibility, we have $0=\hodgen(E\baseo)\leq\newtonn(E\baseo)=\val{\lambdao}$, $0=\hodgen(E\baset)\leq\newtonn(E\baset)=\val{\lambdat}$, $0=\hodgen(E\baseth)\leq\newtonn(E\baseth)=\val{\lambdath}$, $r=\hodgen(E(\baseo,\baset))\leq\newtonn(E(\baseo,\baset))=\val{\lambdao}+\val{\lambdat}$, $r=\hodgen(E(\baset,\baseth))\leq\newtonn(E(\baset,\baseth))=\val{\lambdat}+\val{\lambdath}$, and $s=\hodgen(E(\baseo,\baseth))\leq\newtonn(E(\baseo,\baseth))=\val{\lambdao}+\val{\lambdath}$. So, for $s\geq\val\lambdao\geq\val\lambdat\geq\val\lambdath$, $\val\lambdat\leq r$, and $\val\lambdao+\val\lambdat+\val\lambdath=r+s$, we have admissible filtered $\phi$-modules. The submodules are $E(\baseo,\baseth)$ if $\val\lambdat=r$, $E(\baset,\baseth)$ if $\val\lambdao=s$, and  $E\baseth$, $E(\baset,\baseth)$, $E(\baseo,\baseth)$ if $\val\lambdath=0$. Hence, we have the irreducible admissible filtered $\phi$-modules if $s>\val\lambdao\geq\val\lambdat\geq\val\lambdath$ and $\val\lambdat< r.$

\subsubsection{}\label{cryssixth10}
Assume that neither $L_{1}$ nor $L_{2}$ are $\phi$-invariant, $L_{1}$ is not contained in any $\phi$-invariant subspace, and $\baseo\in L_{2}$. Then $\baset,\baseth\not\in L_{2}$ and, by admissibility, we have $r=\hodgen(E\baseo)\leq\newtonn(E\baseo)=\val{\lambdao}$, $0=\hodgen(E\baset)\leq\newtonn(E\baset)=\val{\lambdat}$, $0=\hodgen(E\baseth)\leq\newtonn(E\baseth)=\val{\lambdath}$, $r=\hodgen(E(\baseo,\baset))\leq\newtonn(E(\baseo,\baset))=\val{\lambdao}+\val{\lambdat}$, $r=\hodgen(E(\baset,\baseth))\leq\newtonn(E(\baset,\baseth))=\val{\lambdat}+\val{\lambdath}$, and $r=\hodgen(E(\baseo,\baseth))\leq\newtonn(E(\baseo,\baseth))=\val{\lambdao}+\val{\lambdath}$. So, for $\val\lambdao\geq\val\lambdat\geq\val\lambdath\geq0$, $r\leq\val\lambdao\leq s$, and $\val\lambdao+\val\lambdat+\val\lambdath=r+s$, we have admissible filtered $\phi$-modules. The submodules are $E\baseth$ if $\val\lambdath=0$, $E\baseo$ if $\val\lambdao=r$, and
$E(\baset,\baseth)$ if $\val\lambdao=s$. Hence, we have the irreducible admissible filtered $\phi$-modules if $\val\lambdao\geq\val\lambdat\geq\val\lambdath>0$ and $r<\val\lambdao<s.$

\subsubsection{}\label{cryssixth11}
Assume that neither $L_{1}$ nor $L_{2}$ are $\phi$-invariant, $L_{1}$ is not contained in any $\phi$-invariant subspace, and $\baset\in L_{2}$. Then $\baseo,\baseth\not\in L_{2}$ and, by admissibility, we have $0=\hodgen(E\baseo)\leq\newtonn(E\baseo)=\val{\lambdao}$, $r=\hodgen(E\baset)\leq\newtonn(E\baset)=\val{\lambdat}$, $0=\hodgen(E\baseth)\leq\newtonn(E\baseth)=\val{\lambdath}$, $r=\hodgen(E(\baseo,\baset))\leq\newtonn(E(\baseo,\baset))=\val{\lambdao}+\val{\lambdat}$, $r=\hodgen(E(\baset,\baseth))\leq\newtonn(E(\baset,\baseth))=\val{\lambdat}+\val{\lambdath}$, and $r=\hodgen(E(\baseo,\baseth))\leq\newtonn(E(\baseo,\baseth))=\val{\lambdao}+\val{\lambdath}$. So, for $\val\lambdao\geq\val\lambdat\geq\val\lambdath\geq0$, $r\leq\val\lambdat$, and $\val\lambdao+\val\lambdat+\val\lambdath=r+s$, we have admissible filtered $\phi$-modules. The submodules are $E\baseth$ if $\val\lambdath=0$, $E\baset$ if $\val\lambdat=r$, $E(\baset,\baseth)$ if $\val\lambdao=s$. Hence, we have the irreducible admissible filtered $\phi$-modules if $\val\lambdao\geq\val\lambdat\geq\val\lambdath>0$ and $\val\lambdat>r.$

\subsubsection{}\label{cryssixth12}
Assume that neither $L_{1}$ nor $L_{2}$ are $\phi$-invariant, $L_{1}$ is not contained in any $\vphi$-invariant subspace, and $\baseth\in L_{2}$. Then $\baseo,\baset\not\in L_{2}$ and, by admissibility, we have $0=\hodgen(E\baseo)\leq\newtonn(E\baseo)=\val{\lambdao}$, $0=\hodgen(E\baset)\leq\newtonn(E\baset)=\val{\lambdat}$, $r=\hodgen(E\baseth)\leq\newtonn(E\baseth)=\val{\lambdath}$, $r=\hodgen(E(\baseo,\baset))\leq\newtonn(E(\baseo,\baset))=\val{\lambdao}+\val{\lambdat}$, $r=\hodgen(E(\baset,\baseth))\leq\newtonn(E(\baset,\baseth))=\val{\lambdat}+\val{\lambdath}$, and $r=\hodgen(E(\baseo,\baseth))\leq\newtonn(E(\baseo,\baseth))=\val{\lambdao}+\val{\lambdath}$. So, for $\val\lambdao\geq\val\lambdat\geq\val\lambdath\geq r$, $\val\lambdao+\val\lambdat+\val\lambdath=r+s$, and $s\geq2r$, we have admissible filtered $\phi$-modules. The submodules are $E\baseth$ if $\val\lambdath=r$. Hence, we have the irreducible admissible filtered $\phi$-modules if $\val\lambdao\geq\val\lambdat\geq\val\lambdath>r$ and $s\geq2r+1.$

\subsubsection{}\label{cryssixth13}
Assume that neither $L_{1}$ nor $L_{2}$ are $\phi$-invariant, $L_{1}$ is not contained in any $\phi$-invariant subspace, and $\baseo,\baset,\baseth\not\in L_{2}$. Then, by admissibility, we have $0=\hodgen(E\baseo)\leq\newtonn(E\baseo)=\val{\lambdao}$, $0=\hodgen(E\baset)\leq\newtonn(E\baset)=\val{\lambdat}$, $0=\hodgen(E\baseth)\leq\newtonn(E\baseth)=\val{\lambdath}$, $r=\hodgen(E(\baseo,\baset))\leq\newtonn(E(\baseo,\baset))=\val{\lambdao}+\val{\lambdat}$, $r=\hodgen(E(\baset,\baseth))\leq\newtonn(E(\baset,\baseth))=\val{\lambdat}+\val{\lambdath}$, and $r=\hodgen(E(\baseo,\baseth))\leq\newtonn(E(\baseo,\baseth))=\val{\lambdao}+\val{\lambdath}$. So, for $s\geq\val\lambdao\geq\val\lambdat\geq\val\lambdath\geq0$ and $\val\lambdao+\val\lambdat+\val\lambdath=r+s$, we have admissible filtered $\phi$-modules. The submodules are $E\baseth$ if $\val\lambdath=0$ and $E(\baset,\baseth)$ if $\val\lambdao=s$. Hence, we have the irreducible admissible filtered $\phi$-modules if $s>\val\lambdao\geq\val\lambdat\geq\val\lambdath>0.$

\subsection{List of the isomorphism classes with $N=0$}\label{ssec:list of N=0}

In the previous subsections, we found all of the admissible filtered $\phi$-modules of Hodge type $(0,r,s)$ for $0<r<s$. In this subsection, we classify the isomorphism classes of the admissible filtered $\phi$-modules on $D=E(\baseo,\baset,\baseth)$.

The following example arises from \ref{cryssecond1}.
\begin{exam}
A filtered $\phi$-module of Hodge type $(0,r,s)$
$$D_{cris}^{1}=D_{cris}^{1}(\lambda);$$
\begin{itemize}
\item $\fil{r} D=E(\baseo,\baseth)\mbox{ and }\fil{s} D=E\baseo$
\item $
[\phi]=
\begin{small}\left(
  \begin{array}{ccc}
    \lambda & 0  & 0 \\
    1 & \lambda & 0 \\
    0 & 0 & \lambda \\
  \end{array}
\right)
\end{small}
$ for $\lambda\in E$.
\item $\val\lambda=\frac{r+s}{3}$ and $s=2r$.
\end{itemize}
\end{exam}

\begin{prop}\label{class crys D1}
\begin{enumerate}
\item $D_{cris}^{1}$ represents admissible filtered $\phi$-modules.
\item The corresponding representations to $D_{cris}^{1}$ are decomposable with submodules $E(\baseo,\baset)$ and $E\baseth$.
\item $D_{cris}^{1}(\lambda)$ is isomorphic to $D_{cris}^{1}(\lambda')$ if and only if $\lambda=\lambda'$.
\end{enumerate}
\end{prop}

\begin{proof}
From \ref{cryssecond1}, if we let $L_{1}=E(\baseo+b\baset+c\baseth)$ and $L_{2}=(\baseo+b\baset+c\baseth,d\baset+\baseth)$, then, by change of a basis: $\baseo\mapsto\baseo+(cd-b)\baset-c\baseth$, $\baset\mapsto\baset$, and $\baseth\mapsto-d\baset+\baseth$, we get $D_{cris}^{1}(\lambda)$. Notice that this change of a basis does not change the matrix presentation of $\phi$ by Lemma \ref{typeofPsecond}. So now part (1) and (2) are clear. Part (3) is clear, since any isomorphism preserves the eigenvalues of $\phi$.
\end{proof}

The following example arises from \ref{crysthird1}.
\begin{exam}
A filtered $\phi$-module of Hodge type $(0,r,s)$
$$D_{cris}^{2}=D_{cris}^{2}(\lambda);$$
\begin{itemize}
\item $\fil{r} D=E(\baseo,\baseth)\mbox{ and }\fil{s} D=E\baseo$
\item $
[\phi]=
\begin{small}\left(
  \begin{array}{ccc}
    \lambda & 0  & 0 \\
    1 & \lambda & 0 \\
    0 & 1 & \lambda \\
  \end{array}
\right)
\end{small}
$ for $\lambda\in E$.
\item $\val\lambda=\frac{r+s}{3}$ and $s\geq2r$.
\end{itemize}
\end{exam}

\begin{prop}\label{class crys D2}
\begin{enumerate}
\item $D_{cris}^{2}$ represents admissible filtered $\phi$-modules.
\item The corresponding representations to $D_{cris}^{2}$ are
      \begin{itemize}
      \item non-split reducible with submodule $E\baseth$ if $s=2r$ and
      \item irreducible if $s>2r$.
      \end{itemize}
\item $D_{cris}^{2}(\lambda)$ is isomorphic to $D_{cris}^{2}(\lambda')$ if and only if $\lambda=\lambda'$.
\end{enumerate}
\end{prop}

\begin{proof}
From \ref{crysthird1}, if we let $L_{1}=E(\baseo+a\baset+b\baseth)$ and $L_{2}=(\baseo+a\baset+b\baseth,\baseth)$, then, by change of a basis: $\baseo\mapsto\baseo-a\baset+(a^{2}-b)\baseth$, $\baset\mapsto\baset-a\baseth$, and $\baseth\mapsto\baseth$, we get $D_{cris}^{2}(\lambda)$. Notice that this change of a basis does not change the matrix presentation of $\phi$ by Lemma \ref{typeofPthird}. So now part (1) and (2) are clear. Part (3) is clear, since any isomorphism reserves the eigenvalues of $\phi$.
\end{proof}

The following example arises from \ref{crysthird2}.
\begin{exam}
A filtered $\phi$-module of Hodge type $(0,r,s)$
$$D_{cris}^{3}=D_{cris}^{3}(\lambda);$$
\begin{itemize}
\item $\fil{r} D=E(\baseo,\baset)\mbox{ and }\fil{s} D=E\baset$
\item $
[\phi]=
\begin{small}\left(
  \begin{array}{ccc}
    \lambda & 0  & 0 \\
    1 & \lambda & 0 \\
    0 & 1 & \lambda \\
  \end{array}
\right)
\end{small}
$ for $\lambda\in E$.
\item $\val\lambda=\frac{r+s}{3}$ and $s\leq2r$.
\end{itemize}
\end{exam}

\begin{prop}\label{class crys D3}
\begin{enumerate}
\item $D_{cris}^{3}$ represents admissible filtered $\phi$-modules.
\item The corresponding representations to $D_{cris}^{3}$ are
      \begin{itemize}
      \item non-split reducible with submodule $E(\baset,\baseth)$ if $s=2r$ and
      \item irreducible if $s<2r$.
      \end{itemize}
\item $D_{cris}^{3}(\lambda)$ is isomorphic to $D_{cris}^{3}(\lambda')$ if and only if $\lambda=\lambda'$.
\end{enumerate}
\end{prop}

\begin{proof}
From \ref{crysthird2}, if we let $L_{1}=E(\baset+a\baseth)$ and $L_{2}=(\baset+a\baseth,\baseo+b\baseth)$, then, by change of a basis: $\baseo\mapsto\baseo-a\baset-b\baseth$, $\baset\mapsto\baset-a\baseth$, and $\baseth\mapsto\baseth$, we get $D_{cris}^{3}(\lambda)$. Notice that this change of a basis does not change the matrix presentation of $\phi$ by Lemma \ref{typeofPthird}. So now part (1) and (2) are clear. Part (3) is clear, since any isomorphism preserves the eigenvalues of $\phi$.
\end{proof}

The following example arises from \ref{crysthird3}.
\begin{exam}
A filtered $\phi$-module of Hodge type $(0,r,s)$
$$D_{cris}^{4}=D_{cris}^{4}(\lambda,\mfl);$$
\begin{itemize}
\item $\fil{r} D=E(\baseo,\baset+\mfl\baseth)\mbox{ and }\fil{s} D=E\baseo$
\item $
[\phi]=
\begin{small}\left(
  \begin{array}{ccc}
    \lambda & 0  & 0 \\
    1 & \lambda & 0 \\
    0 & 1 & \lambda \\
  \end{array}
\right)
\end{small}
$ for $\lambda\in E$.
\item $\mfl\in E$ and $\val\lambda=\frac{r+s}{3}$.
\end{itemize}
\end{exam}

\begin{prop}\label{class crys D4}
\begin{enumerate}
\item $D_{cris}^{4}$ represents admissible filtered $\phi$-modules.
\item The corresponding representations to $D_{cris}^{4}$ are irreducible.
\item $D_{cris}^{4}(\lambda,\mfl)$ is isomorphic to $D_{cris}^{4}(\lambda',\mfl')$ if and only if $\lambda=\lambda'$ and $\mfl=\mfl'$.
\end{enumerate}
\end{prop}

\begin{proof}
From \ref{crysthird3}, if we let $L_{1}=E(\baseo+a\baset+b\baseth)$ and $L_{2}=(\baseo+a\baset+b\baseth,\baset+c\baseth)$, then, by change of a basis: $\baseo\mapsto\baseo-a\baset+(a^{2}-b)\baseth$, $\baset\mapsto\baset-a\baseth$, and $\baseth\mapsto\baseth$, we get $D_{cris}^{4}(\lambda)$. Notice that this change of a basis does not change the matrix presentation of $\phi$ by Lemma \ref{typeofPthird}. So now part (1) and (2) are clear. Since any isomorphism preserves the filtration and the eigenvalues of $\phi$, Lemma \ref{typeofPthird} implies that an automorphism of $D_{cris}^{4}$ is a scalar multiple of the identity map, which completes part (3).
\end{proof}

The following example arises from \ref{crysfourth1}.
\begin{exam}
A filtered $\phi$-module of Hodge type $(0,r,s)$
$$D_{cris}^{5}=D_{cris}^{5}(\lambda,\lambdath);$$
\begin{itemize}
\item $\fil{r} D=E(\baseo+\baseth,\baset)\mbox{ and }\fil{s} D=E(\baseo+\baseth)$
\item $
[\phi]=
\begin{small}\left(
  \begin{array}{ccc}
    \lambda & 0  & 0 \\
    0 & \lambda & 0 \\
    0 & 0 & \lambdath \\
  \end{array}
\right)
\end{small}
$ for $\lambda\neq\lambdath\in E$.
\item $\val\lambda=r$ and $\val\lambdath=s-r$.
\end{itemize}
\end{exam}

\begin{prop}\label{class crys D5}
\begin{enumerate}
\item $D_{cris}^{5}$ represents admissible filtered $\phi$-modules.
\item The corresponding representations to $D_{cris}^{5}$ are decomposable with submodules $E\baset$ and $E(\baseo,\baseth)$.
\item $D_{cris}^{5}(\lambda,\lambdath)$ is isomorphic to $D_{cris}^{5}(\lambda',\lambdath')$ if and only if $\lambda=\lambda'$ and $\lambdath=\lambdath'$.
\end{enumerate}
\end{prop}

\begin{proof}
From \ref{crysfourth1}, if we let $L_{1}=E(a\baseo+b\baset+\baseth)$ and $L_{2}=(a\baseo+b\baset+\baseth,c\baseo+d\baset)$ with $ad-bc\neq0$, then, by change of a basis: $a\baseo+b\baset\mapsto\baseo$, $c\baseo+d\baset\mapsto\baset$, and $\baseth\mapsto\baseth$, we get $D_{cris}^{5}(\lambda,\lambdath)$. Notice that this change of a basis does not change the matrix presentation of $\phi$ by Lemma \ref{typeofPfourth}. So now the part (1) and (2) are clear. The part (3) is also clear, since any isomorphism preserves the eigenvalues of $\phi$.
\end{proof}

The following example arises from \ref{crysfifth1}.
\begin{exam}
A filtered $\phi$-module of Hodge type $(0,r,s)$
$$D_{cris}^{6}=D_{cris}^{6}(\lambda,\lambdath);$$
\begin{itemize}
\item $\fil{r} D=E(\baseo,\baseth)\mbox{ and }\fil{s} D=E\baseth$
\item $
[\phi]=
\begin{small}\left(
  \begin{array}{ccc}
    \lambda & 0  & 0 \\
    1 & \lambda & 0 \\
    0 & 0 & \lambdath \\
  \end{array}
\right)
\end{small}
$ for $\lambda\neq\lambdath\in E$.
\item $\val\lambda=\frac{r}{2}$ and $\val\lambdath=s$.
\end{itemize}
\end{exam}

\begin{prop}\label{class crys D6}
\begin{enumerate}
\item $D_{cris}^{6}$ represents admissible filtered $\phi$-modules.
\item The corresponding representations to $D_{cris}^{6}$ are decomposable with submodules $E\baseth$ and $E(\baseo,\baset)$.
\item $D_{cris}^{6}(\lambda,\lambdath)$ is isomorphic to $D_{cris}^{6}(\lambda',\lambdath')$ if and only if $\lambda=\lambda'$ and $\lambdath=\lambdath'$.
\end{enumerate}
\end{prop}

\begin{proof}
From \ref{crysfifth1}, if we let $L_{1}=E\baseth$ and $L_{2}=(\baseth,\baseo+a\baset)$, then, by change of a basis: $\baseo\mapsto\baseo-a\baset$, $\baset\mapsto\baset$, and $\baseth\mapsto\baseth$, we get $D_{cris}^{6}(\lambda,\lambdath)$. Notice that this change of a basis does not change the matrix presentation of $\phi$ by Lemma \ref{typeofPfifth}. So now the part (1) and (2) are clear. The part (3) is also clear, since any isomorphism preserves the eigenvalues of $\phi$.
\end{proof}

The following example arises from \ref{crysfifth2}.
\begin{exam}
A filtered $\phi$-module of Hodge type $(0,r,s)$
$$D_{cris}^{7}=D_{cris}^{7}(\lambda,\lambdath);$$
\begin{itemize}
\item $\fil{r} D=E(\baseo,\baset)\mbox{ and }\fil{s} D=E\baseo$
\item $
[\phi]=
\begin{small}\left(
  \begin{array}{ccc}
    \lambda & 0  & 0 \\
    1 & \lambda & 0 \\
    0 & 0 & \lambdath \\
  \end{array}
\right)
\end{small}
$ for $\lambda\neq\lambdath\in E$.
\item $\val\lambda=\frac{r+s}{2}$ and $\val\lambdath=0$.
\end{itemize}
\end{exam}

\begin{prop}\label{class crys D7}
\begin{enumerate}
\item $D_{cris}^{7}$ represents admissible filtered $\phi$-modules.
\item The corresponding representations to $D_{cris}^{7}$ are decomposable with submodules $E\baseth$ and $E(\baseo,\baset)$.
\item $D_{cris}^{7}(\lambda,\lambdath)$ is isomorphic to $D_{cris}^{7}(\lambda',\lambdath)$ if and only if $\lambda=\lambda'$ and $\lambdath=\lambdath'$.
\end{enumerate}
\end{prop}

\begin{proof}
From \ref{crysfifth2}, if we let $L_{1}=E(\baseo+a\baset)$ and $L_{2}=(\baseo+a\baset,\baset)$, then, by change of a basis: $\baseo\mapsto\baseo-a\baset$, $\baset\mapsto\baset$, and $\baseth\mapsto\baseth$, we get $D_{cris}^{7}(\lambda,\lambdath)$. Notice that this change of a basis does not change the matrix presentation of $\phi$ by Lemma \ref{typeofPfifth}. So now the part (1) and (2) are clear. The part (3) is also clear, since any isomorphism preserves the eigenvalues of $\phi$.
\end{proof}

The following example arises from \ref{crysfifth3}.
\begin{exam}
A filtered $\phi$-module of Hodge type $(0,r,s)$
$$D_{cris}^{8}=D_{cris}^{8}(\lambda,\lambdath);$$
\begin{itemize}
\item $\fil{r} D=E(\baseo,\baseth)\mbox{ and }\fil{s} D=E\baseo$
\item $
[\phi]=
\begin{small}\left(
  \begin{array}{ccc}
    \lambda & 0  & 0 \\
    1 & \lambda & 0 \\
    0 & 0 & \lambdath \\
  \end{array}
\right)
\end{small}
$ for $\lambda\neq\lambdath\in E$.
\item $\val\lambda=\frac{s}{2}$ and $\val\lambdath=r$.
\end{itemize}
\end{exam}

\begin{prop}\label{class crys D8}
\begin{enumerate}
\item $D_{cris}^{8}$ represents admissible filtered $\phi$-modules.
\item The corresponding representations to $D_{cris}^{8}$ are decomposable with submodules $E\baseth$ and $E(\baseo,\baset)$.
\item $D_{cris}^{8}(\lambda,\lambdath)$ is isomorphic to $D_{cris}^{8}(\lambda',\lambdath')$ if and only if $\lambda=\lambda'$ and $\lambdath=\lambdath'$.
\end{enumerate}
\end{prop}

\begin{proof}
From \ref{crysfifth3}, if we let $L_{1}=E(\baseo+a\baset)$ and $L_{2}=(\baseo+a\baset,\baseth)$, then, by change of a basis: $\baseo\mapsto\baseo-a\baset$, $\baset\mapsto\baset$, and $\baseth\mapsto\baseth$, we get $D_{cris}^{8}(\lambda,\lambdath)$. Notice that this change of a basis does not change the matrix presentation of $\phi$ by Lemma \ref{typeofPfifth}. So now the part (1) and (2) are clear. The part (3) is also clear, since any isomorphism preserves the eigenvalues of $\phi$.
\end{proof}

The following example arises from \ref{crysfifth4}.
\begin{exam}
A filtered $\phi$-module of Hodge type $(0,r,s)$
$$D_{cris}^{9}=D_{cris}^{9}(\lambda,\lambdath);$$
\begin{itemize}
\item $\fil{r} D=E(\baseo,\baset+\baseth)\mbox{ and }\fil{s} D=E\baseo$
\item $
[\phi]=
\begin{small}\left(
  \begin{array}{ccc}
    \lambda & 0  & 0 \\
    1 & \lambda & 0 \\
    0 & 0 & \lambdath \\
  \end{array}
\right)
\end{small}
$ for $\lambda\neq\lambdath\in E$.
\item $\frac{s}{2}\leq\val\lambda\leq\frac{r+s}{2}$ and $2\val\lambda+\val\lambdath=r+s$.
\end{itemize}
\end{exam}

\begin{prop}\label{class crys D9}
\begin{enumerate}
\item $D_{cris}^{9}$ represents admissible filtered $\phi$-modules.
\item The corresponding representations to $D_{cris}^{9}$ are
      \begin{itemize}
      \item non-split reducible with submodule $E(\baseo,\baset)$ if $\val\lambda=\frac{s}{2}$,
      \item non-split reducible with submodule $E\baseth$ if $\val\lambda=\frac{r+s}{2}$, and
      \item irreducible if $\frac{s}{2}<\val\lambda<\frac{r+s}{2}$.
      \end{itemize}
\item $D_{cris}^{9}(\lambda,\lambdath)$ is isomorphic to $D_{cris}^{9}(\lambda',\lambdath')$ if and only if $\lambda=\lambda'$ and $\lambdath=\lambdath'$.
\end{enumerate}
\end{prop}

\begin{proof}
From \ref{crysfifth4}, if we let $L_{1}=E(\baseo+a\baset)$ and $L_{2}=(\baseo+a\baset,\baset+c\baseth)$ with $c\neq0$, then, by change of a basis: $\baseo\mapsto\baseo-a\baset$, $\baset\mapsto\baset$, and $c\baseth\mapsto\baseth$, we get $D_{cris}^{9}(\lambda,\lambdath)$. Notice that this change of a basis does not change the matrix presentation of $\phi$ by Lemma \ref{typeofPfifth}. So now the part (1) and (2) are clear. The part (3) is also clear, since any isomorphism preserves the eigenvalues of $\phi$.
\end{proof}

The following example arises from \ref{crysfifth5}.
\begin{exam}
A filtered $\phi$-module of Hodge type $(0,r,s)$
$$D_{cris}^{10}=D_{cris}^{10}(\lambda,\lambdath);$$
\begin{itemize}
\item $\fil{r} D=E(\baset+\baseth,\baseo)\mbox{ and }\fil{s} D=E(\baset+\baseth)$
\item $
[\phi]=
\begin{small}\left(
  \begin{array}{ccc}
    \lambda & 0  & 0 \\
    1 & \lambda & 0 \\
    0 & 0 & \lambdath \\
  \end{array}
\right)
\end{small}
$ for $\lambda\neq\lambdath\in E$.
\item $\frac{r}{2}\leq\val\lambda\leq r$ and $2\val\lambda+\val\lambdath=r+s$.
\end{itemize}
\end{exam}

\begin{prop}\label{class crys D10}
\begin{enumerate}
\item $D_{cris}^{10}$ represents admissible filtered $\phi$-modules.
\item The corresponding representations to $D_{cris}^{10}$ are
      \begin{itemize}
      \item non-split reducible with submodule $E(\baseo,\baset)$ if $\val\lambda=\frac{r}{2}$,
      \item non-split reducible with submodule $E(\baset,\baseth)$ if $\val\lambda=r$, and
      \item irreducible if $\frac{r}{2}<\val\lambda<r$.
      \end{itemize}
\item $D_{cris}^{10}(\lambda,\lambdath)$ is isomorphic to $D_{cris}^{10}(\lambda',\lambdath')$ if and only if $\lambda=\lambda'$ and $\lambdath=\lambdath'$.
\end{enumerate}
\end{prop}

\begin{proof}
From \ref{crysfifth5}, if we let $L_{1}=E(\baset+a\baseth)$ and $L_{2}=(\baset+a\baseth,\baseo+b\baseth)$ with $a\neq0$, then, by change of a basis: $\baseo\mapsto\baseo+\frac{b}{a}\baset$, $\baset\mapsto\baset$, and $a\baseth\mapsto\baseth$, we get $D_{cris}^{10}(\lambda,\lambdath)$. Notice that this change of a basis does not change the matrix presentation of $\phi$ by Lemma \ref{typeofPfifth}. So now the part (1) and (2) are clear. The part (3) is also clear, since any isomorphism preserves the eigenvalues of $\phi$.
\end{proof}

The following example arises from \ref{crysfifth6}.
\begin{exam}
A filtered $\phi$-module of Hodge type $(0,r,s)$
$$D_{cris}^{11}=D_{cris}^{11}(\lambda,\lambdath);$$
\begin{itemize}
\item $\fil{r} D=E(\baseo+\baseth,\baset)\mbox{ and }\fil{s} D=E(\baseo+\baseth)$
\item $
[\phi]=
\begin{small}\left(
  \begin{array}{ccc}
    \lambda & 0  & 0 \\
    1 & \lambda & 0 \\
    0 & 0 & \lambdath \\
  \end{array}
\right)
\end{small}
$ for $\lambda\neq\lambdath\in E$.
\item $r\leq\val\lambda\leq \frac{r+s}{2}$ and $2\val\lambda+\val\lambdath=r+s$.
\end{itemize}
\end{exam}

\begin{prop}\label{class crys D11}
\begin{enumerate}
\item $D_{cris}^{11}$ represents admissible filtered $\phi$-modules.
\item The corresponding representations to $D_{cris}^{11}$ are
      \begin{itemize}
      \item non-split reducible with submodule $E\baset$ if $\val\lambda=r$,
      \item non-split reducible with submodule $E\baseth$ if $\val\lambda=\frac{r+s}{2}$, and
      \item irreducible if $r<\val\lambda<\frac{r+s}{2}$.
      \end{itemize}
\item $D_{cris}^{11}(\lambda,\lambdath)$ is isomorphic to $D_{cris}^{11}(\lambda',\lambdath')$ if and only if $\lambda=\lambda'$ and $\lambdath=\lambdath'$.
\end{enumerate}
\end{prop}

\begin{proof}
From \ref{crysfifth6}, if we let $L_{1}=E(\baseo+a\baset+b\baseth)$ and $L_{2}=(\baseo+a\baset+b\baseth,\baset)$ with $b\neq0$, then, by change of a basis: $\baseo\mapsto\baseo-a\baset$, $\baset\mapsto\baset$, and $b\baseth\mapsto\baseth$, we get $D_{cris}^{11}(\lambda,\lambdath)$. Notice that this change of a basis does not change the matrix presentation of $\phi$ by Lemma \ref{typeofPfifth}. So now the part (1) and (2) are clear. The part (3) is also clear, since any isomorphism preserves the eigenvalues of $\phi$.
\end{proof}

The following example arises from \ref{crysfifth7}.
\begin{exam}
A filtered $\phi$-module of Hodge type $(0,r,s)$
$$D_{cris}^{12}=D_{cris}^{12}(\lambda,\lambdath);$$
\begin{itemize}
\item $\fil{r} D=E(\baseo,\baseth)\mbox{ and }\fil{s} D=E(\baseo+\baseth)$
\item $
[\phi]=
\begin{small}\left(
  \begin{array}{ccc}
    \lambda & 0  & 0 \\
    1 & \lambda & 0 \\
    0 & 0 & \lambdath \\
  \end{array}
\right)
\end{small}
$ for $\lambda\neq\lambdath\in E$.
\item $\frac{r}{2}\leq\val\lambda\leq \frac{s}{2}$ and $2\val\lambda+\val\lambdath=r+s$.
\end{itemize}
\end{exam}

\begin{prop}\label{class crys D12}
\begin{enumerate}
\item $D_{cris}^{12}$ represents admissible filtered $\phi$-modules.
\item The corresponding representations to $D_{cris}^{12}$ are
      \begin{itemize}
      \item non-split reducible with submodule $E(\baseo,\baset)$ if $\val\lambda=\frac{r}{2}$,
      \item non-split reducible with submodule $E\baseth$ if $\val\lambda=\frac{s}{2}$, and
      \item irreducible if $\frac{r}{2}<\val\lambda<\frac{s}{2}$.
      \end{itemize}
\item $D_{cris}^{12}(\lambda,\lambdath)$ is isomorphic to $D_{cris}^{12}(\lambda',\lambdath')$ if and only if $\lambda=\lambda'$ and $\lambdath=\lambdath'$.
\end{enumerate}
\end{prop}

\begin{proof}
From \ref{crysfifth7}, if we let $L_{1}=E(\baseo+a\baset+b\baseth)$ and $L_{2}=(\baseo+a\baset+b\baseth,\baseth)$ with $b\neq0$, then, by change of a basis: $\baseo\mapsto\baseo-a\baset$, $\baset\mapsto\baset$, and $b\baseth\mapsto\baseth$, we get $D_{cris}^{12}(\lambda,\lambdath)$. Notice that this change of a basis does not change the matrix presentation of $\phi$ by Lemma \ref{typeofPfifth}. So now the part (1) and (2) are clear. The part (3) is also clear, since any isomorphism preserves the eigenvalues of $\phi$.
\end{proof}

The following example arises from \ref{crysfifth8}.
\begin{exam}
A filtered $\phi$-module of Hodge type $(0,r,s)$
$$D_{cris}^{13}=D_{cris}^{13}(\lambda,\lambdath,\mfl);$$
\begin{itemize}
\item $\fil{r} D=E(\baseo+\baseth,\baset+\mfl\baseth)\mbox{ and }\fil{s} D=E(\baseo+\baseth)$
\item $
[\phi]=
\begin{small}\left(
  \begin{array}{ccc}
    \lambda & 0  & 0 \\
    1 & \lambda & 0 \\
    0 & 0 & \lambdath \\
  \end{array}
\right)
\end{small}
$ for $\lambda\neq\lambdath\in E$.
\item $\mfl\in E^{\times}$, $\frac{r}{2}\leq\val\lambda\leq \frac{r+s}{2}$, and $2\val\lambda+\val\lambdath=r+s$.
\end{itemize}
\end{exam}

\begin{prop}\label{class crys D13}
\begin{enumerate}
\item $D_{cris}^{13}$ represents admissible filtered $\phi$-modules.
\item The corresponding representations to $D_{cris}^{13}$ are
      \begin{itemize}
      \item non-split reducible with submodule $E(\baseo,\baset)$ if $\val\lambda=\frac{r}{2}$,
      \item non-split reducible with submodule $E\baseth$ if $\val\lambda=\frac{r+s}{2}$, and
      \item irreducible if $\frac{r}{2}<\val\lambda<\frac{r+s}{2}$.
      \end{itemize}
\item $D_{cris}^{13}(\lambda,\lambdath,\mfl)$ is isomorphic to $D_{cris}^{13}(\lambda',\lambdath',\mfl')$ if and only if $\lambda=\lambda'$, $\lambdath=\lambdath'$, and $\mfl=\mfl'$.
\end{enumerate}
\end{prop}

\begin{proof}
From \ref{crysfifth8}, if we let $L_{1}=E(\baseo+a\baset+b\baseth)$ and $L_{2}=(\baseo+a\baset+b\baseth,\baset+c\baseth)$ with $bc\neq0$, then, by change of a basis: $\baseo\mapsto\baseo-a\baset$, $\baset\mapsto\baset$, and $b\baseth\mapsto\baseth$, we get $D_{cris}^{13}(\lambda,\lambdath)$. Notice that this change of a basis does not change the matrix presentation of $\phi$ by Lemma \ref{typeofPfifth}. So now the part (1) and (2) are clear. Since an isomorphism preserves the filtration and the eigenvalues of $\phi$, Lemma \ref{typeofPfifth} implies that such an isomorphism should be a scalar multiple of the identity map, which completes the part (3).
\end{proof}

The following example arises from \ref{cryssixth1}.
\begin{exam}
A filtered $\phi$-module of Hodge type $(0,r,s)$
$$D_{cris}^{14}=D_{cris}^{14}(\lambdao,\lambdat,\lambdath);$$
\begin{itemize}
\item $\fil{r} D=E(\baseo,\baset)\mbox{ and }\fil{s} D=E\baseo$
\item $
[\phi]=
\begin{small}\left(
  \begin{array}{ccc}
    \lambdao & 0  & 0 \\
    0 & \lambdat & 0 \\
    0 & 0 & \lambdath \\
  \end{array}
\right)
\end{small}
$ for distinct $\lambda_{i}$'s in $E$.
\item $\val\lambdao=s$, $\val\lambdat=r$, and $\val\lambdath=0$.
\end{itemize}
\end{exam}

\begin{prop}\label{class crys D14}
\begin{enumerate}
\item $D_{cris}^{14}$ represents admissible filtered $\phi$-modules.
\item Every $\phi$-invariant subspace is a submodule of $D_{cris}^{14}$ and so the corresponding representations to $D_{cris}^{14}$ are decomposable.
\item $D_{cris}^{14}(\lambdao,\lambdat,\lambdath)$ is isomorphic to $D_{cris}^{14}(\lambdao',\lambdat',\lambdath')$ if and only if $\lambdao=\lambdao'$, $\lambdat=\lambdat'$, and $\lambdath=\lambdath'$
\end{enumerate}
\end{prop}

\begin{proof}
From \ref{cryssixth1}, the part (1) and (2) are clear. Since every morphism should preserve the filtration, Lemma \ref{typeofPsixth} implies that the only isomorphisms are scalar multiple of the identity map, which completes the part (3).
\end{proof}

The following example arises from \ref{cryssixth2}.
\begin{exam}
A filtered $\phi$-module of Hodge type $(0,r,s)$
$$D_{cris}^{15}=D_{cris}^{15}(\lambdao,\lambdat,\lambdath);$$
\begin{itemize}
\item $\fil{r} D=E(\baseo,\baset+\baseth)\mbox{ and }\fil{s} D=E\baseo$
\item $
[\phi]=
\begin{small}\left(
  \begin{array}{ccc}
    \lambdao & 0  & 0 \\
    0 & \lambdat & 0 \\
    0 & 0 & \lambdath \\
  \end{array}
\right)
\end{small}
$ for distinct $\lambda_{i}$'s in $E$.
\item $s=\val\lambdao\geq\val\lambdat\geq\val\lambdath\geq0$ and $\val\lambdat+\val\lambdath=r$.
\end{itemize}
\end{exam}

\begin{prop}\label{class crys D15}
\begin{enumerate}
\item $D_{cris}^{15}$ represents admissible filtered $\phi$-modules.
\item The corresponding representations to $D_{cris}^{15}$ are decomposable with submodules $E\baseo$ and $E(\baset,\baseth)$; moreover, if $\val\lambdath=0$, then $E\baseth$ and $E(\baseo,\baseth)$ are also submodules.
\item $D_{cris}^{15}(\lambdao,\lambdat,\lambdath)$ is isomorphic to $D_{cris}^{15}(\lambdao',\lambdat',\lambdath')$ if and only if either
$$
\left\{
  \begin{array}{ll}
    \lambdao=\lambdao',\lambdat=\lambdat', \lambdath=\lambdath', & \hbox{or} \\
    \lambdao=\lambdao', \lambdat=\lambdath', \lambdath=\lambdat'. & \hbox{}
  \end{array}
\right.
$$
\end{enumerate}
\end{prop}

\begin{proof}
From \ref{cryssixth2}, if we let $L_{1}=E\baseo$ and $L_{2}=E(\baseo,\baset+a\baseth)$ with $a\neq0$, then, by change of a basis: $\baseo\mapsto\baseo$, $\baset\mapsto\baset$, and $a\baseth\mapsto\baseth$, we get $D_{cris}^{15}$. So now the part (1) and (2) are clear. Since every isomorphism preserves the filtration, if there is an isomorphism then it should fix $\baseo$ but it can either fix or swap $\baset$ and $\baseth$ by lemma \ref{typeofPsixth}. So we get the part (3).
\end{proof}

The following example arises from \ref{cryssixth3}.
\begin{exam}
A filtered $\phi$-module of Hodge type $(0,r,s)$
$$D_{cris}^{16}=D_{cris}^{16}(\lambdao,\lambdat,\lambdath);$$
\begin{itemize}
\item $\fil{r} D=E(\baseo,\baset)\mbox{ and }\fil{s} D=E(\baseo+\baset)$
\item $
[\phi]=
\begin{small}\left(
  \begin{array}{ccc}
    \lambdao & 0  & 0 \\
    0 & \lambdat & 0 \\
    0 & 0 & \lambdath \\
  \end{array}
\right)
\end{small}
$ for distinct $\lambda_{i}$'s in $E$.
\item $\val\lambdao\geq\val\lambdat\geq r, \val\lambdath=0$ and $\val\lambdao+\val\lambdat=r+s$.
\end{itemize}
\end{exam}

\begin{prop}\label{class crys D16}
\begin{enumerate}
\item $D_{cris}^{16}$ represents admissible filtered $\phi$-modules.
\item The corresponding representations to $D_{cris}^{16}$ are decomposable with submodules $E\baseth$ and $E(\baseo,\baset)$; moreover, if $\val\lambdat=r$, then $E\baset$ and $E(\baset,\baseth)$ are also submodules.
\item $D_{cris}^{16}(\lambdao,\lambdat,\lambdath)$ is isomorphic to $D_{cris}^{16}(\lambdao',\lambdat',\lambdath')$ if and only if either
$$
\left\{
  \begin{array}{ll}
    \lambdao=\lambdao',\lambdat=\lambdat', \lambdath=\lambdath', & \hbox{or} \\
    \lambdao=\lambdat',\lambdat=\lambdao', \lambdath=\lambdath'. & \hbox{}
  \end{array}
\right.
$$
\end{enumerate}
\end{prop}

\begin{proof}
From \ref{cryssixth3}, if we let $L_{1}=E(\baseo+a\baset)$ and $L_{2}=E(\baseo,\baset)$ with $a\neq0$, then, by change of a basis: $\baseo\mapsto\baseo$, $a\baset\mapsto\baset$, and $\baseth\mapsto\baseth$, we get $D_{cris}^{16}$. So now the part (1) and (2) are clear. Since every isomorphism preserves the filtration, if there is an isomorphism then it should fix $\baseth$ but it can either fix or swap $\baseo$ and $\baset$ by lemma \ref{typeofPsixth}. So we get the part (3).
\end{proof}

The following example arises from \ref{cryssixth4}.
\begin{exam}
A filtered $\phi$-module of Hodge type $(0,r,s)$
$$D_{cris}^{17}=D_{cris}^{17}(\lambdao,\lambdat,\lambdath);$$
\begin{itemize}
\item $\fil{r} D=E(\baseo+\baset,\baseth)\mbox{ and }\fil{s} D=E(\baseo+\baset)$
\item $
[\phi]=
\begin{small}\left(
  \begin{array}{ccc}
    \lambdao & 0  & 0 \\
    0 & \lambdat & 0 \\
    0 & 0 & \lambdath \\
  \end{array}
\right)
\end{small}
$ for distinct $\lambda_{i}$'s in $E$.
\item $\val\lambdao\geq\val\lambdat\geq\val\lambdath=r$, $\val\lambdao+\val\lambdat=s$, and $s\geq 2r$.
\end{itemize}
\end{exam}

\begin{prop}\label{class crys D17}
\begin{enumerate}
\item $D_{cris}^{17}$ represents admissible filtered $\phi$-modules.
\item The corresponding representations to $D_{cris}^{17}$ are decomposable with submodules $E\baseth$ and $E(\baseo,\baset)$.
\item $D_{cris}^{17}(\lambdao,\lambdat,\lambdath)$ is isomorphic to $D_{cris}^{17}(\lambdao',\lambdat',\lambdath')$ if and only if either
$$
\left\{
  \begin{array}{ll}
    \lambdao=\lambdao',\lambdat=\lambdat', \lambdath=\lambdath', & \hbox{or} \\
    \lambdao=\lambdat',\lambdat=\lambdao', \lambdath=\lambdath'. & \hbox{}
  \end{array}
\right.
$$
\end{enumerate}
\end{prop}

\begin{proof}
From \ref{cryssixth4}, if we let $L_{1}=E(\baseo+a\baset)$ and $L_{2}=E(\baseo+a\baset,\baseth)$ with $a\neq0$, then, by change of a basis: $\baseo\mapsto\baseo$, $a\baset\mapsto\baset$, and $\baseth\mapsto\baseth$, we get $D_{cris}^{17}$. So now the part (1) and (2) are clear. Since every isomorphism preserves the filtration, if there is an isomorphism then it should fix $\baseth$ but it can either fix or swap $\baseo$ and $\baset$ by lemma \ref{typeofPsixth}. So we get the part (3).
\end{proof}

The following example arises from \ref{cryssixth5}.
\begin{exam}
A filtered $\phi$-module of Hodge type $(0,r,s)$
$$D_{cris}^{18}=D_{cris}^{18}(\lambdao,\lambdat,\lambdath);$$
\begin{itemize}
\item $\fil{r} D=E(\baseo+\baset,\baset+\baseth)\mbox{ and }\fil{s} D=E(\baseo+\baset)$
\item $
[\phi]=
\begin{small}\left(
  \begin{array}{ccc}
    \lambdao & 0  & 0 \\
    0 & \lambdat & 0 \\
    0 & 0 & \lambdath \\
  \end{array}
\right)
\end{small}
$ for distinct $\lambda_{i}$'s in $E$.
\item $s\geq\val\lambdao\geq\val\lambdat\geq\val\lambdath\geq0$,
$\val\lambdath\leq r$, and $\val\lambdao+\val\lambdat+\val\lambdath=r+s$.
\end{itemize}
\end{exam}

\begin{prop}\label{class crys D18}
\begin{enumerate}
\item $D_{cris}^{18}$ represents admissible filtered $\phi$-modules.
\item The corresponding representations to $D_{cris}^{18}$ are
    \begin{itemize}
    \item non-split reducible with submodule $E\baseth$ if $\val\lambdath=0$,
    \item non-split reducible with submodule $E(\baseo,\baset)$ if $\val\lambdath=r$,
    \item non-split reducible with submodule $E(\baset,\baseth)$ if $\val\lambdao=s$, and
    \item irreducible if $s>\val\lambdao\geq\val\lambdat\geq\val\lambdath>0$ and $\val\lambdath<r$.
    \end{itemize}
\item $D_{cris}^{18}(\lambdao,\lambdat,\lambdath)$ is isomorphic to $D_{cris}^{18}(\lambdao',\lambdat',\lambdath')$ if and only if either
$$
\left\{
  \begin{array}{ll}
    \lambdao=\lambdao',\lambdat=\lambdat', \lambdath=\lambdath', & \hbox{or} \\
    \lambdao=\lambdat',\lambdat=\lambdao', \lambdath=\lambdath'. & \hbox{}
  \end{array}
\right.
$$
\end{enumerate}
\end{prop}

\begin{proof}
From \ref{cryssixth5}, if we let $L_{1}=E(\baseo+a\baset)$ and $L_{2}=E(\baseo+a\baset,\baset+b\baseth)$ with $ab\neq0$, then, by change of a basis: $\baseo\mapsto\baseo$, $a\baset\mapsto\baset$, and $ab\baseth\mapsto\baseth$, we get $D_{cris}^{18}$. So now the part (1) and (2) are clear. Since every isomorphism preserves the filtration, if there is an isomorphism then it can either fix or swap $\baseo$ and $\baset$ to preserve $\fils D$, by lemma \ref{typeofPsixth}. If it swap those two, by sending $\baseth$ to $-\baseth$ it can also preserve $\filr D$. So we get the part (3).
\end{proof}

The following example arises from \ref{cryssixth6}.
\begin{exam}
A filtered $\phi$-module of Hodge type $(0,r,s)$
$$D_{cris}^{19}=D_{cris}^{19}(\lambdao,\lambdat,\lambdath);$$
\begin{itemize}
\item $\fil{r} D=E(\baset+\baseth,\baseo)\mbox{ and }\fil{s} D=E(\baset+\baseth)$
\item $
[\phi]=
\begin{small}\left(
  \begin{array}{ccc}
    \lambdao & 0  & 0 \\
    0 & \lambdat & 0 \\
    0 & 0 & \lambdath \\
  \end{array}
\right)
\end{small}
$ for distinct $\lambda_{i}$'s in $E$.
\item $r=\val\lambdao\geq\val\lambdat\geq\val\lambdath$, $\val\lambdat+\val\lambdath=s$, and $s\leq 2r$.
\end{itemize}
\end{exam}

\begin{prop}\label{class crys D19}
\begin{enumerate}
\item $D_{cris}^{19}$ represents admissible filtered $\phi$-modules.
\item The corresponding representations to $D_{cris}^{19}$ are decomposable with submodules $E\baseo$ and $E(\baset,\baseth)$.
\item $D_{cris}^{19}(\lambdao,\lambdat,\lambdath)$ is isomorphic to $D_{cris}^{19}(\lambdao',\lambdat',\lambdath')$ if and only if either
$$
\left\{
  \begin{array}{ll}
    \lambdao=\lambdao',\lambdat=\lambdat', \lambdath=\lambdath', & \hbox{or} \\
    \lambdao=\lambdao',\lambdat=\lambdath', \lambdath=\lambdat'. & \hbox{}
  \end{array}
\right.
$$
\end{enumerate}
\end{prop}

\begin{proof}
From \ref{cryssixth6}, if we let $L_{1}=E(\baset+a\baseth)$ and $L_{2}=E(\baset+a\baseth,\baseo)$ with $a\neq0$, then, by change of a basis: $\baseo\mapsto\baseo$, $\baset\mapsto\baset$, and $a\baseth\mapsto\baseth$, we get $D_{cris}^{19}$. So now the part (1) and (2) are clear. Since every isomorphism preserves the filtration, if there is an isomorphism then it can either fix or swap $\baset$ and $\baseth$, by lemma \ref{typeofPsixth}. So we get the part (3).
\end{proof}

The following example arises from \ref{cryssixth7}.
\begin{exam}
A filtered $\phi$-module of Hodge type $(0,r,s)$
$$D_{cris}^{20}=D_{cris}^{20}(\lambdao,\lambdat,\lambdath);$$
\begin{itemize}
\item $\fil{r} D=E(\baset+\baseth,\baseo+\baseth)\mbox{ and }\fil{s} D=E(\baset+\baseth)$
\item $
[\phi]=
\begin{small}\left(
  \begin{array}{ccc}
    \lambdao & 0  & 0 \\
    0 & \lambdat & 0 \\
    0 & 0 & \lambdath \\
  \end{array}
\right)
\end{small}
$ for distinct $\lambda_{i}$'s in $E$.
\item $r\geq\val\lambdao\geq\val\lambdat\geq\val\lambdath$, $\val\lambdao+\val\lambdat+\val\lambdath=r+s$, and $s\leq 2r$.
\end{itemize}
\end{exam}

\begin{prop}\label{class crys D20}
\begin{enumerate}
\item $D_{cris}^{20}$ represents admissible filtered $\phi$-modules.
\item The corresponding representations to $D_{cris}^{20}$ are
      \begin{itemize}
      \item non-split reducible with submodules $E(\baset,\baseth)$ if $\val\lambdao=r$ and
      \item irreducible if $r>\val\lambdao\geq\val\lambdat\geq\val\lambdath$ and so $s\leq2r-1$.
      \end{itemize}
\item $D_{cris}^{20}(\lambdao,\lambdat,\lambdath)$ is isomorphic to $D_{cris}^{20}(\lambdao',\lambdat',\lambdath')$ if and only if either
$$
\left\{
  \begin{array}{ll}
    \lambdao=\lambdao',\lambdat=\lambdat', \lambdath=\lambdath', & \hbox{or} \\
    \lambdao=\lambdao',\lambdat=\lambdath', \lambdath=\lambdat'. & \hbox{}
  \end{array}
\right.
$$
\end{enumerate}
\end{prop}

\begin{proof}
From \ref{cryssixth7}, if we let $L_{1}=E(\baset+a\baseth)$ and $L_{2}=E(\baset+a\baseth,b\baseo+\baseth)$ with $ab\neq0$, then, by change of a basis: $\baseo\mapsto\baseo$, $\baset\mapsto ab\baset$, and $\baseth\mapsto b\baseth$, we get $D_{cris}^{20}$. So now the part (1) and (2) are clear. Since every isomorphism preserves the filtration, if there is an isomorphism then it can either fix or swap $\baset$ and $\baseth$ to preserve $\fils D$, by lemma \ref{typeofPsixth}. If it swap $\baset$ and $\baseth$, it should send $\baseo$ to $-\baseo$ to preserve $\filr D$. So we get the part (3).
\end{proof}

The following example arises from \ref{cryssixth8}.
\begin{exam}
A filtered $\phi$-module of Hodge type $(0,r,s)$
$$D_{cris}^{21}=D_{cris}^{21}(\lambdao,\lambdat,\lambdath);$$
\begin{itemize}
\item $\fil{r} D=E(\baseo+\baseth,\baset)\mbox{ and }\fil{s} D=E(\baseo+\baseth)$
\item $
[\phi]=
\begin{small}\left(
  \begin{array}{ccc}
    \lambdao & 0  & 0 \\
    0 & \lambdat & 0 \\
    0 & 0 & \lambdath \\
  \end{array}
\right)
\end{small}
$ for distinct $\lambda_{i}$'s in $E$.
\item $\val\lambdao\geq\val\lambdat=r\geq\val\lambdath\geq0$ and $\val\lambdao+\val\lambdath=s$.
\end{itemize}
\end{exam}

\begin{prop}\label{class crys D21}
\begin{enumerate}
\item $D_{cris}^{21}$ represents admissible filtered $\phi$-modules.
\item The corresponding representations to $D_{cris}^{21}$ are decomposable with submodules $E\baset$ and $E(\baseo,\baseth)$; moreover, if $\val\lambdao=s$ then $E\baseth$ and $E(\baset,\baseth)$ are also submodules.
\item $D_{cris}^{21}(\lambdao,\lambdat,\lambdath)$ is isomorphic to $D_{cris}^{21}(\lambdao',\lambdat',\lambdath')$ if and only if either
$$
\left\{
  \begin{array}{ll}
    \lambdao=\lambdao',\lambdat=\lambdat', \lambdath=\lambdath', & \hbox{or} \\
    \lambdao=\lambdath',\lambdat=\lambdat', \lambdath=\lambdao'. & \hbox{}
  \end{array}
\right.
$$
\end{enumerate}
\end{prop}

\begin{proof}
From \ref{cryssixth8}, if we let $L_{1}=E(\baseo+a\baseth)$ and $L_{2}=E(\baseo+a\baseth,\baset)$ with $a\neq0$, then, by change of a basis: $\baseo\mapsto\baseo$, $\baset\mapsto \baset$, and $a\baseth\mapsto \baseth$, we get $D_{cris}^{21}$. So now the part (1) and (2) are clear. Since every isomorphism preserves the filtration, if there is an isomorphism then it can either fix or swap $\baseo$ and $\baseth$, by lemma \ref{typeofPsixth}. So we get the part (3).
\end{proof}

The following example arises from \ref{cryssixth9}.
\begin{exam}
A filtered $\phi$-module of Hodge type $(0,r,s)$
$$D_{cris}^{22}=D_{cris}^{22}(\lambdao,\lambdat,\lambdath);$$
\begin{itemize}
\item $\fil{r} D=E(\baseo+\baseth,\baset+\baseth)\mbox{ and }\fil{s} D=E(\baseo+\baseth)$
\item $
[\phi]=
\begin{small}\left(
  \begin{array}{ccc}
    \lambdao & 0  & 0 \\
    0 & \lambdat & 0 \\
    0 & 0 & \lambdath \\
  \end{array}
\right)
\end{small}
$ for distinct $\lambda_{i}$'s in $E$.
\item $s\geq\val\lambdao\geq\val\lambdat\geq\val\lambdath$, $\val\lambdat\leq r$, and $\val\lambdao+\val\lambdat+\val\lambdath=r+s$.
\end{itemize}
\end{exam}

\begin{prop}\label{class crys D22}
\begin{enumerate}
\item $D_{cris}^{22}$ represents admissible filtered $\phi$-modules.
\item The corresponding representations to $D_{cris}^{22}$ are
      \begin{itemize}
      \item non-split reducible with submodule $E\baseth$ if $\val\lambdath=0$,
      \item non-split reducible with submodule $E(\baseo,\baseth)$ if $\val\lambdat=r$,
      \item non-split reducible with submodule $E(\lambdat,\lambdath)$ if $\val\lambdao=s$, and
      \item irreducible if $s>\val\lambdao\geq\val\lambdat\geq\val\lambdath$ and $\val\lambdat<r$.
      \end{itemize}
\item $D_{cris}^{22}(\lambdao,\lambdat,\lambdath)$ is isomorphic to $D_{cris}^{22}(\lambdao',\lambdat',\lambdath')$ if and only if either
$$
\left\{
  \begin{array}{ll}
    \lambdao=\lambdao',\lambdat=\lambdat', \lambdath=\lambdath', & \hbox{or} \\
    \lambdao=\lambdath',\lambdat=\lambdat', \lambdath=\lambdao'. & \hbox{}
  \end{array}
\right.
$$
\end{enumerate}
\end{prop}

\begin{proof}
From \ref{cryssixth9}, if we let $L_{1}=E(\baseo+a\baseth)$ and $L_{2}=E(\baseo+a\baseth,\baset+b\baseth)$ with $ab\neq0$, then, by change of a basis: $\baseo\mapsto a\baseo$, $\baset\mapsto b\baset$, and $\baseth\mapsto \baseth$, we get $D_{cris}^{22}$. So now the part (1) and (2) are clear. Since every isomorphism preserves the filtration, if there is an isomorphism then it can either fix or swap $\baset$ and $\baseth$ to preserve $\fils D$, by lemma \ref{typeofPsixth}. If it swap $\baseo$ and $\baseth$, it should send $\baset$ to $-\baset$ to preserve $\filr D$. So we get the part (3).
\end{proof}

The following example arises from \ref{cryssixth10}.
\begin{exam}
A filtered $\phi$-module of Hodge type $(0,r,s)$
$$D_{cris}^{23}=D_{cris}^{23}(\lambdao,\lambdat,\lambdath);$$
\begin{itemize}
\item $\fil{r} D=E(\baseo+\baset+\baseth,\baseo)\mbox{ and }\fil{s} D=E(\baseo+\baset+\baseth)$
\item $
[\phi]=
\begin{small}\left(
  \begin{array}{ccc}
    \lambdao & 0  & 0 \\
    0 & \lambdat & 0 \\
    0 & 0 & \lambdath \\
  \end{array}
\right)
\end{small}
$ for distinct $\lambda_{i}$'s in $E$.
\item $\val\lambdao\geq\val\lambdat\geq\val\lambdath\geq0$, $r\leq\val\lambdao\leq s$, and $\val\lambdao+\val\lambdat+\val\lambdath=r+s$.
\end{itemize}
\end{exam}

\begin{prop}\label{class crys D23}
\begin{enumerate}
\item $D_{cris}^{23}$ represents admissible filtered $\phi$-modules.
\item The corresponding representations to $D_{cris}^{23}$ are
     \begin{itemize}
     \item non-split reducible with submodule $E\baseth$ if $\val\lambdath=0$,
     \item non-split reducible with submodule $E\baseo$ if $\val\lambdao=r$,
     \item non-split reducible with submodule $E(\baset,\baseth)$ if $\val\lambdao=s$, and
     \item irreducible if $\val\lambdao\geq\val\lambdat\geq\val\lambdath>0$ and $r<\val\lambdao<s$.
     \end{itemize}
\item $D_{cris}^{23}(\lambdao,\lambdat,\lambdath)$ is isomorphic to $D_{cris}^{23}(\lambdao',\lambdat',\lambdath')$ if and only if either
$$
\left\{
  \begin{array}{ll}
    \lambdao=\lambdao',\lambdat=\lambdat', \lambdath=\lambdath', & \hbox{or} \\
    \lambdao=\lambdao',\lambdat=\lambdath', \lambdath=\lambdat'. & \hbox{}
  \end{array}
\right.
$$
\end{enumerate}
\end{prop}

\begin{proof}
From \ref{cryssixth10}, if we let $L_{1}=E(\baseo+a\baset+b\baseth)$ and $L_{2}=E(\baseo+a\baset+b\baseth,\baseo)$ with $ab\neq0$, then, by change of a basis: $\baseo\mapsto\baseo$, $a\baset\mapsto\baset$, and $b\baseth\mapsto\baseth$, we get $D_{cris}^{23}$. So now the part (1) and (2) are clear. Since every isomorphism preserves the filtration, if there is an isomorphism then it should fix $\baseo$ to preserve $\filr D$, by lemma \ref{typeofPsixth}. To preserve $\fils D$, it can either fix or swap $\baset$ and $\baseth$. So we get the part (3).
\end{proof}

The following example arises from \ref{cryssixth11}.
\begin{exam}
A filtered $\phi$-module of Hodge type $(0,r,s)$
$$D_{cris}^{24}=D_{cris}^{24}(\lambdao,\lambdat,\lambdath);$$
\begin{itemize}
\item $\fil{r} D=E(\baseo+\baset+\baseth,\baset)\mbox{ and }\fil{s} D=E(\baseo+\baset+\baseth)$
\item $
[\phi]=
\begin{small}\left(
  \begin{array}{ccc}
    \lambdao & 0  & 0 \\
    0 & \lambdat & 0 \\
    0 & 0 & \lambdath \\
  \end{array}
\right)
\end{small}
$ for distinct $\lambda_{i}$'s in $E$.
\item $\val\lambdao\geq\val\lambdat\geq\val\lambdath\geq0$, $\val\lambdat\geq r$, and $\val\lambdao+\val\lambdat+\val\lambdath=r+s$.
\end{itemize}
\end{exam}

\begin{prop}\label{class crys D24}
\begin{enumerate}
\item $D_{cris}^{24}$ represents admissible filtered $\phi$-modules.
\item The corresponding representations to $D_{cris}^{24}$ are
     \begin{itemize}
     \item non-split reducible with submodule $E\baseth$ if $\val\lambdath=0$,
     \item non-split reducible with submodule $E\baset$ if $\val\lambdat=r$,
     \item non-split reducible with submodule $E(\baset,\baseth)$ if $\val\lambdao=s$, and
     \item irreducible if $\val\lambdao\geq\val\lambdat\geq\val\lambdath>0$ and $\val\lambdat>r$.
     \end{itemize}
\item $D_{cris}^{24}(\lambdao,\lambdat,\lambdath)$ is isomorphic to $D_{cris}^{24}(\lambdao',\lambdat',\lambdath')$ if and only if either
$$
\left\{
  \begin{array}{ll}
    \lambdao=\lambdao',\lambdat=\lambdat', \lambdath=\lambdath', & \hbox{or} \\
    \lambdao=\lambdath',\lambdat=\lambdat', \lambdath=\lambdao'. & \hbox{}
  \end{array}
\right.
$$
\end{enumerate}
\end{prop}

\begin{proof}
From \ref{cryssixth11}, if we let $L_{1}=E(\baseo+a\baset+b\baseth)$ and $L_{2}=E(\baseo+a\baset+b\baseth,\baset)$ with $ab\neq0$, then, by change of a basis: $\baseo\mapsto\baseo$, $a\baset\mapsto\baset$, and $b\baseth\mapsto\baseth$, we get $D_{cris}^{24}$. So now the part (1) and (2) are clear. Since every isomorphism preserves the filtration, if there is an isomorphism then it should fix $\baset$ to preserve $\filr D$, by lemma \ref{typeofPsixth}. To preserve $\fils D$, it can either fix or swap $\baseo$ and $\baseth$. So we get the part (3).
\end{proof}

The following example arises from \ref{cryssixth12}.
\begin{exam}
A filtered $\phi$-module of Hodge type $(0,r,s)$
$$D_{cris}^{25}=D_{cris}^{25}(\lambdao,\lambdat,\lambdath);$$
\begin{itemize}
\item $\fil{r} D=E(\baseo+\baset+\baseth,\baseth)\mbox{ and }\fil{s} D=E(\baseo+\baset+\baseth)$
\item $
[\phi]=
\begin{small}\left(
  \begin{array}{ccc}
    \lambdao & 0  & 0 \\
    0 & \lambdat & 0 \\
    0 & 0 & \lambdath \\
  \end{array}
\right)
\end{small}
$ for distinct $\lambda_{i}$'s in $E$.
\item $\val\lambdao\geq\val\lambdat\geq\val\lambdath\geq r$, $\val\lambdao+\val\lambdat+\val\lambdath=r+s$, and $s\geq2r$.
\end{itemize}
\end{exam}

\begin{prop}\label{class crys D25}
\begin{enumerate}
\item $D_{cris}^{25}$ represents admissible filtered $\phi$-modules.
\item The corresponding representations to $D_{cris}^{25}$ are
     \begin{itemize}
     \item non-split reducible with submodule $E\baseth$ if $\val\lambdath=r$ and
     \item irreducible if $\val\lambdao\geq\val\lambdat\geq\val\lambdath>r$ and so $s\geq2r+1$.
     \end{itemize}
\item $D_{cris}^{25}(\lambdao,\lambdat,\lambdath)$ is isomorphic to $D_{cris}^{25}(\lambdao',\lambdat',\lambdath')$ if and only if either
$$
\left\{
  \begin{array}{ll}
    \lambdao=\lambdao',\lambdat=\lambdat', \lambdath=\lambdath', & \hbox{or} \\
    \lambdao=\lambdat',\lambdat=\lambdao', \lambdath=\lambdath'. & \hbox{}
  \end{array}
\right.
$$
\end{enumerate}
\end{prop}

\begin{proof}
From \ref{cryssixth12}, if we let $L_{1}=E(\baseo+a\baset+b\baseth)$ and $L_{2}=E(\baseo+a\baset+b\baseth,\baseth)$ with $ab\neq0$, then, by change of a basis: $\baseo\mapsto\baseo$, $a\baset\mapsto\baset$, and $b\baseth\mapsto\baseth$, we get $D_{cris}^{25}$. So now the part (1) and (2) are clear. Since every isomorphism preserves the filtration, if there is an isomorphism then it should fix $\baseth$ to preserve $\filr D$, by lemma \ref{typeofPsixth}. To preserve $\fils D$, it can either fix or swap $\baseo$ and $\baset$. So we get the part (3).
\end{proof}

The following example arises from \ref{cryssixth13}.
\begin{exam}
A filtered $\phi$-module of Hodge type $(0,r,s)$
$$D_{cris}^{26}=D_{cris}^{26}(\lambdao,\lambdat,\lambdath,\mfl);$$
\begin{itemize}
\item $\fil{r} D=E(\baseo+\baset+\baseth,\baset+\mfl\baseth)\mbox{ and }\fil{s} D=E(\baseo+\baset+\baseth)$
\item $
[\phi]=
\begin{small}\left(
  \begin{array}{ccc}
    \lambdao & 0  & 0 \\
    0 & \lambdat & 0 \\
    0 & 0 & \lambdath \\
  \end{array}
\right)
\end{small}
$ for distinct $\lambda_{i}$'s in $E$.
\item $\mfl\in E\setminus\{0,1\}$, $s\geq\val\lambdao\geq\val\lambdat\geq\val\lambdath\geq 0$, and $\val\lambdao+\val\lambdat+\val\lambdath=r+s$.
\end{itemize}
\end{exam}

\begin{prop}\label{class crys D26}
\begin{enumerate}
\item $D_{cris}^{26}$ represents admissible filtered $\phi$-modules.
\item The corresponding representations to $D_{cris}^{26}$ are
     \begin{itemize}
     \item non-split reducible with submodule $E\baseth$ if $\val\lambdath=0$,
     \item non-split reducible with submodule $E(\baset,\baseth)$ if $\val\lambdao=s$, and
     \item irreducible if $s>\val\lambdao\geq\val\lambdat\geq\val\lambdath>0$.
     \end{itemize}
\item $D_{cris}^{26}(\lambdao,\lambdat,\lambdath,\mfl)$ is isomorphic to $D_{cris}^{26}(\lambdao',\lambdat',\lambdath',\mfl')$ if and only if one of the following holds:
$$
\left\{
  \begin{array}{ll}
\lambdao=\lambdao', \lambdat=\lambdat', \lambdath=\lambdath',\mbox{ and } \mathfrak{L}=\mathfrak{L}', & \hbox{}\\
\lambdao=\lambdat', \lambdat=\lambdath', \lambdath=\lambdao',\mbox{ and } \mathfrak{L}'+\frac{1}{\mathfrak{L}}=1, & \hbox{}\\
\lambdao=\lambdath', \lambdat=\lambdao', \lambdath=\lambdat',\mbox{ and } \mathfrak{L}+\frac{1}{\mathfrak{L}'}=1, & \hbox{}\\
\lambdao=\lambdao', \lambdat=\lambdath', \lambdath=\lambdat',\mbox{ and } \mathfrak{L}\mathfrak{L}'=1, & \hbox{}\\
\lambdao=\lambdath', \lambdat=\lambdat', \lambdath=\lambdao', \mbox{ and } \frac{1}{\mathfrak{L}}+\frac{1}{\mathfrak{L}'}=1, & \hbox{or}\\
\lambdao=\lambdat', \lambdat=\lambdao', \lambdath=\lambdath', \mbox{ and } \mathfrak{L}+\mathfrak{L}'=1.
  \end{array}
\right.
$$
\end{enumerate}
\end{prop}

\begin{proof}
From \ref{cryssixth13}, if we let $L_{1}=E(\baseo+a\baset+b\baseth)$ and $L_{2}=E(\baseo+a\baset+b\baseth,c\baset+d\baseth)$ with $abcd\neq0$ and $ad-bc\neq0$, then, by change of a basis: $\baseo\mapsto\baseo$, $a\baset\mapsto\baset$, and $b\baseth\mapsto\baseth$, we get $D_{cris}^{25}$. So now the part (1) and (2) are clear.

By Lemma \ref{typeofPsixth}, to preserve $\fils D$, every isomorphism should be a permutation on the bases modulo scalar multiples. Let $T$ be an isomorphism from $D_{cris}^{26}(\lambdao,\lambdat,\lambdath,\mfl)$ to $D^{26}(\lambdao',\lambdat',\lambdath',\mfl')$. If $\lambdao=\lambdat',\lambdat=\lambdath',\lambdath=\lambdao'$, then $T$ sends $\baseo\mapsto\baset,\baset\mapsto\baseth,\baseth\mapsto\baseo$ and $T(\filr D)=E(\baseo+\baset+\baseth, \baset+(1-\frac{1}{\mfl})\baseth)$, and so we get $\mfl'+\frac{1}{\mfl}=1$. If $\lambdao=\lambdath',\lambdat=\lambdao',\lambdath=\lambdat'$, then $T$ sends $\baseo\mapsto\baseth,\baset\mapsto\baseo,\baseth\mapsto\baset$ and $T(\filr D)=E(\baseo+\baset+\baseth, \baset+\frac{1}{1-\mfl}\baseth)$, and so we get $\mfl+\frac{1}{\mfl'}=1$. If $\lambdao=\lambdao',\lambdat=\lambdath',\lambdath=\lambdat'$, then $T$ sends $\baseo\mapsto\baseo,\baset\mapsto\baseth,\baseth\mapsto\baset$ and $T(\filr D)=E(\baseo+\baset+\baseth, \baset+\frac{1}{\mfl}\baseth)$, and so we get $\mfl\mfl'=1$. If $\lambdao=\lambdath',\lambdat=\lambdat',\lambdath=\lambdao'$, then $T$ sends $\baseo\mapsto\baseth,\baset\mapsto\baset,\baseth\mapsto\baseo$ and $T(\filr D)=E(\baseo+\baset+\baseth, \baset+\frac{\mfl}{\mfl-1}\baseth)$, and so we get $\frac{1}{\mfl}+\frac{1}{\mfl'}=1$. If $\lambdao=\lambdat',\lambdat=\lambdao',\lambdath=\lambdath'$, then $T$ sends $\baseo\mapsto\baset,\baset\mapsto\baseo,\baseth\mapsto\baseth$ and $T(\filr D)=E(\baseo+\baset+\baseth, \baset+(1-\mfl)\baseth)$, and so we get $\mfl+\mfl'=1$. The converse is also very routine and easy to check.
\end{proof}

\begin{prop}
Assume that $1\leq i<j\leq 26$. Then there is an isomorphism from $D_{cris}^{i}$ to $D_{cris}^{j}$ if and only if one of the following holds:
\begin{enumerate}

\item $i=17$, $j=19$, $\lambdao=\lambdath'$, $\lambdat=\lambdat'$, and $\lambdath=\lambdao'$,
\item $i=17$, $j=19$, $\lambdao=\lambdat'$, $\lambdat=\lambdath'$, and $\lambdath=\lambdao'$,
\item $i=17$, $j=21$, $\lambdao=\lambdao'$, $\lambdat=\lambdath'$, and $\lambdath=\lambdat'$,
\item $i=17$, $j=21$, $\lambdao=\lambdath'$, $\lambdat=\lambdao'$, and $\lambdath=\lambdat'$,
\item $i=19$, $j=21$, $\lambdao=\lambdat'$, $\lambdat=\lambdao'$, and $\lambdath=\lambdath'$,
\item $i=19$, $j=21$, $\lambdao=\lambdat'$, $\lambdat=\lambdath'$, and $\lambdath=\lambdao'$,
\item $i=18$, $j=20$, $\lambdao=\lambdat'$, $\lambdat=\lambdath'$, and $\lambdath=\lambdao'$,
\item $i=18$, $j=22$, $\lambdao=\lambdao'$, $\lambdat=\lambdath'$, and $\lambdath=\lambdat'$,
\item $i=20$, $j=22$, $\lambdao=\lambdat'$, $\lambdat=\lambdao'$, and $\lambdath=\lambdath'$,
\item $i=23$, $j=24$, $\lambdao=\lambdat'$, $\lambdat=\lambdao'$, and $\lambdath=\lambdath'$,
\item $i=23$, $j=24$, $\lambdao=\lambdat'$, $\lambdat=\lambdath'$, and $\lambdath=\lambdao'$,
\item $i=23$, $j=25$, $\lambdao=\lambdath'$, $\lambdat=\lambdat'$, and $\lambdath=\lambdao'$,
\item $i=23$, $j=25$, $\lambdao=\lambdath'$, $\lambdat=\lambdao'$, and $\lambdath=\lambdat'$,
\item $i=24$, $j=25$, $\lambdao=\lambdao'$, $\lambdat=\lambdath'$, and $\lambdath=\lambdat'$, or
\item $i=24$, $j=25$, $\lambdao=\lambdat'$, $\lambdat=\lambdath'$, and $\lambdath=\lambdao'$,
\end{enumerate}
where $D_{cris}^{i}=D_{cris}^{i}(\lambdao,\lambdat,\lambdath)$ and $D_{cris}^{j}=D_{cris}^{j}(\lambdao',\lambdat',\lambdath')$.
\end{prop}

\begin{proof}
Since an isomorphism preserves the Jordan forms of the Frobenius map, if $D^{i}_{cris}$ is isomorphic to $D^{j}_{cris}$, then both $i$ and $j$ belong to either $\{1\}$, $[2,4]$, $\{5\}$, $[6,13]$, or $[14,26]$.

If $i$ and $j$ belong to $[2,4]$, the isomorphism should be of the form in Lemma \ref{typeofPthird}. But it is straightforward that such a form can not preserve the filtration if $i\neq j$. Similarly, if $i$ and $j$ belong to $[6,13]$, then the isomorphism should be of the form in Lemma \ref{typeofPfifth}, and such a form can not preserve the filtration either if $i\neq j$.

Assume that $D_{cris}^{i}=D_{cris}^{i}(\lambdao,\lambdat,\lambdath)$ is isomorphic to $D_{cris}^{j}=D_{cris}^{j}(\lambdao',\lambdat',\lambdath')$ for $i<j$ in $[14,26]$. Then $\{\lambdao,\lambdat,\lambdath\}=\{\lambdao',\lambdat',\lambdath'\}$. So the question reduces to whether or not there is a matrix of the form in Lemma \ref{typeofPsixth} preserving the filtration. Now it is straightforward to check that $D_{cris}^{i}$ for $i\in\{14,15,16,26\}$ is not isomorphic to any other $D_{cris}^{j}$ for $j$ in $[14,26]$. Let $i< j$ in $[17,25]$. If $D_{cris}^{i}$ is isomorphic to $D_{cris}^{j}$, their filtration should have the same form. So both $i$ and $j$ belong to either $\{17,19,21\}$, $\{18,20,22\}$, or $\{23,24,25\}$.\\
If there is an isomorphism from $D_{cris}^{17}$ to $D_{cris}^{19}$, to preserve $\fil
{s}D$ either $(\baseo\mapsto\baset,\baset\mapsto\baseth)$ or $(\baseo\mapsto\baseth,\baset\mapsto\baset)$. For either case, the isomorphism sends $\baseth$ to $\baseo$ and so it preserves $\fil{r}D$. Hence, we get (1) and (2). The cases (3), (4), (5), and (6) are similar.\\
If there is an isomorphism from $D_{cris}^{18}$ to $D_{cris}^{20}$, to preserve $\fil
{s}D$ either $(\baseo\mapsto\baset,\baset\mapsto\baseth)$ or $(\baseo\mapsto\baseth,\baset\mapsto\baset)$. For either case, the isomorphism sends $\baseth$ to $\baseo$. It is easy to check that $(\baseo\mapsto\baseth,\baset\mapsto\baset,\baseth\mapsto\baseo)$ does not preserve the filtration $\filr D$ while $(\baseo\mapsto\baset,\baset\mapsto\baseth,\baseth\mapsto\baseo)$ does, so we get (7). The cases (8) and (9) are similar.\\
If there is an isomorphism from $D_{cris}^{23}$ to $D_{cris}^{24}$, to preserve $\fil
{r}D$ the only condition we need is $\baseo\mapsto\baset$ and there are no condition to preserve $\fils D$. So we get (10) and (11). The cases (12), (13), (14), and (15) are similar.

The converse is straightforward.
\end{proof}

\begin{prop}
Every $3$-dimensional crystalline representation of $G_{\QP}$ with regular Hodge--Tate weights is isomorphic to a representation corresponding to some $D_{cris}^{i}$ up to twist by a power of the cyclotomic character.
\end{prop}

\begin{proof}
We found all the admissible filtered $\phi$-modules in the previous subsections. Since the list of filtered modules in this subsection represents all we found in the previous subsections, we are done.
\end{proof}

\section{Admissible filtered $(\phi,N)$-modules with $\rank N=1$}

In this section, we classify the admissible filtered $(\phi,N)$-modules of Hodge type $(0,r,s)$ for $0<r<s$ and with the monodromy operator $N$ of rank $1$. Assume first that $N$ has rank $1$. By choice of a basis for $D=E(\baseo,\baset,\baseth)$, we may set
$$N\baseo=\baseth\mbox{ and }N\baset=0=N\baseth.$$
From the equation $N\phi=p\phi N$, we should have that $$\phi\baseo=px\baseo+u\baset+v\baseth,\hspace{0.1cm}\phi\baset=y\baset+w\baseth,\mbox{ and }\phi\baseth=x\baseth,$$ for some $u$, $v$, $w$, and $x\not=0\not=y$. By change of a basis for $D=E(\baseo,\baset,\baseth)$, we can say a bit more.
\begin{lemm}
There is an invertible matrix $P$ such that $P[N]P^{-1}=[N]$ and
$$
P[\phi] P^{-1}=
\left\{
  \begin{array}{ll}
\tiny\left(
\begin{array}{ccc}
    px &  0 &  0 \\
    0  &  x  & 0 \\
    0  &  0  & x \\
\end{array}
\right), & \hbox{ if $y=x$ and $w=0$;} \\
\tiny\left(
\begin{array}{ccc}
    px &  0 &  0 \\
    0  &  x  & 0 \\
    0  &  1  & x \\
\end{array}
\right), & \hbox{if $y=x$ and $w\not=0$;} \\
\tiny\left(
\begin{array}{ccc}
    px &  0  & 0 \\
    0  &  px & 0 \\
    0  &  0  & x \\
\end{array}
\right), & \hbox{if $y=px$ and $u=0$;} \\
\tiny\left(
\begin{array}{ccc}
    px &  0  & 0 \\
    1  &  px & 0 \\
    0  &  0  & x \\
\end{array}
\right), & \hbox{if $y=px$ and $u\not=0$;}\\
\tiny\left(
\begin{array}{ccc}
    px &  0 &  0 \\
    0  &  y  & 0 \\
    0  &  0  & x \\
\end{array}
\right), & \hbox{if $px\not=y\not=x$.}
  \end{array}
\right.
$$
\end{lemm}

\begin{proof}
For each case, use the following matrices for $P$ in order:

$$
\begin{tiny}\left(
\begin{array}{ccc}
    1 &  0 &  0 \\
    \frac{u}{x(1-p)}  &  1  & 0 \\
    \frac{v}{x(1-p)}  &  0  & 1 \\
\end{array}
\right)
\end{tiny},
\begin{tiny}
\left(
\begin{array}{ccc}
    1 &  0 &  0 \\
    \frac{uw}{x(1-p)}  &  w  & 0 \\
    \frac{vx(1-p)-uw}{x^{2}(1-p)^{2}}  &  0  & 1 \\
\end{array}
\right)
\end{tiny},
\begin{tiny}
\left(
\begin{array}{ccc}
    1  &  0 &  0 \\
    0  &  1  & 0 \\
    \frac{v}{x(1-p)}  &  \frac{w}{x(1-p)}  & 1 \\
\end{array}
\right)
\end{tiny},
$$
$$
\begin{tiny}\left(
\begin{array}{ccc}
    u &  0 &  0 \\
    0  &  1 & 0 \\
    \frac{u^{2}w+uvx(1-p)}{x^{2}(1-p)^{2}}  &  \frac{uw}{x(1-p)}  &  u \\
\end{array}
\right)
\end{tiny},
\mbox{ and }
\begin{tiny}\left(
\begin{array}{ccc}
    1 &  0 &  0 \\
    \frac{u}{y-px}  &  1  & 0 \\
    \frac{vx-vy+uw}{x(x-y)(1-p)}  &  \frac{w}{x-y}  & 1 \\
\end{array}
\right)
\end{tiny}.
$$
\end{proof}
Hence, we may assume that $[\phi]$ has one of the forms in the Lemma above. For each type of $[\phi]$, we collect all the admissible filtered $(\phi,N)$-modules with $\rank N=1$ in the following subsections.

\subsection{The first case of $\rank N$=1}
Assume that
$$[\phi]=
\begin{small}\left(
\begin{array}{ccc}
    p\lambda &  0 &  0 \\
    0  &  \lambda  & 0 \\
    0  &  0  & \lambda \\
\end{array}
\right)
\end{small}
\mbox{ and }
[N]=
\begin{small}\left(
\begin{array}{ccc}
    0  &  0 &  0 \\
    0  &  0  & 0 \\
    1  &  0  & 0 \\
\end{array}
\right)
\end{small}.
$$
By admissibility, $$\val\lambda=\frac{r+s-1}{3}.$$

\begin{lemm}\label{rank1typeofPfirst}
For a $3\times 3$-matrix $P=(P_{i,j})$, $P[\phi]=[\phi] P$ and $P[N]=[N]P$ if and only if $P$ is a lower triangle matrix with $P_{1,1}=P_{3,3}$ and $P_{2,1}=0=P_{3,1}$.
\end{lemm}

\begin{proof}
The equation $P[N]=[N]P$ forces that $P$ be a lower triangle matrix with $P_{1,1}=P_{3,3}$. Then $P[\phi]=[\phi] P$ forces $P_{2,1}=0=P_{3,1}$.
\end{proof}

\begin{lemm} \label{rank1invariantfirst}
\begin{enumerate}
\item every $1$-dimensional subspace of $E(\baset,\baseth)$ is the only $\phi$- and $N$-invariant subspace of dimension $1$.
\item $E(\baset,\baseth)$ and $E(\baseo,\baseth)$ are the only $\phi$- and $N$-invariant subspaces of dimension $2$.
\end{enumerate}
\end{lemm}

\begin{proof}
By Lemma \ref{crysinvariantfourth}, we know that the $\phi$-invariant subspaces of dimension $1$ are $E\baseo$ and the $1$-dimensional subspaces of $E(\baset,\baseth)$ and that the $\phi$-invariant subspaces of dimension $2$ are $E(\baset,\baseth)$ and the $2$-dimensional subspaces including $\baseo$. Now it is easy to check which ones are $N$-invariant.
\end{proof}

We start to collect the admissible filtered $(\phi,N)$-modules in this case.
\subsubsection{}
Assume that $L_{1}$ is $\phi$- and $N$-invariant. Then, by admissibility,
$s=\hodgen(L_{1})\leq\newtonn(L_{1})=\val\lambda$, which contradicts to $\val\lambda=\frac{r+s-1}{3}$.

\subsubsection{}
Assume that $L_{2}=E(\baseo,\baseth)$. Then, by admissibility,
$r+s=\hodgen(L_{2})\leq\newtonn(L_{2})=2\val\lambda+1$, which contradicts to $\val\lambda=\frac{r+s-1}{3}$.

\subsubsection{} \label{rank1first1}
Assume that neither $L_{1}$ nor $L_{2}$ are $\phi$- and $N$-invariant and $L_{1}\subset E(\baseo,\baseth)$.
Then $L_{1}\not\subset E(\baset,\baseth)$. Let $D_{1}$ be a $\phi$- and $N$-invariant subspace of dimension $1$.
By admissibility, $s=\hodgen(E(\baseo,\baseth))\leq\newtonn(E(\baseo,\baseth))=2\val\lambda+1$, $r=\hodgen(E(\baset,\baseth))\leq\newtonn(E(\baset,\baseth))=2\val\lambda$, for $D_{1}$ with $D_{1}\subset L_{2}$ $r=\hodgen(D_{1})\leq\newtonn(D_{1})=\val\lambda$, and for $D_{1}$ with $D_{1}\not\subset L_{2}$ $0=\hodgen(D_{1})\leq\newtonn(D_{1})=\val\lambda$.
So admissible filtered $(\phi,N)$-modules occur in this case if and only if $s=2r+1$, but the corresponding representations are decomposable with
submodules $E(\baset,\baseth)\cap L_{2}$ and $E(\baseo,\baseth)$.

\subsubsection{} \label{rank1first2}
Assume that neither $L_{1}$ nor $L_{2}$ are $\phi$- and $N$-invariant and $L_{1}\not\subset E(\baseo,\baseth)$.
Then $L_{1}\not\subset E(\baset,\baseth)$. Let $D_{1}$ be a $\phi$- and $N$-invariant subspace of dimension $1$.
By admissibility, $r=\hodgen(E(\baseo,\baseth))\leq\newtonn(E(\baseo,\baseth))=2\val\lambda+1$, $r=\hodgen(E(\baset,\baseth))\leq\newtonn(E(\baset,\baseth))=2\val\lambda$, for $D_{1}$ with $D_{1}\subset L_{2}$ $r=\hodgen(D_{1})\leq\newtonn(D_{1})=\val\lambda$, and for $D_{1}$ with $D_{1}\not\subset L_{2}$ $0=\hodgen(D_{1})\leq\newtonn(D_{1})=\val\lambda$.
So admissible filtered $(\phi,N)$-modules occur in this case if and only if $s\geq 2r+1$. The corresponding representations are non-split reducible with submodules $E(\baset,\baseth)\cap L_{2}$ if $s=2r+1$ and irreducible if $s>2r+1$.

\subsection{The second case of $\rank N=1$}
Assume that
$$[\phi]=
\begin{small}\left(
\begin{array}{ccc}
    p\lambda &  0 &  0 \\
    0  &  \lambda  & 0 \\
    0  &  1  & \lambda \\
\end{array}
\right)
\end{small}\mbox{ and }
[N]=
\begin{small}\left(
\begin{array}{ccc}
    0  &  0 &  0 \\
    0  &  0  & 0 \\
    1  &  0  & 0 \\
\end{array}
\right)
\end{small}.
$$
By admissibility, $$\val\lambda=\frac{r+s-1}{3}.$$

\begin{lemm}\label{rank1typeofPsecond}
For a $3\times 3$-matrix $P=(P_{i,j})$, $P[\phi]=[\phi] P$ and $P[N]=[N]P$ if and only if $P$ is a lower triangle matrix with $P_{1,1}=P_{2,2}=P_{3,3}$ and $P_{2,1}=0=P_{3,1}$.
\end{lemm}

\begin{proof}
The equation $P[N]=[N]P$ forces that $P$ be a lower triangle matrix with $P_{1,1}=P_{3,3}$. Then $P[\phi]=[\phi] P$ forces $P_{2,1}=0=P_{3,1}$ and $P_{1,1}=P_{2,2}=P_{3,3}$.
\end{proof}

\begin{lemm}\label{rank1invariantsecond}
$E\baseth$, $E(\baseo,\baseth)$, and $E(\baset,\baseth)$ are the only nontrivial proper $\phi$- and $N$-invariant subspaces.
\end{lemm}

\begin{proof}
By Lemma \ref{crysinvariantfifth}, we know that $E\baseo$, $E\baseth$, $E(\baseo,\baseth)$, and $E(\baset,\baseth)$ are  the only nontrivial proper $\phi$-invariant subspaces. Now it is easy to check which ones are $N$-invariant.
\end{proof}

We start to collect the admissible filtered $(\phi,N)$-modules in this case.
\subsubsection{}
Assume that $L_{1}=E\baseth$. Then, by admissibility,
$s=\hodgen(L_{1})\leq\newtonn(L_{1})=\val\lambda$, which contradicts to $\val\lambda=\frac{r+s-1}{3}$.

\subsubsection{}
Assume that $L_{2}=E(\baset,\baseth)$. Then, by admissibility,
$r+s=\hodgen(L_{2})\leq\newtonn(L_{2})=2\val\lambda$, which contradicts to $\val\lambda=\frac{r+s-1}{3}$.

\subsubsection{}
Assume that $L_{2}=E(\baseo,\baseth)$. Then, by admissibility,
$r+s=\hodgen(L_{2})\leq\newtonn(L_{2})=2\val\lambda+1$, which contradicts to $\val\lambda=\frac{r+s-1}{3}$.

\subsubsection{} \label{rank1second1}
Assume that neither $L_{1}$ nor $L_{2}$ are $\phi$- and $N$-invariant and $L_{1}\subset E(\baset,\baseth)$.
Then $L_{1}\not\subset E(\baseo,\baseth)$ and $\baseth\not\in L_{2}$. By admissibility,
$0=\hodgen(E\baseth)\leq\newtonn(E\baseth)=\val\lambda$, $s=\hodgen(E(\baset,\baseth))\leq\newtonn(E(\baset,\baseth))=2\val\lambda$, and $r=\hodgen(E(\baseo,\baseth))\leq\newtonn(E(\baseo,\baseth))=2\val\lambda+1$.
So admissible filtered $(\phi,N)$-modules occur in this case if and only if $s\leq 2r-2$. The corresponding representations are non-split reducible with submodule $E(\baset,\baseth)$ if $s=2r-2$ and irreducible if $s<2r-2$.

\subsubsection{} \label{rank1second2}
Assume that neither $L_{1}$ nor $L_{2}$ are $\phi$- and $N$-invariant and $L_{1}\subset E(\baseo,\baseth)$.
Then $L_{1}\not\subset E(\baset,\baseth)$ and $\baseth\not\in L_{2}$. By admissibility,
$0=\hodgen(E\baseth)\leq\newtonn(E\baseth)=\val\lambda$, $r=\hodgen(E(\baset,\baseth))\leq\newtonn(E(\baset,\baseth))=2\val\lambda$, and $s=\hodgen(E(\baseo,\baseth))\leq\newtonn(E(\baseo,\baseth))=2\val\lambda+1$.
So admissible filtered $(\phi,N)$-modules occur in this case if and only if $s\leq 2r+1$. The corresponding representations are non-split reducible with submodule $E(\baseo,\baseth)$ if $s=2r+1$ and irreducible if $s<2r+1$.

\subsubsection{} \label{rank1second3}
Assume that neither $L_{1}$ nor $L_{2}$ are $\phi$- and $N$-invariant, $L_{1}$ is not contained in any $\phi$- and $N$-invariant subspaces, and $\baseth\in L_{2}$.
Then, by admissibility,
$r=\hodgen(E\baseth)\leq\newtonn(E\baseth)=\val\lambda$, $r=\hodgen(E(\baset,\baseth))\leq\newtonn(E(\baset,\baseth))=2\val\lambda$, and $r=\hodgen(E(\baseo,\baseth))\leq\newtonn(E(\baseo,\baseth))=2\val\lambda+1$.
So admissible filtered $(\phi,N)$-modules occur in this case if and only if $s\geq 2r+1$. The corresponding representations are non-split reducible with submodule $E\baseth$ if $s=2r+1$ and irreducible if $s>2r+1$.

\subsubsection{} \label{rank1second4}
Assume that neither $L_{1}$ nor $L_{2}$ are $\phi$- and $N$-invariant, $L_{1}$ is not contained in any $\phi$- and $N$-invariant subspaces, and $\baseth\not\in L_{2}$.
Then, by admissibility,
$0=\hodgen(E\baseth)\leq\newtonn(E\baseth)=\val\lambda$, $r=\hodgen(E(\baset,\baseth))\leq2\val\lambda$, and $r=\hodgen(E(\baseo,\baseth))\leq2\val\lambda+1$.
So admissible filtered $(\phi,N)$-modules occur in this case if and only if $r+2\leq 2s$.
If $0<r<s$, then $r+2<2s$, and so there are always admissible filtered $(\phi,N)$-modules in this case and the corresponding representations are irreducible.

\subsection{The third case of $\rank N=1$}
Assume that
$$[\phi]=
\begin{small}\left(
\begin{array}{ccc}
    p\lambda &  0 &  0 \\
    0  &  p\lambda  & 0 \\
    0  &  0  & \lambda \\
\end{array}
\right)
\end{small}
\mbox{ and }
[N]=
\begin{small}\left(
\begin{array}{ccc}
    0  &  0 &  0 \\
    0  &  0  & 0 \\
    1  &  0  & 0 \\
\end{array}
\right)
\end{small}.
$$
By admissibility, $$\val\lambda=\frac{r+s-2}{3}.$$

\begin{lemm}\label{rank1typeofPthird}
For a $3\times 3$-matrix $P=(P_{i,j})$, $P[\phi]=[\phi] P$ and $P[N]=[N]P$ if and only if $P$ is a lower triangle matrix with $P_{1,1}=P_{3,3}$ and $P_{3,1}=0=P_{3,2}$.
\end{lemm}

\begin{proof}
The equation $P[N]=[N]P$ forces that $P$ be a lower triangle matrix with $P_{1,1}=P_{3,3}$. Then $P[\phi]=[\phi] P$ forces $P_{3,1}=0=P_{3,2}$.
\end{proof}

\begin{lemm}\label{rank1invariantthrid}
\begin{enumerate}
\item $E\baset$ and $E\baseth$ are the only $\phi$- and $N$-invariant subspace of dimension $1$.
\item For each $(a,b)\in E^{2}\setminus\{(0,0)\}$, $E(a\baseo+b\baset,\baseth)$ is the only $\phi$- and $N$-invariant subspace of dimension $2$.
\end{enumerate}
\end{lemm}

\begin{proof}
By Lemma \ref{crysinvariantfourth}, we know that the $1$-dimensional $\phi$-invariant subspaces of $D$ are the subspaces of $E(\baseo,\baset)$ and $E\baseth$ and that the $2$-dimensional subspaces are $E(\baseo,\baset)$ and $E(a\baseo+b\baset,\baseth)$ for each $(a,b)\in E^{2}\setminus\{(0,0)\}$. Now it is easy to check which ones are $N$-invariant.
\end{proof}

We start to collect the admissible filtered $(\phi,N)$-modules in this case.
\subsubsection{}
Assume that $L_{1}=E\baset$. Then, by admissibility,
$s=\hodgen(L_{1})\leq\newtonn(L_{1})=\val\lambda+1$, which contradicts to $\val\lambda=\frac{r+s-2}{3}$.

\subsubsection{}
Assume that $L_{1}=E\baseth$. Then, by admissibility,
$s=\hodgen(L_{1})\leq\newtonn(L_{1})=\val\lambda$, which contradicts to $\val\lambda=\frac{r+s-2}{3}$.

\subsubsection{}
Assume that $L_{2}$ is $\phi$- and $N$-invariant. Then, by admissibility, $r+s=\hodgen(L_{2})\leq\newtonn(L_{2})=2\val\lambda+1$, which contradicts to $\val\lambda=\frac{r+s-2}{3}$.

\subsubsection{}\label{rank1third1}
Assume that neither $L_{1}$ nor $L_{2}$ are $\phi$- and $N$-invariant and $\baset\in L_{2}$.
Then $L_{1}\not\subset E(\baset,\baseth)$ and $\baseth\not\in L_{2}$. Let $D_{2}$ be a $\phi$- and $N$-invariant subspace of dimension $2$.
By admissibility, $r=\hodgen(E\baset)\leq\newtonn(E\baset)=\val\lambda+1$, $0=\hodgen(E\baseth)\leq\newtonn(E\baseth)=\val\lambda$, for $D_{2}$ with $L_{1}\subset D_{2}$ $s=\hodgen(D_{2})\leq\newtonn(D_{2})=2\val\lambda+1$, and for $D_{2}$ with $L_{1}\not\subset D_{2}$ $r=\hodgen(D_{2})\leq\newtonn(D_{2})=2\val\lambda+1$. So admissible filtered $(\phi,N)$-modules occur in this case if and only if $s=2r-1$, and the corresponding representations are decomposable with submodules $E\baset$ and $D_{2}$ with $L_{1}\subset D_{2}$.

\subsubsection{} \label{rank1third2}
Assume that neither $L_{1}$ nor $L_{2}$ are $\phi$- and $N$-invariant and $\baset\not\in L_{2}$. Then $\baseth\not\in L_{2}$. Let $D_{2}$ be a $\phi$- and $N$-invariant subspace of dimension $2$. By admissibility, $0=\hodgen(E\baset)\leq\newtonn(E\baset)=\val\lambda+1$, $0=\hodgen(E\baseth)\leq\newtonn(E\baseth)=\val\lambda$, for $D_{2}$ with $L_{1}\subset D_{2}$ $s=\hodgen(D_{2})\leq=\newtonn(D_{2})=2\val\lambda+1$, and for $D_{2}$ with $L_{1}\not\subset D_{2}$ $r=\hodgen(D_{2})\leq\newtonn(D_{2})=2\val\lambda+1$.
So admissible filtered $(\phi,N)$-modules occur in this case if and only if $s\leq 2r-1$. The corresponding representations are non-split reducible with submodules $D_{2}$ with $L_{1}\subset D_{2}$ if $s=2r-1$ and irreducible if $s<2r-1$.

\subsection{The fourth case of $\rank N=1$}
Assume that
$$[\phi]=
\begin{small}\left(
\begin{array}{ccc}
    p\lambda &  0 &  0 \\
    1  &  p\lambda  & 0 \\
    0  &  0  & \lambda \\
\end{array}
\right)
\end{small}
\mbox{ and }
[N]=
\begin{small}\left(
\begin{array}{ccc}
    0  &  0 &  0 \\
    0  &  0  & 0 \\
    1  &  0  & 0 \\
\end{array}
\right)
\end{small}.
$$
By admissibility, $$\val\lambda=\frac{r+s-2}{3}.$$

\begin{lemm}\label{rank1typeofPfourth}
For a $3\times 3$-matrix $P=(P_{i,j})$, $P[\phi]=[\phi]P$ and $P[N]=[N]P$ if and only if $P$ is a lower triangle matrix with $P_{1,1}=P_{2,2}=P_{3,3}$ and $P_{3,1}=0=P_{3,2}$.
\end{lemm}

\begin{proof}
The equation $P[N]=[N]P$ forces that $P$ be a lower triangle matrix with $P_{1,1}=P_{3,3}$. Then $P[\phi]=[\phi] P$ forces $P_{3,1}=0=P_{3,2}$ and $P_{1,1}=P_{2,2}=P_{3,3}$.
\end{proof}

\begin{lemm}\label{rank1invariantfourth}
$E\baset$, $E\baseth$, and $E(\baset,\baseth)$ are the only nontrivial proper $\phi$- and $N$-invariant subspaces.
\end{lemm}

\begin{proof}
By Lemma \ref{crysinvariantfifth}, we know that $E\baset$, $E\baseth$, $E(\baseo,\baset)$ and $E(\baset,\baseth)$ are the only nontrivial proper $\phi$-invariant subspaces. Now it is easy to check which ones are $N$-invariant.
\end{proof}

We start to collect the admissible filtered $(\phi,N)$-modules in this case.
\subsubsection{}
Assume that $L_{1}=E\baset$. Then, by admissibility,
$s=\hodgen(L_{1})\leq\newtonn(L_{1})=\val\lambda+1$, which contradicts to $\val\lambda=\frac{r+s-2}{3}$.

\subsubsection{}
Assume that $L_{1}=E\baseth$. Then, by admissibility,
$s=\hodgen(L_{1})\leq\newtonn(L_{1})=\val\lambda$, which contradicts to $\val\lambda=\frac{r+s-2}{3}$.

\subsubsection{}
Assume that $L_{2}=E(\baset,\baseth)$. Then, by admissibility,
$r+s=\hodgen(L_{2})\leq\newtonn(L_{2})=2\val\lambda+1$, which contradicts to $\val\lambda=\frac{r+s-2}{3}$.

\subsubsection{}\label{rank1fourth1}
Assume that neither $L_{1}$ nor $L_{2}$ are $\phi$- and $N$-invariant and $L_{1}\subset E(\baset,\baseth)$.
Then $\baset,\baseth\not\in L_{2}$. By admissibility,
$0=\hodgen(E\baset)\leq\newtonn(E\baset)=\val\lambda+1$, $0=\hodgen(E\baseth)\leq\newtonn(E\baseth)=\val\lambda$, and $s=\hodgen(E(\baset,\baseth))\leq\newtonn(E(\baset,\baseth))=2\val\lambda+1$.
So admissible filtered $(\phi,N)$-modules occur in this case if and only if $s\leq 2r-1$. The corresponding representations are non-split reducible with submodule $E(\baset,\baseth)$ if $s=2r-1$ and irreducible if $s<2r-1$.

\subsubsection{}\label{rank1fourth2}
Assume that neither $L_{1}$ nor $L_{2}$ are $\phi$- and $N$-invariant, $L_{1}\not\subset E(\baset,\baseth)$, and $\baset\in L_{2}$.
Then $\baseth\not\in L_{2}$. By admissibility,
$r=\hodgen(E\baset)\leq\newtonn(E\baset)=\val\lambda+1$, $0=\hodgen(E\baseth)\leq\newtonn(E\baseth)=\val\lambda$, and $r=\hodgen(E(\baset,\baseth))\leq\newtonn(E(\baset,\baseth))=2\val\lambda+1$.
So admissible filtered $(\phi,N)$-modules occur in this case if and only if $s\geq 2r-1$. The corresponding representations are non-split reducible with submodule $E\baset$ if $s=2r-1$ and irreducible if $s>2r-1$.

\subsubsection{}\label{rank1fourth3}
Assume that neither $L_{1}$ nor $L_{2}$ are $\phi$- and $N$-invariant, $L_{1}\not\subset E(\baset,\baseth)$, and $\baseth\in L_{2}$.
Then $\baset\not\in L_{2}$. By admissibility,
$0=\hodgen(E\baset)\leq\newtonn(E\baset)=\val\lambda+1$, $r=\hodgen(E\baseth)\leq\newtonn(E\baseth)=\val\lambda$, and $r=\hodgen(E(\baset,\baseth))\leq\newtonn(E(\baset,\baseth))=2\val\lambda+1$.
So admissible filtered $(\phi,N)$-modules occur in this case if and only if $s\geq 2r+2$. The corresponding representations are non-split reducible with submodule $E\baseth$ if $s=2r+2$ and irreducible if $s>2r+2$.

\subsubsection{}\label{rank1fourth4}
Assume that neither $L_{1}$ nor $L_{2}$ are $\phi$- and $N$-invariant, $L_{1}\not\subset E(\baset,\baseth)$, and $\baset,\baseth\not\in L_{2}$. By admissibility, $0=\hodgen(E\baset)\leq\newtonn(E\baset)=\val\lambda+1$, $0=\hodgen(E\baseth)\leq\newtonn(E\baseth)=\val\lambda$, and $r=\hodgen(E(\baset,\baseth))\leq\newtonn(E(\baset,\baseth))=2\val\lambda+1$. So admissible filtered $(\phi,N)$-modules occur in this case if and only if $2s\geq r+1 $. If $0<r<s$, then $r+1< 2s$ and so there are always admissible filtered $(\phi,N)$-modules in this case and the corresponding representations are irreducible.

\subsection{The fifth case of $\rank N=1$}
Assume that
$$[\phi]=
\begin{small}\left(
\begin{array}{ccc}
    p\lambda &  0 &  0 \\
    0  &  \lambdat  & 0 \\
    0  &  0  & \lambda \\
\end{array}
\right)
\end{small}
\mbox{ and }
[N]=
\begin{small}\left(
\begin{array}{ccc}
    0  &  0 &  0 \\
    0  &  0  & 0 \\
    1  &  0  & 0 \\
\end{array}
\right)
\end{small}
$$
with $\lambda\not=\lambdat\not=p\lambda$. By admissibility, $$2\val\lambda+\val\lambdat=r+s-1.$$

\begin{lemm}\label{rank1typeofPfifth}
For a $3\times 3$-matrix $P=(P_{i,j})$, $P[\phi]=[\phi]P$ and $P[N]=[N]P$ if and only if $P$ is a diagonal matrix with $P_{1,1}=P_{3,3}$.
\end{lemm}

\begin{proof}
The equation $P[N]=[N]P$ forces that $P$ be a lower triangle matrix with $P_{1,1}=P_{3,3}$. Then $P[\phi]=[\phi] P$ forces that $P$ be a diagonal matrix with $P_{1,1}=P_{3,3}$.
\end{proof}

\begin{lemm}\label{rank1invariantfifth}
$E\baset$, $E\baseth$, $E(\baset,\baseth)$, and $E(\baseo,\baseth)$ are the only nontrivial proper $\phi$- and $N$-invariant subspaces.
\end{lemm}

\begin{proof}
By Lemma \ref{crysinvariantsixth}, we know that the only nontrivial proper $\phi$-invariant subspaces are $E\baseo$, $E\baset$, $E\baseth$, $E(\baseo,\baset)$, $E(\baset,\baseth)$, and $E(\baseo,\baseth)$. Now it is easy to check which ones are $N$-invariant among these subspaces.
\end{proof}

We start to collect the admissible filtered $(\phi,N)$-modules in this case.
\subsubsection{}
Assume that $L_{2}=E(\baset,\baseth)$. Then, by admissibility, $r\leq\hodgen(E(\baseo,\baseth))\leq2\val\lambda+1$ and $r+s=\hodgen(E(\baset,\baseth))\leq\newtonn(E(\baset,\baseth))=\val\lambda+\val\lambdat$, which contradicts to $2\val\lambda+\val\lambdat=r+s-1$.

\subsubsection{}\label{rank1fifth1}
Assume that $L_{1}=E\baset$ and $L_{2}\not=E(\baset,\baseth)$. Then $\baseth\not\in L_{2}$ and, by admissibility, $s=\hodgen(E\baset)\leq\newtonn(E\baset)=\val\lambdat$, $0=\hodgen(E\baseth)\leq\newtonn(E\baseth)=\val\lambda$,
$s=\hodgen(E(\baset,\baseth))\leq\newtonn(E(\baset,\baseth))=\val\lambda+\val\lambdat$, and $r=\hodgen(E(\baseo,\baseth))\leq\newtonn(E(\baseo,\baseth))=2\val\lambda+1$, which implies $\val\lambda=\frac{r-1}{2}$ and $\val\lambdat=s$. So we have admissible filtered $(\phi,N)$-modules, but the corresponding representations are decomposable with submodules $E(\baseo,\baseth)$ and $E\baset$. Moreover, $E\baseth$ and $E(\baset,\baseth)$ are submodules as well if $r=1$.

\subsubsection{}
Assume that $L_{1}=E\baseth$. Then, by admissibility,
$s=\hodgen(E\baseth)\leq\newtonn(E\baseth)=\val\lambda$ and $s\leq\hodgen(E(\baset,\baseth))\leq\newtonn(E(\baset,\baseth))=\val\lambda+\val\lambdat$, which contradicts to $2\val\lambda+\val\lambdat=r+s-1$.

\subsubsection{}\label{rank1fifth2}
Assume that $L_{1}$ is not $\phi$- and $N$-invariant and $L_{2}=E(\baseo,\baseth)$. Then $L_{1}\not\subset E(\baset,\baseth)$ and, by admissibility, $0=\hodgen(E\baset)\leq\newtonn(E\baset)=\val\lambdat$, $r=\hodgen(E\baseth)\leq\newtonn(E\baseth)=\val\lambda$, $r=\hodgen(E(\baset,\baseth))\leq\newtonn(E(\baset,\baseth))=\val\lambda+\val\lambdat$, and $r+s=\hodgen(E(\baseo,\baseth))\leq\newtonn(E(\baseo,\baseth))=2\val\lambda+1$, which implies $\val\lambda=\frac{r+s-1}{2}$ and $\val\lambdat=0$. So we have admissible filtered $(\phi,N)$-modules, but the corresponding representations are decomposable with submodules $E(\baseo,\baseth)$ and $E\baset$. Moreover, $E\baseth$ and $E(\baset,\baseth)$ are submodules as well if $s=r+1$.

\subsubsection{}\label{rank1fifth3}
Assume that neither $L_{1}$ nor $L_{2}$ are $\phi$- and $N$-invariant and $L_{1}\subset E(\baset,\baseth)$. Then $\baset,\baseth\not\in L_{2}$ and $L_{1}\not\subset E(\baseo,\baseth)$. By admissibility, $0=\hodgen(E\baset)\leq\newtonn(E\baset)=\val\lambdat$, $0=\hodgen(E\baseth)\leq\newtonn(E\baseth)=\val\lambda$, $s=\hodgen(E(\baset,\baseth))\leq\newtonn(E(\baset,\baseth))=\val\lambda+\val\lambdat$, and $r=\hodgen(E(\baseo,\baseth))\leq\newtonn(E(\baseo,\baseth))= 2\val\lambda+1$. So, for $\frac{r-1}{2}\leq\val\lambda\leq r-1$ and $2\val\lambda+\val\lambdat=r+s-1$, we have admissible filtered $(\phi,N)$-modules. The corresponding representations are
\begin{itemize}
\item non-split reducible with submodules $E\baseth$, $E(\baset,\baseth)$, and $E(\baseo,\baseth)$ if $r=1$,
\item non-split reducible with submodule $E(\baseo,\baseth)$ if $\val\lambda=\frac{r-1}{2}$ and if $r>1$,
\item non-split reducible with submodule $E(\baset,\baseth)$ if $\val\lambda=r-1$ and if $r>1$, and
\item irreducible if $\frac{r-1}{2}<\val\lambda<r-1$ and $r>1$.
\end{itemize}

\subsubsection{}\label{rank1fifth4}
Assume that neither $L_{1}$ nor $L_{2}$ are $\phi$- and $N$-invariant, $L_{1}\subset E(\baseo,\baseth)$, and $\baset\in L_{2}$. Then $\baseth\not\in L_{2}$ and $L_{1}\not\subset E(\baset,\baseth)$. By admissibility,
$r=\hodgen(E\baset)\leq\newtonn(E\baset)=\val\lambdat$, $0=\hodgen(E\baseth)\leq\newtonn(E\baseth)=\val\lambda$, $r=\hodgen(E(\baset,\baseth))\leq\newtonn(E(\baset,\baseth))=\val\lambda+\val\lambdat$, and $s=\hodgen(E(\baseo,\baseth))\leq\newtonn(E(\baseo,\baseth))=2\val\lambda+1$, which implies $\val\lambda=\frac{s-1}{2}$ and $\val\lambdat=r$. So we have admissible filtered $(\phi,N)$-modules, but the corresponding representations are decomposable with submodules $E(\baseo,\baseth)$ and $E\baset$.

\subsubsection{}\label{rank1fifth5}
Assume that neither $L_{1}$ nor $L_{2}$ are $\phi$- and $N$-invariant, $L_{1}\subset E(\baseo,\baseth)$, and $\baset\not\in L_{2}$. Then $\baseth\not\in L_{2}$ and $L_{1}\not\subset E(\baset,\baseth)$. By admissibility, $0=\hodgen(E\baset)\leq\newtonn(E\baset)=\val\lambdat$, $0=\hodgen(E\baseth)\leq\newtonn(E\baseth)=\val\lambda$, $r=\hodgen(E(\baset,\baseth))\leq\newtonn(E(\baset,\baseth))=\val\lambda+\val\lambdat$, and $s=\hodgen(E(\baseo,\baseth))\leq\newtonn(E(\baseo,\baseth))=2\val\lambda+1$. So, for $\frac{s-1}{2}\leq\val\lambda\leq\frac{r+s-1}{2}$ and $2\val\lambda+\val\lambdat=r+s-1$, we have admissible filtered $(\phi,N)$-modules.
The corresponding representations are
\begin{itemize}
\item non-split reducible with submodule $E(\baseo,\baseth)$ if $\val\lambda=\frac{s-1}{2}$,
\item non-split reducible with submodules $E\baset$ and $E(\baset,\baseth)$ if $\val\lambda=\frac{r+s-1}{2}$ and if $s=r+1$,
\item non-split reducible with submodule $E\baset$ if $\val\lambda=\frac{r+s-1}{2}$ and if $s>r+1$, and
\item irreducible if $\frac{s-1}{2}<\val\lambda<\frac{r+s-1}{2}$.
\end{itemize}

\subsubsection{}\label{rank1fifth6}
Assume that neither $L_{1}$ nor $L_{2}$ are $\phi$- and $N$-invariant, $L_{1}$ is not contained in any $\phi$- and $N$-invariant subspaces, and $\baset\in L_{2}$. Then $\baseth\not\in L_{2}$ and, by admissibility, $r=\hodgen(E\baset)\leq\newtonn(E\baset)=\val\lambdat$, $0=\hodgen(E\baseth)\leq\newtonn(E\baseth)=\val\lambda$, $r=\hodgen(E(\baset,\baseth))\leq\newtonn(E(\baset,\baseth))=\val\lambda+\val\lambdat$, and $r=\hodgen(E(\baseo,\baseth))\leq\newtonn(E(\baseo,\baseth))=2\val\lambda+1$. So, for $\frac{r-1}{2}\leq\val\lambda\leq\frac{s-1}{2}$ and $2\val\lambda+\val\lambdat=r+s-1$, we have admissible filtered $(\phi,N)$-modules.
The corresponding representations are
\begin{itemize}
\item non-split reducible with submodules $E(\baseo,\baseth)$ and $E\baseth$ if $\val\lambda=\frac{r-1}{2}$ and if $r=1$,
\item non-split reducible with submodule $E(\baseo,\baseth)$ if $\val\lambda=\frac{r-1}{2}$ and if $r>1$,
\item non-split reducible with submodule $E\baset$ if $\val\lambda=\frac{s-1}{2}$, and
\item irreducible if $\frac{r-1}{2}<\val\lambda<\frac{s-1}{2}$.
\end{itemize}

\subsubsection{}\label{rank1fifth7}
Assume that neither $L_{1}$ nor $L_{2}$ are $\phi$- and $N$-invariant, $L_{1}$ is not contained in any $\phi$- and $N$-invariant subspaces, and $\baseth\in L_{2}$. Then $\baset\not\in L_{2}$ and, by admissibility, $0=\hodgen(E\baset)\leq\newtonn(E\baset)=\val\lambdat$, $r=\hodgen(E\baseth)\leq\newtonn(E\baseth)=\val\lambda$, $r=\hodgen(E(\baset,\baseth))\leq\newtonn(E(\baset,\baseth))=\val\lambda+\val\lambdat$, and $r=\hodgen(E(\baseo,\baseth))\leq\newtonn(E(\baseo,\baseth))=2\val\lambda+1$. So, for $r\leq\val\lambda\leq\frac{r+s-1}{2}$, $2\val\lambda+\val\lambdat=r+s-1$, and $s\geq r+1$, we have admissible filtered $(\phi,N)$-modules.
The corresponding representations are
\begin{itemize}
\item non-split reducible with submodules $E\baset$, $E\baseth$, and $E(\baset,\baseth)$ if $s=r+1$,
\item non-split reducible with submodule $E\baseth$ if $\val\lambda=r$ and if $s>r+1$,
\item non-split reducible with submodule $E\baset$ if $\val\lambda=\frac{r+s-1}{2}$ and if $s>r+1$, and
\item irreducible if $r<\val\lambda<\frac{r+s-1}{2}$ and if $s>r+1$.
\end{itemize}

\subsubsection{}\label{rank1fifth8}
Assume that neither $L_{1}$ nor $L_{2}$ are $\phi$- and $N$-invariant, $L_{1}$ is not contained in any $\phi$- and $N$-invariant subspaces, and $\baset,\baseth\not\in L_{2}$. Then, by admissibility, $0=\hodgen(E\baset)\leq\newtonn(E\baset)=\val\lambdat$, $0=\hodgen(E\baseth)\leq\newtonn(E\baseth)=\val\lambda$, $r=\hodgen(E(\baset,\baseth))\leq\newtonn(E(\baset,\baseth))=\val\lambda+\val\lambdat$, and $r=\hodgen(E(\baseo,\baseth))\leq\newtonn(E(\baseo,\baseth))=2\val\lambda+1$. So, for $\frac{r-1}{2}\leq\val\lambda\leq\frac{r+s-1}{2}$ and $2\val\lambda+\val\lambdat=r+s-1$, we have admissible filtered $(\phi,N)$-modules.
The corresponding representations are
\begin{itemize}
\item non-split reducible with submodules $E\baseth$ and $E(\baseo,\baseth)$ if $\val\lambda=\frac{r-1}{2}$ and if $r=1$,
\item non-split reducible with submodule $E(\baseo,\baseth)$ if $\val\lambda=\frac{r-1}{2}$ and if $r>1$,
\item non-split reducible with submodules $E\baset$ and $E(\baset,\baseth)$ if $\val\lambda=\frac{r+s-1}{2}$ and if $s=r+1$,
\item non-split reducible with submodule $E\baset$ if $\val\lambda=\frac{r+s-1}{2}$ and if $s>r+1$, and
\item irreducible if $\frac{r-1}{2}<\val\lambda<\frac{r+s-1}{2}$.
\end{itemize}

\subsection{List of the isomorphism classes with $\rank N=1$}\label{ssec:list of N=1}
In the previous subsections, we found all of the admissible filtered $(\phi,N)$-modules of Hodge type $(0,r,s)$ for $0<r<s$ with $\rank N=1$. In this subsection, we classify the isomorphism classes of the admissible filtered $(\phi,N)$-modules on $D=E(\baseo,\baset,\baseth)$.

The following example arises from \ref{rank1first1}.
\begin{exam}
A filtered $(\phi,N)$-module of Hodge type $(0,r,s)$
$$D^{1}_{\rank  N=1}=D^{1}_{\rank  N=1}(\lambda,\mfl);$$
\begin{itemize}
\item $\fil{r} D=E(\baseo+\mfl\baseth,\baset)$ and $\fil{s} D=E(\baseo+\mfl\baseth)$.
\item $[N]=
\begin{small}\left(
  \begin{array}{ccc}
    0 & 0  & 0 \\
    0 & 0  & 0 \\
    1 & 0  & 0 \\
  \end{array}
\right)
\end{small}
$ and
$[\phi]=
\begin{small}\left(
  \begin{array}{ccc}
    p\lambda & 0  & 0 \\
    0 & \lambda & 0 \\
    0 & 0 & \lambda \\
  \end{array}
\right)
\end{small}
$ for $\lambda$ in $E$.
\item $\mfl\in E$, $\val\lambda=\frac{r+s-1}{3}$, and $s=2r+1$.
\end{itemize}
\end{exam}

\begin{prop} \label{class rank1 D1}
\begin{enumerate}
\item $D^{1}_{\rank  N=1}$ represents admissible filtered $(\phi,N)$-modules with $\rank N=1$.
\item The corresponding representations to $D^{1}_{\rank N=1}$ are decomposable with submodules $E\baset$ and $E(\baseo,\baseth)$.
\item $D^{1}_{\rank  N=1}(\lambda,\mfl)$ is isomorphic to $D^{1}_{\rank  N=1}(\lambda',\mfl')$ if and only if $\lambda=\lambda'$ and $\mfl=\mfl'$.
\end{enumerate}
\end{prop}

\begin{proof}
From \ref{rank1first1}, if we let $L_{1}=E(\baseo+a\baseth)$ and $L_{2}=(\baseo+a\baseth,\baset+b\baseth)$, then, by change of a basis: $\baseo\mapsto\baseo$, $\baset\mapsto\baset-b\baseth$, and $\baseth\mapsto\baseth$, we get $D^{1}_{\rank N=1}$. Notice that this change of a basis does not change the matrix presentation of $\phi$ by Lemma \ref{rank1typeofPfirst}. So now the part (1) and (2) are immediate from \ref{rank1first1}.

For the part (3), assume that $T$ is an isomorphism from $D^{1}_{\rank  N=1}(\lambda,\mfl)$ to $D^{1}_{\rank  N=1}(\lambda',\mfl')$. Clearly, $\lambda=\lambda'$, and, by Lemma \ref{rank1typeofPfirst}, $[T]$ is a lower triangle matrix such that $T_{1,1}=T_{3,3}$ and $T_{2,1}=T_{3,1}=0$. Since $T$ preserves the filtration, $[T]$ should be a diagonal matrix with $T_{1,1}=T_{3,3}$, which implies that $\mfl=\mfl'$. The converse is clear.
\end{proof}

The following example arises from \ref{rank1first2}.
\begin{exam}
A filtered $(\phi,N)$-module of Hodge type $(0,r,s)$
$$D^{2}_{\rank  N=1}=D^{2}_{\rank  N=1}(\lambda,[\mflo:\mflt]);$$
\begin{itemize}
\item $\fil{r} D=E(\baseo+\baset,\mflo\baset+\mflt\baseth)$ and $\fil{s} D=E(\baseo+\baset)$.
\item $[N]=
\begin{small}\left(
  \begin{array}{ccc}
    0 & 0  & 0 \\
    0 & 0  & 0 \\
    1 & 0  & 0 \\
  \end{array}
\right)
\end{small}
$ and
$[\phi]=
\begin{small}\left(
  \begin{array}{ccc}
    p\lambda & 0  & 0 \\
    0 & \lambda & 0 \\
    0 & 0 & \lambda \\
  \end{array}
\right)
\end{small}
$ for $\lambda$ in $E$.
\item $[\mflo:\mflt]\in \mathbb{P}^{1}(E)$, $\val\lambda=\frac{r+s-1}{3}$, and $s\geq 2r+1$.
\end{itemize}
\end{exam}

\begin{prop} \label{class rank1 D2}
\begin{enumerate}
\item $D^{2}_{\rank  N=1}$ represents admissible filtered $(\phi,N)$-modules with $\rank N=1$.
\item The corresponding representations to $D^{2}_{\rank N=1}$ are
    \begin{itemize}
    \item non-split reducible with submodule $E(\mflo\baset+\mflt\baseth)$ if $s=2r+1$ and
    \item irreducible if $s>2r+1$.
    \end{itemize}
\item $D^{2}_{\rank  N=1}(\lambda,[\mflo:\mflt])$ is isomorphic to $D^{2}_{\rank  N=1}(\lambda',[\mflo':\mflt'])$ if and only if $\lambda=\lambda'$ and $[\mflo:\mflt]=[\mflo':\mflt']$.
\end{enumerate}
\end{prop}

\begin{proof}
From \ref{rank1first2}, if we let $L_{1}=E(\baseo+a\baset+b\baseth)$ and $L_{2}=E(\baseo+a\baset+b\baseth,c\baset+d\baseth)$ with $a\neq0$ and $[c,d]\in\mathbb{P}^{1}(E)$, then, by change of a basis: $\baseo\mapsto\baseo$, $a\baset\mapsto\baset-b\baseth$, and $\baseth\mapsto\baseth$, we get $D^{2}_{\rank N=1}$. Now the part (1) and (2) are immediate from \ref{rank1first2}.
For the part (3), use the same argument as in Proposition \ref{class rank1 D1}.
\end{proof}

The following example arises from \ref{rank1second1}.
\begin{exam}
A filtered $(\phi,N)$-module of Hodge type $(0,r,s)$
$$D^{3}_{\rank  N=1}=D^{3}_{\rank  N=1}(\lambda,\mathfrak{L});$$
\begin{itemize}
\item $\fil{r} D=E(\baset,\baseo+\mathfrak{L}\baseth)$ and $\fil{s} D=E\baset$.
\item
$[N]=
\begin{small}\left(
  \begin{array}{ccc}
    0 & 0  & 0 \\
    0 & 0  & 0 \\
    1 & 0  & 0 \\
  \end{array}
\right)
\end{small}
$ and
$[\phi]=
\begin{small}\left(
  \begin{array}{ccc}
    p\lambda & 0  & 0 \\
    0 & \lambda & 0 \\
    0 & 1 & \lambda \\
  \end{array}
\right)
\end{small}$
for $\lambda$ in $E$.
\item $\mathfrak{L}\in E$, $\val\lambda=\frac{r+s-1}{3}$, and $s\leq 2r-2$.
\end{itemize}
\end{exam}

\begin{prop}\label{class rank1 D3}
\begin{enumerate}
\item $D^{3}_{\rank  N=1}$ represents admissible filtered $(\phi,N)$-modules with $\rank N=1$.
\item The corresponding representations to $D^{3}_{\rank N=1}$ are
     \begin{itemize}
     \item non-split reducible with submodule $E(\baset,\baseth)$ if $s=2r-2$ and
     \item irreducible if $s<2r-2$.
     \end{itemize}
\item $D^{3}_{\rank  N=1}(\lambda,\mfl)$ is isomorphic to $D^{3}_{\rank  N=1}(\lambda',\mfl')$ if and only if $\lambda=\lambda'$ and $\mfl=\mfl'$.
\end{enumerate}
\end{prop}

\begin{proof}
From \ref{rank1second1} if we let $L_{1}=E(\baset+a\baseth)$ and $L_{2}=E(\baset+a\baseth,\baseo+b\baseth)$, then, by change of a basis: $\baseo\mapsto\baseo$, $\baset\mapsto\baset-a\baseth$, and $\baseth\mapsto\baseth$, we get $D^{3}_{\rank N=1}$. So now the part (1) and (2) are immediate from \ref{rank1second1}.

For the part (3), assume that $T$ is an isomorphism from $D^{3}_{\rank  N=1}(\lambda,\mfl)$ to $D^{3}_{\rank  N=1}(\lambda',\mfl')$. Clearly, $\lambda=\lambda'$, and, by lemma \ref{rank1typeofPsecond}, $[T]$ is a lower triangle matrix such that $T_{1,1}=T_{2,2}=T_{3,3}$ and $T_{2,1}=T_{3,1}=0$. Since $T$ preserves the filtration, $[T]$ should be a scalar multiple of the identity, which implies $\mfl=\mfl'$. The converse is clear.
\end{proof}

The following example arises from \ref{rank1second2}.
\begin{exam}
A filtered $(\phi,N)$-module of Hodge type $(0,r,s)$
$$D^{4}_{\rank  N=1}=D^{4}_{\rank  N=1}(\lambda,\mathfrak{L});$$
\begin{itemize}
\item $\fil{r} D=E(\baseo+\mathfrak{L}\baseth,\baset)$ and $\fil{s} D=E(\baseo+\mathfrak{L}\baseth)$.
\item
$[N]=
\begin{small}\left(
  \begin{array}{ccc}
    0 & 0  & 0 \\
    0 & 0  & 0 \\
    1 & 0  & 0 \\
  \end{array}
\right)
\end{small}
$ and
$[\phi]=
\begin{small}\left(
  \begin{array}{ccc}
    p\lambda & 0  & 0 \\
    0 & \lambda & 0 \\
    0 & 1 & \lambda \\
  \end{array}
\right)
\end{small}
$ for $\lambda$ in $E$.
\item $\mfl\in E$, $\val\lambda=\frac{r+s-1}{3}$, and $s\leq 2r+1$.
\end{itemize}
\end{exam}

\begin{prop} \label{class rank1 D4}
\begin{enumerate}
\item $D^{4}_{\rank  N=1}$ represents admissible filtered $(\phi,N)$-modules with $\rank N=1$.
\item The corresponding representations to $D^{4}_{\rank N=1}$ are
     \begin{itemize}
     \item non-split reducible with submodule $E(\baseo,\baseth)$ if $s=2r+1$ and
     \item irreducible if $s<2r+1$.
     \end{itemize}
\item $D^{4}_{\rank  N=1}(\lambda,\mfl)$ is isomorphic to $D^{4}_{\rank  N=1}(\lambda',\mfl')$ if and only if $\lambda=\lambda'$ and $\mfl=\mfl'$.
\end{enumerate}
\end{prop}

\begin{proof}
From \ref{rank1second2}, if we let $L_{1}=E(\baseo+a\baseth)$ and $L_{2}=E(\baseo+a\baseth,\baset+b\baseth)$, then, by change of a basis: $\baseo\mapsto\baseo$, $\baset\mapsto\baset-b\baseth$, and $\baseth\mapsto\baseth$, we get $D^{4}_{\rank N=1}$. So now the part (1) and (2) are immediate from \ref{rank1second2}. For the part (3), use the same argument as in Proposition \ref{class rank1 D3}.
\end{proof}

The following example arises from \ref{rank1second3}.
\begin{exam}
A filtered $(\phi,N)$-module of Hodge type $(0,r,s)$
$$D^{5}_{\rank  N=1}=D^{5}_{\rank  N=1}(\lambda,\mfl);$$
\begin{itemize}
\item $\fil{r} D=E(\baseo+\mathfrak{L}\baset,\baseth)$ and $\fil{s} D=E(\baseo+\mathfrak{L}\baset)$.
\item
$[N]=
\begin{small}\left(
  \begin{array}{ccc}
    0 & 0  & 0 \\
    0 & 0  & 0 \\
    1 & 0  & 0 \\
  \end{array}
\right)
\end{small}
$ and
$[\phi]=
\begin{small}\left(
  \begin{array}{ccc}
    p\lambda & 0  & 0 \\
    0 & \lambda & 0 \\
    0 & 1 & \lambda \\
  \end{array}
\right)
\end{small}
$ for $\lambda$ in $E$.
\item $\mfl\in E^{\times}$, $\val\lambda=\frac{r+s-1}{3}$, and $s\geq 2r+1$.
\end{itemize}
\end{exam}

\begin{prop} \label{class rank1 D5}
\begin{enumerate}
\item $D^{5}_{\rank  N=1}$ represents admissible filtered $(\phi,N)$-modules with $\rank N=1$.
\item The corresponding representations to $D^{5}_{\rank N=1}$ are
     \begin{itemize}
     \item non-split reducible with submodule $E\baseth$ if $s=2r+1$ and
     \item irreducible if $s>2r+1$.
     \end{itemize}
\item $D^{5}_{\rank  N=1}(\lambda,\mfl)$ is isomorphic to $D^{5}_{\rank  N=1}(\lambda',\mfl')$ if and only if $\lambda=\lambda'$ and $\mfl=\mfl'$.
\end{enumerate}
\end{prop}

\begin{proof}
From \ref{rank1second3}, if we let $L_{1}=E(\baseo+a\baset+b\baseth)$ with $a\not=0$ and $L_{2}=E(\baseo+a\baset+b\baseth,\baseth)$, then, by change of a basis: $\baseo\mapsto\baseo$, $\baset\mapsto\baset-\frac{b}{a}\baseth$, and $\baseth\mapsto\baseth$, we get $D^{5}_{\rank N=1}$. So now the part (1) and (2) are immediate from \ref{rank1second3}. For the part (3), use the same argument as in Proposition \ref{class rank1 D3}.
\end{proof}

The following example arises from \ref{rank1second4}.
\begin{exam}
A filtered $(\phi,N)$-module of Hodge type $(0,r,s)$
$$D^{6}_{\rank  N=1}=D^{6}_{\rank  N=1}(\lambda,\mflo,\mflt);$$
\begin{itemize}
\item $\fil{r} D=E(\baseo+\mflo\baset,\baset+\mflt\baseth)$ and $\fil{s} D=E(\baseo+\mflo\baset)$.
\item
$[N]=
\begin{small}\left(
  \begin{array}{ccc}
    0 & 0  & 0 \\
    0 & 0  & 0 \\
    1 & 0  & 0 \\
  \end{array}
\right)
\end{small}
$ and
$[\phi]=
\begin{small}\left(
  \begin{array}{ccc}
    p\lambda & 0  & 0 \\
    0 & \lambda & 0 \\
    0 & 1 & \lambda \\
  \end{array}
\right)
\end{small}
$ for $\lambda$ in $E$.
\item $\mathfrak{L}_{1}\in E^{\times}$, $\mathfrak{L}_{2}\in E$, and $\val\lambda=\frac{r+s-1}{3}$.
\end{itemize}
\end{exam}

\begin{prop} \label{class rank1 D6}
\begin{enumerate}
\item $D^{6}_{\rank  N=1}$ represents admissible filtered $(\phi,N)$-modules with $\rank N=1$.
\item The corresponding representations to $D^{6}_{\rank  N=1}$ are irreducible.
\item $D^{6}_{\rank  N=1}(\lambda,\mflo,\mflt)$ is isomorphic to $D^{6}_{\rank  N=1}(\lambda',\mflo',\mflt')$ if and only if $\lambda=\lambda'$, $\mflo=\mflo'$, and $\mflt=\mflt'$.
\end{enumerate}
\end{prop}

\begin{proof}
From \ref{rank1second4}, if we let $L_{1}=E(\baseo+a\baset+b\baseth)$ with $a\not=0$ and $L_{2}=E(\baseo+a\baset+b\baseth,\baset+c\baseth)$, then, by change of a basis: $\baseo\mapsto\baseo$, $\baset\mapsto\baset-\frac{b}{a}\baseth$, and $\baseth\mapsto\baseth$, we get $D^{6}_{\rank N=1}$. So now the part (1) and (2) are immediate from \ref{rank1second4}. For the part (3), use the same argument as in Proposition \ref{class rank1 D3}.
\end{proof}

The following example arises from \ref{rank1third1}.
\begin{exam}
A filtered $(\phi,N)$-module of Hodge type $(0,r,s)$
$$D^{7}_{\rank  N=1}=D^{7}_{\rank  N=1}(\lambda,\mfl);$$
\begin{itemize}
\item $\fil{r} D=E(\baseo+\mfl\baseth,\baset)$ and $\fil{s} D=E(\baseo+\mfl\baseth)$.
\item
$[N]=
\begin{small}\left(
  \begin{array}{ccc}
    0 & 0  & 0 \\
    0 & 0  & 0 \\
    1 & 0  & 0 \\
  \end{array}
\right)
\end{small}
$ and
$[\phi]=
\begin{small}\left(
  \begin{array}{ccc}
    p\lambda & 0  & 0 \\
    0 & p\lambda & 0 \\
    0 & 0 & \lambda \\
  \end{array}
\right)
\end{small}
$ for $\lambda$ in $E$.
\item $\mfl\in E$, $\val\lambda=\frac{r+s-2}{3}$, and $s=2r-1$.
\end{itemize}
\end{exam}

\begin{prop} \label{class rank1 D7}
\begin{enumerate}
\item $D^{7}_{\rank  N=1}$ represents admissible filtered $(\phi,N)$-modules with $\rank N=1$.
\item The corresponding representations to $D^{7}_{\rank  N=1}$ are decomposable with submodules $E\baset$ and $E(\baseo,\baseth)$.
\item $D^{7}_{\rank  N=1}(\lambda,\mfl)$ is isomorphic to $D^{7}_{\rank  N=1}(\lambda',\mfl')$ if and only if $\lambda=\lambda'$ and $\mfl=\mfl'$.
\end{enumerate}
\end{prop}

\begin{proof}
From \ref{rank1third1}, if we let $L_{1}=E(\baseo+a\baset+b\baseth)$ and $L_{2}=E(\baseo+a\baset+b\baseth,\baset)$, then, by change of a basis: $\baseo\mapsto\baseo-a\baset$, $\baset\mapsto\baset$, and $\baseth\mapsto\baseth$, we get $D^{7}_{\rank N=1}$. So now the part (1) and (2) are immediate from \ref{rank1third1}.

For the part (3), assume that $T$ is an isomorphism from $D^{7}_{\rank  N=1}(\lambda,\mfl)$ to $D^{7}_{\rank  N=1}(\lambda',\mfl')$. Clearly, $\lambda=\lambda'$, and, by lemma \ref{rank1typeofPthird}, $[T]$ is a lower triangle matrix such that $T_{1,1}=T_{3,3}$ and $T_{3,1}=T_{3,2}=0$. Since $T$ preserves the filtration, $[T]$ should be a diagonal matrix with $T_{1,1}=T_{3,3}$, which implies $\mfl=\mfl'$. The converse is clear.
\end{proof}

The following example arises from \ref{rank1third2}.
\begin{exam}
A filtered $(\phi,N)$-module of Hodge type $(0,r,s)$
$$D^{8}_{\rank  N=1}=D^{8}_{\rank  N=1}(\lambda,[\mflo:\mflt]);$$
\begin{itemize}
\item $\fil{r} D=E(\baseo,\baset+\baseth)$ and
$\fil{s} D=E(\mflo\baseo+\mflt(\baset+\baseth))$
\item $
[N]=
\begin{small}\left(
  \begin{array}{ccc}
    0 & 0  & 0 \\
    0 & 0  & 0 \\
    1 & 0  & 0 \\
  \end{array}
\right)
\end{small}
$ and $
[\phi]=
\begin{small}
\left(
  \begin{array}{ccc}
    p\lambda & 0  & 0 \\
    0 & p\lambda & 0 \\
    0 & 0 & \lambda \\
  \end{array}
\right)
\end{small}
$ for $\lambda$ in $E$.
\item $[\mflo:\mflt]\in \mathbb{P}^{1}(E)$, $\val\lambda=\frac{r+s-2}{3}$, and $s\leq 2r-1$.
\end{itemize}
\end{exam}

\begin{prop} \label{class rank1 D8}
\begin{enumerate}
\item $D^{8}_{\rank  N=1}$ represents admissible filtered $(\phi,N)$-modules with $\rank N=1$.
\item The corresponding representations to $D^{8}_{rank{}N=1}$ are
     \begin{itemize}
     \item non-split reducible with submodule $E(\mflo\baseo+\mflt\baset,\baseth)$ if $s=2r-1$ and
     \item irreducible if $s<2r-1$.
     \end{itemize}
\item $D^{8}_{\rank  N=1}(\lambda,[\mflo:\mflt])$ is isomorphic to $D^{8}_{\rank  N=1}(\lambda',[\mflo':\mflt'])$ if and only if $\lambda=\lambda'$ and $[\mflo:\mflt]=[\mflo':\mflt']$.
\end{enumerate}
\end{prop}

\begin{proof}
From \ref{rank1third2}, we may let either $L_{1}=E(\baseo+a\baset+b\baseth)$ or $L_{1}=E(\baset+c\baseth)$ with $c\neq0$. If $L_{1}=E(\baseo+a\baset+b\baseth)$, then we may let $L_{2}=E(\baseo+a\baset+b\baseth,\baset+d\baseth)$ with $d\not=0$ and so, by change of a basis: $\baseo\mapsto\baseo+(b-ad)\baset$, $\baset\mapsto d\baset$, and $\baseth\mapsto\baseth$, we get $D^{8}_{\rank  N=1}$ for $[\mflo:\mflt]\neq[0:1]$. If $L_{1}=E(\baset+c\baseth)$ with $c\not=0$, then we may let $L_{2}=E(\baset+c\baseth,\baseo+f\baseth)$, and so, by change of a basis: $\baseo\mapsto\baseo+f\baset$, $\baset\mapsto c\baset$, and $\baseth\mapsto\baseth$, we get the other part of $D^{8}_{\rank  N=1}$. So now the part (1) and (2) are immediate from \ref{rank1third2}. For the part (3), use the same argument as in Proposition \ref{class rank1 D7}
\end{proof}

The following example arises from \ref{rank1fourth1}.
\begin{exam}
A filtered $(\phi,N)$-module of Hodge type $(0,r,s)$
$$D^{9}_{\rank  N=1}=D^{9}_{\rank  N=1}(\lambda,\mfl);$$
\begin{itemize}
\item $\fil{r} D=E(\baset+\mfl\baseth,\baseo)$ and $\fil{s} D=E(\baset+\mfl\baseth)$.
\item $
[N]=
\begin{small}\left(
  \begin{array}{ccc}
    0 & 0  & 0 \\
    0 & 0  & 0 \\
    1 & 0  & 0 \\
  \end{array}
\right)
\end{small}
$ and $
[\phi]=
\begin{small}\left(
  \begin{array}{ccc}
    p\lambda & 0  & 0 \\
    1 & p\lambda & 0 \\
    0 & 0 & \lambda \\
  \end{array}
\right)
\end{small}
$ for $\lambda$ in $E$.
\item $\mfl\in E^{\times}$, $\val\lambda=\frac{r+s-2}{3}$, and $s\leq 2r-1$.
\end{itemize}
\end{exam}

\begin{prop} \label{class rank1 D9}
\begin{enumerate}
\item $D^{9}_{\rank  N=1}$ represents admissible filtered $(\phi,N)$-modules with $\rank N=1$.
\item The corresponding representations to $D^{9}_{\rank N=1}$ are
     \begin{itemize}
     \item non-split reducible with submodule $E(\baset,\baseth)$ if $s=2r-1$ and
     \item irreducible if $s<2r-1$.
     \end{itemize}
\item $D^{9}_{\rank  N=1}(\lambda,\mfl)$ is isomorphic to $D^{9}_{\rank  N=1}(\lambda',\mfl')$ if and only if $\lambda=\lambda'$ and $\mfl=\mfl'$.
\end{enumerate}
\end{prop}

\begin{proof}
From \ref{rank1fourth1}, if we let $L_{1}=E(\baset+a\baseth)$ with $a\not=0$ and $L_{2}=E(\baset+a\baseth,\baseo+b\baseth)$, then, by change of a basis: $\baseo\mapsto\baseo+\frac{b}{a}\baset$, $\baset\mapsto\baset$, and $\baseth\mapsto\baseth$, we get $D^{9}_{\rank N=1}$. So now the part (1) and (2) are immediate from \ref{rank1fourth1}.

For the part (3), assume that $T$ is an isomorphism from $D^{9}_{\rank  N=1}(\lambda,\mfl)$ to $D^{9}_{\rank  N=1}(\lambda',\mfl')$. Then $\lambda=\lambda'$ and, by Lemma \ref{rank1typeofPfourth}, $[T]$ is a lower triangle matrix with $T_{1,1}=T_{2,2}=T_{3,3}$ and $T_{3,1}=0=T_{3,2}$. Since $T$ preserves the filtration, it should be a scalar multiple of the identity, which implies that $\mfl=\mfl'$. The converse is trivial.
\end{proof}

The following example arises from \ref{rank1fourth2}.
\begin{exam}
A filtered $(\phi,N)$-module of Hodge type $(0,r,s)$
$$D^{10}_{\rank  N=1}=D^{10}_{\rank  N=1}(\lambda,\mfl);$$
\begin{itemize}
\item $\fil{r} D=E(\baseo+\mfl\baseth,\baset)$ and $\fil{s} D=E(\baseo+\mfl\baseth)$.
\item $
[N]=
\begin{small}\left(
  \begin{array}{ccc}
    0 & 0  & 0 \\
    0 & 0  & 0 \\
    1 & 0  & 0 \\
  \end{array}
\right)
\end{small}
$ and $
[\phi]=
\begin{small}\left(
  \begin{array}{ccc}
    p\lambda & 0  & 0 \\
    1 & p\lambda & 0 \\
    0 & 0 & \lambda \\
  \end{array}
\right)
\end{small}
$ for $\lambda$ in $E$.
\item $\mfl\in E$, $\val\lambda=\frac{r+s-2}{3}$, and $s\geq 2r-1$.
\end{itemize}
\end{exam}

\begin{prop} \label{class rank1 D10}
\begin{enumerate}
\item $D^{10}_{\rank  N=1}$ represents admissible filtered $(\phi,N)$-modules with $\rank N=1$.
\item The corresponding representations to $D^{10}_{\rank N=1}$ are
     \begin{itemize}
     \item non-split reducible with submodule $E\baset$ if $s=2r-1$ and
     \item irreducible if $s>2r-1$.
     \end{itemize}
\item $D^{10}_{\rank  N=1}(\lambda,\mfl)$ is isomorphic to $D^{10}_{\rank  N=1}(\lambda',\mfl')$ if and only if $\lambda=\lambda'$ and $\mfl=\mfl'$.
\end{enumerate}
\end{prop}

\begin{proof}
From \ref{rank1fourth2}, if we let $L_{1}=E(\baseo+a\baset+b\baseth)$ and $L_{2}=E(\baseo+a\baset+b\baseth,\baset)$, then, by change of a basis: $\baseo\mapsto\baseo-a\baset$, $\baset\mapsto\baset$, and $\baseth\mapsto\baseth$, we get $D^{10}_{\rank N=1}$. So now the part (1) and (2) are immediate from \ref{rank1fourth2}. For the part (3), use the same argument as in Proposition \ref{class rank1 D9}.
\end{proof}

The following example arises from \ref{rank1fourth3}.
\begin{exam}
A filtered $(\phi,N)$-module of Hodge type $(0,r,s)$
$$D^{11}_{\rank  N=1}=D^{11}_{\rank  N=1}(\lambda,\mfl);$$
\begin{itemize}
\item $\fil{r} D=E(\baseo,\baseth)$ and $\fil{s} D=E(\baseo+\mfl\baseth)$.
\item $
[N]=
\begin{small}\left(
  \begin{array}{ccc}
    0 & 0  & 0 \\
    0 & 0  & 0 \\
    1 & 0  & 0 \\
  \end{array}
\right)
\end{small}
$ and $
[\phi]=
\begin{small}\left(
  \begin{array}{ccc}
    p\lambda & 0  & 0 \\
    1 & p\lambda & 0 \\
    0 & 0 & \lambda \\
  \end{array}
\right)
\end{small}
$ for $\lambda$ in $E$.
\item $\mfl\in E$, $\val\lambda=\frac{r+s-2}{3}$, and $s\geq 2r+2$.
\end{itemize}
\end{exam}

\begin{prop} \label{class rank1 D11}
\begin{enumerate}
\item $D^{11}_{\rank  N=1}$ represents admissible filtered $(\phi,N)$-modules with $\rank N=1$.
\item The corresponding representations to $D^{11}_{\rank N=1}$ are
     \begin{itemize}
     \item non-split reducible with submodule $E\baseth$ if $s=2r+2$ and
     \item irreducible if $s>2r+2$.
     \end{itemize}
\item $D^{11}_{\rank  N=1}(\lambda,\mfl)$ is isomorphic to $D^{11}_{\rank  N=1}(\lambda',\mfl')$ if and only if $\lambda=\lambda'$ and $\mfl=\mfl'$.
\end{enumerate}
\end{prop}

\begin{proof}
From \ref{rank1fourth3}, if we let $L_{1}=E(\baseo+a\baset+b\baseth)$ and $L_{2}=E(\baseo+a\baset+b\baseth,\baseth)$, then, by change of a basis: $\baseo\mapsto\baseo-a\baset$, $\baset\mapsto\baset$, and $\baseth\mapsto\baseth$, we get $D^{11}_{\rank N=1}$. So now the part (1) and (2) immediate from \ref{rank1fourth3}. For the part (3), use the same argument as in Proposition \ref{class rank1 D9}.
\end{proof}

The following example arises from \ref{rank1fourth4}.
\begin{exam}
A filtered $(\phi,N)$-module of Hodge type $(0,r,s)$
$$D^{12}_{\rank  N=1}=D^{12}_{\rank  N=1}(\lambda,\mflo,\mflt);$$
\begin{itemize}
\item $\fil{r} D=E(\baseo+\mflo\baseth,\baset+\mflt\baseth)$ and $\fil{s} D=E(\baseo+\mflo\baseth)$.
\item $
[N]=
\begin{small}\left(
  \begin{array}{ccc}
    0 & 0  & 0 \\
    0 & 0  & 0 \\
    1 & 0  & 0 \\
  \end{array}
\right)
\end{small}
$ and $
[\phi]=
\begin{small}\left(
  \begin{array}{ccc}
    p\lambda & 0  & 0 \\
    1 & p\lambda & 0 \\
    0 & 0 & \lambda \\
  \end{array}
\right)
\end{small}
$ for $\lambda$ in $E$.
\item $\mflo\in E$, $\mflt\in E^{\times}$, and $\val\lambda=\frac{r+s-2}{3}$.
\end{itemize}
\end{exam}

\begin{prop} \label{class rank1 D12}
\begin{enumerate}
\item $D^{12}_{\rank  N=1}$ represents admissible filtered $(\phi,N)$-modules with $\rank N=1$.
\item The corresponding representations to $D^{12}_{\rank N=1}$ are irreducible.
\item $D^{12}_{\rank  N=1}(\lambda,\mflo,\mflt)$ is isomorphic to $D^{10}_{\rank  N=1}(\lambda',\mflo',\mflt')$ if and only if $\lambda=\lambda'$, $\mflo=\mflo'$, and $\mflt=\mflt'$.
\end{enumerate}
\end{prop}

\begin{proof}
From \ref{rank1fourth4}, if we let $L_{1}=E(\baseo+a\baset+b\baseth)$ and $L_{2}=E(\baseo+a\baset+b\baseth,\baset+c\baseth)$ with $c\not=0$, then, by change of a basis: $\baseo\mapsto\baseo-a\baset$, $\baset\mapsto\baset$, and $\baseth\mapsto\baseth$, we get $D^{12}_{\rank N=1}$. So now the part (1) and (2) are immediate from \ref{rank1fourth4}. For the part (3), use the same argument as in Proposition \ref{class rank1 D9}.
\end{proof}

The following example arises from \ref{rank1fifth1}.
\begin{exam}
A filtered $(\phi,N)$-module of Hodge type $(0,r,s)$
$$D^{13}_{\rank  N=1}=D^{13}_{\rank  N=1}(\lambda,\lambdat,\mathfrak{L});$$
\begin{itemize}
\item $\fil{r} D=E(\baset,\baseo+\mfl\baseth)$ and $\fil{s} D=E\baset$.
\item $
[N]=
\begin{small}\left(
  \begin{array}{ccc}
    0 & 0  & 0 \\
    0 & 0  & 0 \\
    1 & 0  & 0 \\
  \end{array}
\right)
\end{small}
$ and $
[\phi]=
\begin{small}\left(
  \begin{array}{ccc}
    p\lambda & 0  & 0 \\
    0 & \lambdat & 0 \\
    0 & 0 & \lambda \\
  \end{array}
\right)
\end{small}
$ for $\lambda\not=\lambdat\not=p\lambda$ in $E$.
\item $\mfl\in E$, $\val\lambda=\frac{r-1}{2}$, and $\val\lambdat=s$.
\end{itemize}
\end{exam}

\begin{prop} \label{class rank1 D13}
\begin{enumerate}
\item $D^{13}_{\rank  N=1}$ represents admissible filtered $(\phi,N)$-modules with $\rank N=1$.
\item The corresponding representations to $D^{13}_{\rank N=1}$ are decomposable with submodules $E\baset$ and $E(\baseo,\baseth)$; moreover, $E\baseth$ and $E(\baset,\baseth)$ are submodules as well if $r=1$.
\item $D^{13}_{\rank  N=1}(\lambda,\lambdat,\mfl)$ is isomorphic to $D^{13}_{\rank  N=1}(\lambda',\lambdat',\mfl')$ if and only if $\lambda=\lambda'$, $\lambdat=\lambdat'$, and $\mfl=\mfl'$.
\end{enumerate}
\end{prop}

\begin{proof}
From \ref{rank1fifth1}, if we let $L_{1}=E\baset$ and $L_{2}=E(\baset,\baseo+a\baseth)$, then we get $D^{13}_{\rank N=1}$. So now the part (1) and (2) are immediate from \ref{rank1fifth1}.

For the part (3), assume that $T$ is an isomorphism from $D^{13}_{\rank  N=1}(\lambda,\lambdat,\mfl)$ to $D^{13}_{\rank  N=1}(\lambda',\lambdat',\mfl)$. Then $\lambda=\lambda'$, $\lambdat=\lambdat'$, and, by Lemma \ref{rank1typeofPfifth}, $[T]$ is a diagonal matrix with $T_{1,1}=T_{3,3}$. Since $T$ preserves the filtration, it should be a scalar multiple of the identity, which implies that $\mfl=\mfl'$. The converse is trivial.
\end{proof}

The following example arises from \ref{rank1fifth2}.
\begin{exam}
A filtered $(\phi,N)$-module of Hodge type $(0,r,s)$
$$D^{14}_{\rank  N=1}=D^{14}_{\rank  N=1}(\lambda,\lambdat,\mathfrak{L});$$
\begin{itemize}
\item $\fil{r} D=E(\baseo,\baseth)$ and $\fil{s} D=E(\baseo+\mfl\baseth)$.
\item $
[N]=
\begin{small}\left(
  \begin{array}{ccc}
    0 & 0  & 0 \\
    0 & 0  & 0 \\
    1 & 0  & 0 \\
  \end{array}
\right)
\end{small}
$ and $
[\phi]=
\begin{small}\left(
  \begin{array}{ccc}
    p\lambda & 0  & 0 \\
    0 & \lambdat & 0 \\
    0 & 0 & \lambda \\
  \end{array}
\right)
\end{small}
$ for $\lambda\not=\lambdat\not=p\lambda$ in $E$.
\item $\mfl\in E$, $\val\lambda=\frac{r+s-1}{2}$, and $\val\lambdat=0$.
\end{itemize}
\end{exam}

\begin{prop} \label{class rank1 D14}
\begin{enumerate}
\item $D^{14}_{\rank  N=1}$ represents admissible filtered $(\phi,N)$-modules with $\rank N=1$.
\item The corresponding representations to $D^{14}_{\rank N=1}$ are decomposable with submodules $E\baset$ and $E(\baseo,\baseth)$; moreover, $E\baseth$ and $E(\baset,\baseth)$ are submodules as well if $s=r+1$.
\item $D^{14}_{\rank  N=1}(\lambda,\lambdat,\mfl)$ is isomorphic to $D^{14}_{\rank  N=1}(\lambda',\lambdat',\mfl')$ if and only if $\lambda=\lambda'$, $\lambdat=\lambdat'$, and $\mfl=\mfl'$.
\end{enumerate}
\end{prop}

\begin{proof}
From \ref{rank1fifth2}, if we let $L_{1}=E(\baseo+a\baseth)$ and $L_{2}=E(\baseo,\baseth)$, then we get $D^{14}_{\rank N=1}$. So now the part (1) and (2) are immediate from \ref{rank1fifth2}. For the part (3), use the same argument in Proposition \ref{class rank1 D13}.
\end{proof}

The following example arises from \ref{rank1fifth3}.
\begin{exam}
A filtered $(\phi,N)$-module of Hodge type $(0,r,s)$
$$D^{15}_{\rank  N=1}=D^{15}_{\rank  N=1}(\lambda,\lambdat,\mfl);$$
\begin{itemize}
\item $\fil{r} D=E(\baset+\baseth,\baseo+\mfl\baseth)$ and $\fil{s} D=E(\baset+\baseth)$.
\item $
[N]=
\begin{small}\left(
  \begin{array}{ccc}
    0 & 0  & 0 \\
    0 & 0  & 0 \\
    1 & 0  & 0 \\
  \end{array}
\right)
\end{small}
$ and $
[\phi]=
\begin{small}\left(
  \begin{array}{ccc}
    p\lambda & 0  & 0 \\
    0 & \lambdat & 0 \\
    0 & 0 & \lambda \\
  \end{array}
\right)
\end{small}
$ for $\lambda\not=\lambdat\not=p\lambda$ in $E$.
\item $\mathfrak{L}\in E$, $\frac{r-1}{2}\leq\val\lambda\leq r-1$, and $2\val\lambda+\val\lambdat=r+s-1$.
\end{itemize}
\end{exam}

\begin{prop} \label{class rank1 D15}
\begin{enumerate}
\item $D^{15}_{\rank  N=1}$ represents admissible filtered $(\phi,N)$-modules with $\rank N=1$.
\item The corresponding representations to $D^{15}_{\rank N=1}$ are
     \begin{itemize}
     \item non-split reducible with submodules $E\baseth$, $E(\baset,\baseth)$, and $E(\baseo,\baseth)$ if $r=1$,
     \item non-split reducible with submodule $E(\baseo,\baseth)$ if $\val\lambda=\frac{r-1}{2}$ and $r>1$,
     \item non-split reducible with submodule $E(\baset,\baseth)$ if $\val\lambda=r-1$ and $r>1$, and
     \item irreducible if $\frac{r-1}{2}<\val\lambda<r-1$ and $r>1$.
     \end{itemize}
\item $D^{15}_{\rank  N=1}(\lambda,\lambdat,\mfl)$ is isomorphic to $D^{15}_{\rank  N=1}(\lambda',\lambdat',\mfl')$ if and only if $\lambda=\lambda'$, $\lambdat=\lambdat'$, and $\mfl=\mfl'$.
\end{enumerate}
\end{prop}

\begin{proof}
From \ref{rank1fifth3}, if we let $L_{1}=E(\baset+a\baseth)$ and $L_{2}=E(\baset+a\baseth,\baseo+b\baseth)$ with $a\not=0$, then, by change of a basis: $\baseo\mapsto\baseo$, $\baset\mapsto a\baset$, and $\baseth\mapsto\baseth$, we get $D^{15}_{\rank N=1}$. So now the part (1) and (2) are immediate from \ref{rank1fifth3}. For the part (3), use the same argument in Proposition \ref{class rank1 D13}.
\end{proof}

The following example arises from \ref{rank1fifth4}.
\begin{exam}
A filtered $(\phi,N)$-module of Hodge type $(0,r,s)$
$$D^{16}_{\rank  N=1}=D^{16}_{\rank  N=1}(\lambda,\lambdat,\mathfrak{L});$$
\begin{itemize}
\item $\fil{r} D=E(\baseo+\mfl\baseth,\baset)$ and $\fil{s} D=E(\baseo+\mfl\baseth)$.
\item $
[N]=
\begin{small}\left(
  \begin{array}{ccc}
    0 & 0  & 0 \\
    0 & 0  & 0 \\
    1 & 0  & 0 \\
  \end{array}
\right)
\end{small}
$ and $
[\phi]=
\begin{small}\left(
  \begin{array}{ccc}
    p\lambda & 0  & 0 \\
    0 & \lambdat & 0 \\
    0 & 0 & \lambda \\
  \end{array}
\right)
\end{small}
$ for $\lambda\not=\lambdat\not=p\lambda$ in $E$.
\item $\mfl\in E$, $\val\lambda=\frac{s-1}{2}$, and $\val\lambdat=r$.
\end{itemize}
\end{exam}

\begin{prop} \label{class rank1 D16}
\begin{enumerate}
\item $D^{16}_{\rank  N=1}$ represents admissible filtered $(\phi,N)$-modules with $\rank N=1$.
\item The corresponding representations to $D^{16}_{\rank N=1}$ are decomposable with submodules $E\baset$ and $E(\baseo,\baseth)$.
\item $D^{16}_{\rank  N=1}(\lambda,\lambdat,\mfl)$ is isomorphic to $D^{16}_{\rank  N=1}(\lambda',\lambdat',\mfl')$ if and only if $\lambda=\lambda'$, $\lambdat=\lambdat'$, and $\mfl=\mfl'$.
\end{enumerate}
\end{prop}

\begin{proof}
From \ref{rank1fifth4}, if we let $L_{1}=E(\baseo+a\baseth)$ and $L_{2}=E(\baseo+a\baseth,\baset)$, then we get $D^{16}_{\rank N=1}$. So now the part (1) and (2) are immediate from \ref{rank1fifth4}. For the part (3), use the same argument in Proposition \ref{class rank1 D13}.
\end{proof}

The following example arises from \ref{rank1fifth5}.
\begin{exam}
A filtered $(\phi,N)$-module of Hodge type $(0,r,s)$
$$D^{17}_{\rank  N=1}=D^{17}_{\rank  N=1}(\lambda,\lambdat,\mfl);$$
\begin{itemize}
\item $\fil{r} D=E(\baseo+\mfl\baseth,\baset+\baseth)$ and $\fil{s} D=E(\baseo+\mfl\baseth)$.
\item $
[N]=
\begin{small}\left(
  \begin{array}{ccc}
    0 & 0  & 0 \\
    0 & 0  & 0 \\
    1 & 0  & 0 \\
  \end{array}
\right)
\end{small}
$ and $
[\phi]=
\begin{small}\left(
  \begin{array}{ccc}
    p\lambda & 0  & 0 \\
    0 & \lambdat & 0 \\
    0 & 0 & \lambda \\
  \end{array}
\right)
\end{small}
$ for $\lambda\not=\lambdat\not=p\lambda$ in $E$.
\item $\mathfrak{L}\in E$, $\frac{s-1}{2}\leq\val\lambda\leq\frac{r+s-1}{2}$, and $2\val\lambda+\val\lambdat=r+s-1$.
\end{itemize}
\end{exam}

\begin{prop} \label{class rank1 D17}
\begin{enumerate}
\item $D^{17}_{\rank  N=1}$ represents admissible filtered $(\phi,N)$-modules with $\rank N=1$.
\item The corresponding representations to $D^{17}_{\rank N=1}$ are
     \begin{itemize}
     \item non-split reducible with submodules $E(\baseo,\baseth)$ if $\val\lambda=\frac{s-1}{2}$,
     \item non-split reducible with submodule $E\baset$ and $E(\baset,\baseth)$ if $\val\lambda=\frac{r+s-1}{2}$ and $s=r+1$,
     \item non-split reducible with submodule $E\baset$ if $\val\lambda=\frac{r+s-1}{2}$ and $s>r+1$, and
     \item irreducible if $\frac{s-1}{2}<\val\lambda<\frac{r+s-1}{2}$.
     \end{itemize}
\item $D^{17}_{\rank  N=1}(\lambda,\lambdat,\mfl)$ is isomorphic to $D^{17}_{\rank  N=1}(\lambda',\lambdat',\mfl')$ if and only if $\lambda=\lambda'$, $\lambdat=\lambdat'$, and $\mfl=\mfl'$.
\end{enumerate}
\end{prop}

\begin{proof}
From \ref{rank1fifth5}, if we let $L_{1}=E(\baseo+a\baseth)$ and $L_{2}=E(\baseo+a\baseth,\baset+b\baseth)$ with $b\not=0$, then, by change of a basis: $\baseo\mapsto\baseo$, $\baset\mapsto b\baset$, and $\baseth\mapsto\baseth$, we get $D^{17}_{\rank N=1}$. So now the part (1) and (2) are immediate from \ref{rank1fifth5}. For the part (3), use the same argument as in Proposition \ref{class rank1 D13}.
\end{proof}

The following example arises from \ref{rank1fifth6}.
\begin{exam}
A filtered $(\phi,N)$-module of Hodge type $(0,r,s)$
$$D^{18}_{\rank  N=1}=D^{18}_{\rank  N=1}(\lambda,\lambdat,\mfl);$$
\begin{itemize}
\item $\fil{r} D=E(\baseo+\mfl\baseth,\baset)$ and $\fil{s} D=E(\baseo+\baset+\mfl\baseth)$.
\item $
[N]=
\begin{small}\left(
  \begin{array}{ccc}
    0 & 0  & 0 \\
    0 & 0  & 0 \\
    1 & 0  & 0 \\
  \end{array}
\right)
\end{small}
$ and $
[\phi]=
\begin{small}\left(
  \begin{array}{ccc}
    p\lambda & 0  & 0 \\
    0 & \lambdat & 0 \\
    0 & 0 & \lambda \\
  \end{array}
\right)
\end{small}
$ for $\lambda\not=\lambdat\not=p\lambda$ in $E$.
\item $\mfl\in E$, $\frac{r-1}{2}\leq\val\lambda\leq\frac{s-1}{2}$, and $2\val\lambda+\val\lambdat=r+s-1$.
\end{itemize}
\end{exam}

\begin{prop} \label{class rank1 D18}
\begin{enumerate}
\item $D^{18}_{\rank  N=1}$ represents admissible filtered $(\phi,N)$-modules with $\rank N=1$.
\item The corresponding representations to $D^{18}_{\rank N=1}$ are
     \begin{itemize}
     \item non-split reducible with submodules $E\baseth$ and $E(\baseo,\baseth)$ if $\val\lambda=\frac{r-1}{2}$ and $r=1$,
     \item non-split reducible with submodule $E(\baseo,\baseth)$ if $\val\lambda=\frac{r-1}{2}$ and $r>1$,
     \item non-split reducible with submodule $E\baset$ if $\val\lambda=\frac{s-1}{2}$, and
     \item irreducible if $\frac{r-1}{2}<\val\lambda<\frac{s-1}{2}$.
     \end{itemize}
\item $D^{18}_{\rank  N=1}(\lambda,\lambdat,\mfl)$ is isomorphic to $D^{18}_{\rank  N=1}(\lambda',\lambdat',\mfl')$ if and only if $\lambda=\lambda'$, $\lambdat=\lambdat'$, and $\mfl=\mfl'$.
\end{enumerate}
\end{prop}

\begin{proof}
From \ref{rank1fifth6}, if we let $L_{1}=E(\baseo+a\baset+b\baseth)$ and $L_{2}=E(\baseo+a\baset+b\baseth,\baset)$ with $a\not=0$, then, by change of a basis: $\baseo\mapsto\baseo$, $\baset\mapsto\frac{1}{a}\baset$, and $\baseth\mapsto\baseth$, we get $D^{18}_{\rank N=1}$. So now the part (1) and (2) are immediate from \ref{rank1fifth6}. For the part (3), use the same argument as in Proposition \ref{class rank1 D13}.
\end{proof}

The following example arises from \ref{rank1fifth7}.
\begin{exam}
A filtered $(\phi,N)$-module of Hodge type $(0,r,s)$
$$D^{19}_{\rank  N=1}=D^{19}_{\rank  N=1}(\lambda,\lambdat,\mfl);$$
\begin{itemize}
\item $\fil{r} D=E(\baseo+\baset,\baseth)$ and $\fil{s} D=E(\baseo+\baset+\mfl\baseth)$.
\item $
[N]=
\begin{small}\left(
  \begin{array}{ccc}
    0 & 0  & 0 \\
    0 & 0  & 0 \\
    1 & 0  & 0 \\
  \end{array}
\right)
\end{small}
$ and $
[\phi]=
\begin{small}\left(
  \begin{array}{ccc}
    p\lambda & 0  & 0 \\
    0 & \lambdat & 0 \\
    0 & 0 & \lambda \\
  \end{array}
\right)
\end{small}
$ for $\lambda\not=\lambdat\not=p\lambda$ in $E$.
\item $\mfl\in E$, $r\leq\val\lambda\leq\frac{r+s-1}{2}$, $2\val\lambda+\val\lambdat=r+s-1$, and $s\geq r+1$.
\end{itemize}
\end{exam}

\begin{prop} \label{class rank1 D19}
\begin{enumerate}
\item $D^{19}_{\rank  N=1}$ represents admissible filtered $(\phi,N)$-modules with $\rank N=1$.
\item The corresponding representations to $D^{19}_{\rank N=1}$ are
     \begin{itemize}
     \item non-split reducible with submodules $E\baset$, $E\baseth$, and $E(\baset,\baseth)$ if $s=r+1$,
     \item non-split reducible with submodule $E\baseth$ if $\val\lambda=r$ and $s>r+1$,
     \item non-split reducible with submodule $E\baset$ if $\val\lambda=\frac{r+s-1}{2}$ and $s>r+1$, and
     \item irreducible if $r<\val\lambda<\frac{r+s-1}{2}$ and $s>r+1$.
     \end{itemize}
\item $D^{19}_{\rank  N=1}(\lambda,\lambdat,\mfl)$ is isomorphic to $D^{19}_{\rank  N=1}(\lambda',\lambdat',\mfl')$ if and only if $\lambda=\lambda'$, $\lambdat=\lambdat'$, and $\mfl=\mfl'$.
\end{enumerate}
\end{prop}

\begin{proof}
From \ref{rank1fifth7}, if we let $L_{1}=E(\baseo+a\baset+b\baseth)$ and $L_{2}=E(\baseo+a\baset+b\baseth,\baseth)$ with $a\not=0$, then, by change of a basis: $\baseo\mapsto\baseo$, $\baset\mapsto\frac{1}{a}\baset$, and $\baseth\mapsto\baseth$, we get $D^{19}_{\rank N=1}$. So now the part (1) and (2) are immediate from \ref{rank1fifth7}. For the part (3), use the same argument as in Proposition \ref{class rank1 D13}.
\end{proof}

The following example arises from \ref{rank1fifth8}.
\begin{exam}
A filtered $(\phi,N)$-module  of Hodge type $(0,r,s)$
$$D^{20}_{\rank  N=1}=D^{20}_{\rank  N=1}(\lambda,\lambdat,\mflo,\mflt);$$
\begin{itemize}
\item $\fil{r} D=E(\baseo+\baset+\mflo\baseth,\baset+\mflt\baseth)$ and $\fil{s} D=E(\baseo+\baset+\mflo\baseth)$.
\item $
[N]=
\begin{small}\left(
  \begin{array}{ccc}
    0 & 0  & 0 \\
    0 & 0  & 0 \\
    1 & 0  & 0 \\
  \end{array}
\right)
\end{small}
$ and $
[\phi]=
\begin{small}\left(
  \begin{array}{ccc}
    p\lambda & 0  & 0 \\
    0 & \lambdat & 0 \\
    0 & 0 & \lambda \\
  \end{array}
\right)
\end{small}
$ for $\lambda\not=\lambdat\not=p\lambda$ in $E$.
\item $\mflo\in E$, $\mflt\in E^{\times}$, $\frac{r-1}{2}\leq\val\lambda\leq\frac{r+s-1}{2}$, and $2\val\lambda+\val\lambdat=r+s-1$.
\end{itemize}
\end{exam}

\begin{prop} \label{class rank1 D20}
\begin{enumerate}
\item $D^{20}_{\rank  N=1}$ represents admissible filtered $(\phi,N)$-modules with $\rank N=1$.
\item The corresponding representations to $D^{20}_{\rank N=1}$ are
     \begin{itemize}
     \item non-split reducible with submodules $E\baseth$ and $E(\baseo,\baseth)$ if $\val\lambda=\frac{r-1}{2}$ and $r=1$,
     \item non-split reducible with submodule $E(\baseo,\baseth)$ if $\val\lambda=\frac{r-1}{2}$ and $r>1$,
     \item non-split reducible with submodule $E\baset$ and $E(\baset,\baseth)$ if $\val\lambda=\frac{r+s-1}{2}$ and $s=r+1$,
     \item non-split reducible with submodule $E\baset$ if $\val\lambda=\frac{r+s-1}{2}$ and $s>r+1$, and
     \item irreducible if $\frac{r-1}{2}<\val\lambda<\frac{r+s-1}{2}$.
     \end{itemize}
\item $D^{20}_{\rank  N=1}(\lambda,\lambdat,\mflo,\mflt)$ is isomorphic to $D^{20}_{\rank  N=1}(\lambda',\lambdat',\mflo',\mflt')$ if and only if $\lambda=\lambda'$, $\lambdat=\lambdat'$, $\mflo=\mflo'$, and $\mflt=\mflt'$.
\end{enumerate}
\end{prop}

\begin{proof}
From \ref{rank1fifth8}, if we let $L_{1}=E(\baseo+a\baset+b\baseth)$ and $L_{2}=E(\baseo+a\baset+b\baseth,\baset+c\baseth)$ with $ca\not=0$ then, by change of a basis: $\baseo\mapsto\baseo$, $\baset\mapsto\frac{1}{a}\baset$, and $\baseth\mapsto\baseth$, we get $D^{20}_{\rank N=1}$. So now the part (1) and (2) are immediate from \ref{rank1fifth8}. For the part (3), use the same argument as in Proposition \ref{class rank1 D13}.
\end{proof}

\begin{prop}
Every semi-stable representation of $G_{\QP}$ with regular Hodge--Tate weights and with $\rank N=1$ is isomorphic to a representation corresponding to some $D^{i}_{\rank N=1}$ up to twist by a power of the cyclotomic character.
\end{prop}

\begin{proof}
We found all the admissible filtered $(\phi,N)$-modules with $\rank N=1$ in the previous subsections. Since the list of filtered modules in this subsection represents all of the modules in the previous subsections, we are done.
\end{proof}

\begin{prop}
Let $i,j\in\{1,2,...,20\}$. If $D^{i}_{\rank N=1}$ is isomorphic to $D^{j}_{\rank N=1}$, then $i=j$.
\end{prop}

\begin{proof}
Since an isomorphism preserves the Jordan forms of the Frobenius map, if $D^{i}_{\rank N=1}$ is isomorphic to $D^{j}_{\rank N=1}$, then both $i$ and $j$ belong to either $\{1,2\}$, $[3,6]$, $\{7,8\}$, $[9,12]$, or $[13,20]$.

$D^{1}_{\rank N=1}$ is not isomorphic to $D^{2}_{\rank N=1}$, since any invertible linear map of the form in Lemma \ref{rank1typeofPfirst} can not preserve $\fils D$. Likewise, $D^{7}_{\rank N=1}$ is not isomorphic to $D^{8}_{\rank N=1}$, since any invertible linear map of the form in Lemma \ref{rank1typeofPthird} can not preserve the filtration.

Let $i\neq j\in[3,6]$. Then $D^{i}_{\rank N=1}$ is not isomorphic to $D^{j}_{\rank N=1}$, since any invertible linear map of the form in Lemma \ref{rank1typeofPsecond} can not preserve the filtration. Likewise, for any $i\neq j\in [9,12]$, $D^{i}_{\rank N=1}$ is not isomorphic to $D^{j}_{\rank N=1}$, since any invertible linear map of the form in Lemma \ref{rank1typeofPfourth} can not preserve the filtration.

Let $i\neq j\in[13,20]$. By the same argument, any invertible linear map of the form in Lemma \ref{rank1typeofPfifth} can not preserve the filtration of $D^{i}_{\rank N=1}$ to the one of $D^{j}_{\rank N=1}$. Hence, if $i\neq j$ there are no isomorphism between $D^{i}_{\rank N=1}$ and $D^{j}_{\rank N=1}$.
\end{proof}

\section{Admissible filtered $(\phi,N)$-modules with $\rank N=2$}
In this section, we classify the admissible filtered $(\phi,N)$-modules of Hodge type $(0,r,s)$ for $0<r<s$ and with the monodromy operator $N$ of rank $2$. We follow exactly the same argument in the previous section. Assume first that $N$ has rank $2$. By choice of a basis for $D=E(\baseo,\baset,\baseth)$, we may set
$$N\baseo=\baset,\hspace{0.1cm}N\baset=\baseth,\mbox{ and }N\baseth=0.$$ From the equation $N\phi=p\phi N$, we should have that $$\phi\baseo=p^{2}x\baseo+py\baset+z\baseth,\hspace{0.2cm}\phi\baset=px\baset+y\baseth,\mbox{ and }\phi\baseth=x\baseth,$$ for some $y$, $z$, and $x\not=0$. By change of a basis, we can say a bit more.

\begin{lemm}
There is an invertible matrix $P$ such that
$$P[N]P^{-1}=[N]\mbox{ and }P[\phi]P^{-1}=
\begin{tiny}\left(
\begin{array}{ccc}
    p^{2}x &  0 & 0 \\
    0 &  px & 0 \\
    0 &  0 & x \\
\end{array}
\right)
\end{tiny}.
$$
\end{lemm}

\begin{proof}
Use the matrix $P=
\begin{tiny}\left(
\begin{array}{ccc}
    1 &  0 & 0 \\
    \frac{y}{x(1-p)} &  1 & 0 \\
    \frac{py^{2}+zx(1-p)}{x^{2}(1-p)(1-p^{2})} &  \frac{y}{x(1-p)} & 1 \\
\end{array}
\right)
\end{tiny}
$.
\end{proof}

In this section, we assume
$$
[\phi]=
\begin{small}
\left(
\begin{array}{ccc}
    p^{2}\lambda &  0 & 0 \\
    0 &  p\lambda & 0 \\
    0 &  0 & \lambda \\
\end{array}
\right)
\end{small}
\mbox{ and }
[N]=
\begin{small}\left(
\begin{array}{ccc}
    0 &  0 & 0 \\
    1 &  0 & 0 \\
    0 &  1 & 0 \\
\end{array}
\right)
\end{small}.$$
By admissibility, we have $$\val\lambda=\frac{r+s-3}{3}.$$

\begin{lemm}\label{rank2typeofP}
For a $3\times 3$-matrix $P=(P_{i,j})$, $P[\phi]=[\phi]P$ and $P[N]=[N]P$ if and only if $P$ is a scalar multiple of the identity.
\end{lemm}

\begin{proof}
The equation $P[N]=[N]P$ forces that $P$ be a lower triangle matrix with $P_{1,1}=P_{2,2}=P_{3,3}$ and with $P_{2,1}=P_{3,2}$. Then $P[\phi]=[\phi]P$ forces that $P$ be a scalar multiple of the identity.
\end{proof}

\subsection{The only case of $\rank N=2$}
In this subsection, we collect the admissible filtered $(\phi,N)$-modules with $\rank N=2$.
\begin{lemm} \label{rank2invariant}
$E\baseth$ and $E(\baset,\baseth)$ are the only nontrivial proper $\phi$- and $N$-invariant subspaces.
\end{lemm}

\begin{proof}
By Lemma \ref{crysinvariantsixth}, we know that the only nontrivial proper $\phi$-invariant subspaces of $D$ are $E\baseo$, $E\baset$, $E\baseth$, $E(\baseo,\baset)$, $E(\baset,\baseth)$, and $E(\baseo,\baseth)$. Now it is easy to check which ones are $N$-invariant among these subspaces.
\end{proof}

We start to collect the admissible filtered $(\phi,N)$-modules in this case.
\subsubsection{}
Assume that $L_{1}=E\baseth$. Then, by admissibility, we have $s=\hodgen(E\baseth)\leq\newtonn(E\baseth)=\val\lambda$, which contradicts to $\val\lambda=\frac{r+s-3}{3}$.

\subsubsection{}
Assume that $L_{2}=E(\baset,\baseth)$. Then, by admissibility, we have $r+s=\hodgen(E(\baset,\baseth))\leq\newtonn(E(\baset,\baseth))=2\val\lambda+1$, which also contradicts to $\val\lambda=\frac{r+s-3}{3}$.

\subsubsection{} \label{rank2first1}
Assume that neither $L_{1}$ nor $L_{2}$ are $\phi$- and $N$-invariant and $L_{1}\subset E(\baset,\baseth)$. Then $\baseth\not\in L_{2}$ and, by admissibility, $0=\hodgen(E\baseth)\leq\newtonn(E\baseth)=\val\lambda$ and $s=\hodgen(E(\baset,\baseth))\leq\newtonn(E(\baset,\baseth))=2\val\lambda+1.$ Hence, there exist admissible filtered $(\phi,N)$-modules if and only if $s\leq 2r-3$. The corresponding representations are
\begin{itemize}
\item non-split reducible with submodule $E(\baset,\baseth)$ if $s=2r-3$, and
\item irreducible if $s<2r-3$.
\end{itemize}

\subsubsection{}  \label{rank2first2}
Assume that neither $L_{1}$ nor $L_{2}$ are $\phi$- and $N$-invariant and $\baseth\in L_{2}$. Then $L_{1}\not\subset E(\baset,\baseth)$ and, by admissibility, $r=\hodgen(E\baseth)\leq\newtonn(E\baseth)=\val\lambda$ and $r=\hodgen(E(\baset,\baseth))\leq\newtonn(E(\baset,\baseth))=2\val\lambda+1$. Hence, there exist admissible filtered $(\phi,N)$-modules if and only if $s\geq 2r+3$.  The corresponding representations are
\begin{itemize}
\item non-split reducible with submodule $E\baseth$ if $s=2r+3$, and
\item irreducible if $s>2r+3$.
\end{itemize}

\subsubsection{}  \label{rank2first3}
Assume that neither $L_{1}$ nor $L_{2}$ are $\phi$- and $N$-invariant, $L_{1}\not\subset E(\baset,\baseth)$, and $\baseth\not\in L_{2}$. Then, by admissibility, $0=\hodgen(E\baseth)\leq\newtonn(E\baseth)=\val\lambda$ and $r=\hodgen(E(\baset,\baseth))\leq\newtonn(E(\baset,\baseth))=2\val\lambda+1$. Hence, there exist admissible filtered $(\phi,N)$-modules if and only if $2s\geq r+3$. The corresponding representations are
\begin{itemize}
\item non-split reducible with submodules $E\baseth$ and $E(\baset,\baseth)$ if $2s=r+3$, or equivalently, if $r=1$ and $s=2$, and
\item irreducible if $2s>r+3$, or equivalently, if $s>2$.
\end{itemize}

\subsection{List of the isomorphism classes with $\rank N=2$} \label{ssec:list of N=2}
In the previous subsection, we found all of the admissible filtered $(\phi,N)$-modules of Hodge type $(0,r,s)$ for $0<r<s$ with $\rank N=2$. In this subsection, we classify the isomorphism classes of the admissible filtered $(\phi,N)$-modules on $D=E(\baseo,\baset,\baseth)$.

The following example arises from \ref{rank2first1}.
\begin{exam}
A filtered $(\phi,N)$-module of Hodge type $(0,r,s)$
$$D^{1}_{\rank  N=2}=D^{1}_{\rank  N=2}(\lambda,\mflo,\mflt);$$
\begin{itemize}
\item $\fil{r} D=E(\baset+\mflo\baseth,\baseo+\mflt\baseth)$ and $\fil{s} D=E(\baset+\mflo\baseth)$.
\item $
[N]=
\begin{small}\left(
  \begin{array}{ccc}
    0 & 0  & 0 \\
    1 & 0  & 0 \\
    0 & 1  & 0 \\
  \end{array}
\right)
\end{small}
$ and $
[\phi]=
\begin{small}\left(
  \begin{array}{ccc}
    p^{2}\lambda & 0  & 0 \\
    0 & p\lambda & 0 \\
    0 & 0 & \lambda \\
  \end{array}
\right)
\end{small}
$ for $\lambda$ in $E$.
\item $\mathfrak{L}_{1},\mathfrak{L}_{2}\in E$, $\val\lambda=\frac{r+s-3}{3}$, and $s\leq2r-3$.
\end{itemize}
\end{exam}

\begin{prop} \label{class rank2 D1}
\begin{enumerate}
\item $D^{1}_{\rank N=2}$ represents admissible filtered $(\phi,N)$-modules with $\rank N=2$.
\item The corresponding representations to $D^{1}_{\rank N=2}$ are
     \begin{itemize}
     \item non-split reducible with submodule $E(\baset,\baseth)$ if $s=2r-3$ and
     \item irreducible if $s<2r-3$.
     \end{itemize}
\item $D^{1}_{\rank  N=2}(\lambda,\mflo,\mflt)$ is isomorphic to $D^{1}_{\rank  N=2}(\lambda',\mflo',\mflt')$ if and only if $\lambda=\lambda'$, $\mflo=\mflo'$, and $\mflt=\mflt'$.
\end{enumerate}
\end{prop}

\begin{proof}
The part (1) and (2) are immediate from \ref{rank2first1}. For the part (3), assume that $T$ is an isomorphism from $D^{1}_{\rank  N=2}(\lambda,\mflo,\mflt)$ to $D^{1}_{\rank  N=2}(\lambda',\mflo',\mflt')$. Clearly, $\lambda=\lambda'$, and we know that $T$ is a scalar multiple of the identity by lemma \ref{rank2typeofP}. Since $T$ preserves the filtration, it is easy to check that $\mflo=\mflo'$ and $\mflt=\mflt'$. The converse is clear.
\end{proof}

The following example arises from \ref{rank2first2}.
\begin{exam}
A filtered $(\phi,N)$-module of Hodge type $(0,r,s)$
$$D^{2}_{\rank  N=2}=D^{2}_{\rank  N=2}(\lambda,\mflo,\mflt);$$
\begin{itemize}
\item $\fil{r} D=E(\baseo+\mflo\baset+\mflt\baseth,\baseth)$ and $\fil{s} D=E(\baseo+\mflo\baset+\mflt\baseth)$.
\item $
[N]=
\begin{small}\left(
  \begin{array}{ccc}
    0 & 0  & 0 \\
    1 & 0  & 0 \\
    0 & 1  & 0 \\
  \end{array}
\right)
\end{small}
$ and $
[\phi]=
\begin{small}\left(
  \begin{array}{ccc}
    p^{2}\lambda & 0  & 0 \\
    0 & p\lambda & 0 \\
    0 & 0 & \lambda \\
  \end{array}
\right)
\end{small}
$ for $\lambda$ in $E$.
\item $\mflo,\mflt\in E$, $\val\lambda=\frac{r+s-3}{3}$, and $s\geq2r+3$.
\end{itemize}
\end{exam}

\begin{prop} \label{class rank2 D2}
\begin{enumerate}
\item $D^{2}_{\rank N=2}$ represents admissible filtered $(\phi,N)$-modules with $\rank N=2$.
\item The corresponding representations to $D^{2}_{\rank N=2}$ are
     \begin{itemize}
     \item non-split reducible with submodule $E\baseth$ if $s=2r+3$ and
     \item irreducible if $s>2r+3$.
     \end{itemize}
\item $D^{2}_{\rank  N=2}(\lambda,\mflo,\mflt)$ is isomorphic to $D^{2}_{\rank  N=2}(\lambda',\mflo',\mflt')$ if and only if $\lambda=\lambda'$, $\mflo=\mflo'$, and $\mflt=\mflt'$.
\end{enumerate}
\end{prop}

\begin{proof}
The part (1) and (2) are immediate from \ref{rank2first2}. For the part (3), use the same argument as in Proposition \ref{class rank2 D1}.
\end{proof}

The following example arises from \ref{rank2first3}.
\begin{exam}
A filtered $(\phi,N)$-module of Hodge type $(0,r,s)$
$$D^{3}_{\rank  N=2}=D^{3}_{\rank  N=2}(\lambda,\mflo,\mflt,\mflth);$$
\begin{itemize}
\item $\fil{r} D=E(\baseo+\mflo\baset+\mflt\baseth,\baset+\mflth\baseth)$ and $\fil{s} D=E(\baseo+\mflo\baset+\mflt\baseth)$.
\item $
[N]=
\begin{small}\left(
  \begin{array}{ccc}
    0 & 0  & 0 \\
    1 & 0  & 0 \\
    0 & 1  & 0 \\
  \end{array}
\right)
\end{small}
$ and $
[\phi]=
\begin{small}\left(
  \begin{array}{ccc}
    p^{2}\lambda & 0  & 0 \\
    0 & p\lambda & 0 \\
    0 & 0 & \lambda \\
  \end{array}
\right)
\end{small}
$ for $\lambda$ in $E$.
\item $\mflo,\mflt,\mflth\in E$ and $\val\lambda=\frac{r+s-3}{3}$.
\end{itemize}
\end{exam}

\begin{prop} \label{class rank2 D3}
\begin{enumerate}
\item $D^{3}_{\rank N=2}$ represents admissible filtered $(\phi,N)$-modules with $\rank N=2$.
\item The corresponding representations to $D^{3}_{\rank N=2}$ are
     \begin{itemize}
     \item non-split reducible with submodules $E\baseth$ and $E(\baset,\baseth)$ if $s=2$ and
     \item irreducible if $s>2$.
     \end{itemize}
\item $D^{3}_{\rank  N=2}(\lambda,\mflo,\mflt,\mflth)$ is isomorphic to $D^{3}_{\rank  N=2}(\lambda',\mflo',\mflt',\mflth')$ if and only if $\lambda=\lambda'$, $\mflo=\mflo'$, $\mflt=\mflt'$, and $\mflth=\mflth'$.
\end{enumerate}
\end{prop}

\begin{proof}
The part (1) and (2) are immediate from \ref{rank2first3}. For the part (3), use the same argument as in Proposition \ref{class rank2 D1}.
\end{proof}

\begin{rema}
According to $D^{1}_{\rank N=2}$, $D^{2}_{\rank N=2}$, and $D^{3}_{\rank N=2}$, there are no irreducible admissible filtered $(\phi,N)$-modules of Hodge--Tate weights $(0,1,2)$ and with $\rank N=2$. But this is not a surprising result. For the $2$-dimensional case, there are no irreducible, semi-stable, non-crystalline representations of $G_{\QP}$ with Hodge--Tate weights $(0,1)$.
\end{rema}

\begin{prop}
Let $i,j\in\{1,2,3\}$. If $D^{i}_{\rank N=2}$ is isomorphic to $D^{j}_{\rank N=2}$, then $i=j$.
\end{prop}

\begin{proof}
$D_{\rank N=2}^{1}$ is not isomorphic to $D^{2}_{\rank N=2}$, since $D^{1}_{\rank N=2}$ is defined for $s\leq2r-4$ and $D^{2}_{\rank N=2}$ for $s\geq2r+4$. For the other pairs, assume that $T$ is an isomorphism from $D^{i}_{\rank N=2}$ to $D^{j}_{\rank N=2}$. For each $k$, we let $\vphi_{k}$ be the Frobenius map for $D^{k}_{\rank N=2}$. Then obviously $\phi_{i}=\phi_{j}$. By Lemma \ref{rank2typeofP}, we know that $T$ is a scalar multiple of the identity. However, for each pair $(i,j)$ with $i\not=j$, it is easy to check that any scalar multiple of the identity can not preserve the filtration.
\end{proof}

\begin{prop}
Every $3$-dimensional semi-stable representation of $G_{\QP}$ with regular Hodge--Tate weights and with $\rank N=2$ is isomorphic to a representation corresponding to some $D^{i}_{\rank N=2}$ up to twist by a power of the cyclotomic character.
\end{prop}

\begin{proof}
We found all the admissible filtered $(\phi,N)$-modules with $\rank N=2$ in the previous subsection. Since the list of filtered modules in this subsection represents all of the modules in the previous subsection, we are done.
\end{proof}

\bibliographystyle{alpha}

\end{document}